\documentclass[reqno, a4paper]{amsart}

\usepackage[english]{babel}

\usepackage[utf8]{inputenc} 
\usepackage[T1]{fontenc}
\usepackage{lmodern}

\usepackage{mathtools}
\mathtoolsset{showonlyrefs}    
\usepackage{amsfonts,amssymb,amsthm,dsfont,mathrsfs}
\numberwithin{equation}{section}
\usepackage{nicefrac}
\usepackage{bm}
\usepackage{bbm}

\usepackage[shortlabels,inline]{enumitem}
\setlist[enumerate]{leftmargin=1cm,topsep=1mm}
\setlist[itemize]{leftmargin=1cm,topsep=1mm}

\usepackage{tabularx}
\usepackage{caption}
\usepackage{tikz}
\usetikzlibrary{shapes.geometric, arrows, calc, positioning}

\usepackage[pdftex,
	pdftitle={Discrete-to-continuum limits of semilinear stochastic evolution equations in Banach spaces},
	pdfauthor={Y. van Gennip, J. Latz, J. Willems},
	bookmarksopen,
	colorlinks,
	allcolors={[RGB]{0, 50, 240}}
]{hyperref}

\newcommand{\bbP}{\mathbb{P}}

\newcommand{\bbT}{\mathbb{T}}

\newcommand{\R}{\mathbb{R}}
\newcommand{\C}{\mathbb{C}}
\newcommand{\N}{\mathbb{N}}

\newcommand{\cB}{\mathcal{B}}

\newcommand{\cD}{\mathcal{D}}

\newcommand{\cF}{\mathcal{F}}

\newcommand{\cI}{\mathcal{I}}

\newcommand{\cL}{\mathcal{L}}
\newcommand{\cM}{\mathcal{M}}

\newcommand{\fI}{\mathfrak{I}}

\newcommand{\norm}[2]{     \| #1       \|_{ #2 }}
\DeclarePairedDelimiter{\Norm}{\lVert}{\rVert}
\newcommand{\scalar}[2]{     \langle #1       \rangle_{ #2 }}

\DeclarePairedDelimiter{\abs}{\lvert}{\rvert}

\newcommand{\sgn}{\operatorname{sgn}}

\newcommand{\rd}{\mathop{}\!\mathrm{d}}
\newcommand{\from}{\colon}
\newcommand{\clos}[1]{\overline{ #1 }} 
\newcommand{\id}{\operatorname{Id}}
\newcommand{\dom}[1]{\mathsf D(#1)}

\newcommand{\deq}{\coloneqq}

\newcommand{\indic}[1]{\mathbf{1}_{#1}}
\newcommand{\emb}{\hookrightarrow}

\newcommand{\LO}{\mathscr L}

\newtheorem{lemma}{Lemma}[section]
\newtheorem{proposition}[lemma]{Proposition}
\newtheorem{theorem}[lemma]{Theorem}
\newtheorem{corollary}[lemma]{Corollary}

\theoremstyle{remark}

\theoremstyle{definition}
\newtheorem{definition}[lemma]{Definition}
\newtheorem{assumption}[lemma]{Assumption}
\newtheorem{example}[lemma]{Example}
\newtheorem{setting}[lemma]{Setting}

\renewcommand{\Re}{\operatorname{Re}}

\begin{document}
	\author{Yves van Gennip \and Jonas Latz \and Joshua Willems}
	
	\address[Yves van Gennip]{Delft Institute of Applied Mathematics\\
		Delft University of Technology\\
		The Netherlands.}
	
	\email{y.vangennip@tudelft.nl}
	
	\address[Jonas Latz]{Department of Mathematics\\
		University of Manchester\\
		United Kingdom.}
	
	\email{jonas.latz@manchester.ac.uk}

	\address[Joshua Willems]{Delft Institute of Applied Mathematics\\
		Delft University of Technology\\
		The Netherlands.}
	
	\email{j.willems@tudelft.nl}
	
	\title[Discrete-to-continuum limits of SPDEs]
	{Discrete-to-continuum limits of semilinear stochastic evolution equations in Banach spaces}
	
	\keywords{%
		Discrete-to-continuum convergence,
		stochastic evolution equations,
		generalized Whittle--Mat\'ern operators,
		geometric graphs%
	}
	
	\subjclass[2020]{Primary: 
		60H15, 
		35K58, 
		65M12; 
		secondary: 
		35R11, 
		35R01, 
		35R02. 
	}
	
	\begin{abstract}
		We study the convergence of semilinear parabolic stochastic evolution equations, posed on a sequence of Banach spaces approximating a limiting space and driven by additive white noise projected onto the former spaces.
		Under appropriate uniformity and convergence conditions on the linear operators, nonlinear drifts and initial data, we establish convergence of the associated mild solution processes when lifted to a common state space.
		Our framework is applied to the case where the limiting problem is a stochastic partial differential equation whose linear part is a generalized Whittle–Matérn operator on a manifold $\mathcal{M}$, discretized by a sequence of graphs constructed from a (random) point cloud. 
		In this setting we obtain various discrete-to-continuum convergence results for solutions lifted to $L^q(\mathcal{M})$ for $q \in [2,\infty]$, one of which recovers the $L^\infty$-convergence of a finite-difference discretization of certain (fractional) stochastic Allen--Cahn equations.
	\end{abstract}
	
	\maketitle
	
	\section{Introduction}
	\label{sec:intro}
	
	\subsection{Background and motivation} 
	\label{sec:intro:background-motivation}
	
	We establish discrete-to-continuum limits of stochastic evolution equations of the form \eqref{eq:intro-SPDE}, i.e., semilinear parabolic stochastic partial diffential equations (SPDEs) driven by Gaussian white noise. 
	Such SPDEs of evolution play an important role in the modeling of physical and other systems, such as fluid dynamics \cite{BENSOUSSAN1973195,VandenEijn,Frisch_Lesieur_Brissaud_1974,Mattingly}, quantum optics \cite{carmichael2009open}, phase separation \cite{DAPRATO1996241},  diffusion in random media \cite{Havlin01011987,konig2016parabolic}, and population dynamics \cite{Xiong}. Given their significance, there has been a considerable interest in the analysis and numerical analysis of SPDEs; see the introductory textbooks \cite{LiuRockner2015} and \cite{lord2014introduction}, respectively.
	
	We consider the convergence of a sequence of abstract continuous-time equations, each posed on a different Banach space in order to model the approximation of an evolution SPDE by equations that are continuous in time and discrete in space. 
	This framework covers a typical setting where the spatial domains are finite graphs and the limiting differential operator in space is approximated by the corresponding graphical variants. 
	If the finite graphs approximate an underlying manifold, then the graphical differential operators are related to a finite-difference approximation of the differential operator. 
    As an example, we study a class of semilinear SPDEs whose linear part is given by a fractional-order elliptic differential operator on a manifold (known as a \emph{Whittle--Mat\'ern operator}), discretized by a sequence of graphs constructed from a (possibly randomly sampled) cloud of points.
    \emph{Linear} equations of this type have previously been studied in the context of statistics and machine learning~\cite{pmlr-v151-nikitin22a, SanzAlonsoYang2022, SanzAlonsoYang2022JMLR}.
    By verifying in detail that the hypotheses of our abstract framework are satisfied in this situation, we establish the discrete-to-continuum convergence of this scheme.
    Although the main advantage of our general results is their applicability to highly unstructured discretizations, we additionally show that they recover the $L^\infty$-convergence of finite-difference discretizations of the fractional stochastic Allen--Cahn equation on the one-dimensional flat torus.
	To further illustrate the significance of results of this form, we now discuss a few examples of semidiscrete SPDEs as well as their continuum limits.
	
	Semidiscrete models appear frequently in the numerical analysis of (S)PDEs of evolution; since the challenges of spatial, temporal and spatiotemporal discretization are different, these settings are often analyzed separately. This leads to thorough studies of (S)PDEs that are discretized in space but not in time. In the context of SPDEs of evolution, we refer to \cite{Brehier,Gyongy2014,Kruse2014} for examples. 
	
	Stochastic PDEs on graphs also appear naturally as models in the physical sciences, e.g., for interacting particle systems \cite{DaCosta2019} or the representation of disordered media \cite{Havlin01011987}. In the former case, the continuum limit represents the large particle limit in the interacting particle system. 
	
	In the data science literature, (S)PDEs on graphs have recently gained popularity as semi-supervised learning techniques. In a semi-supervised learning problem we are given a set of labeled features as well as a set of unlabeled features, and the goal is to use the former features to recover the labels of the latter. Features are, for example, images, text, or voice recordings; corresponding labels may be descriptors of the content of the images, the author of the text, or a transcript of the voice recording, respectively. Given an appropriate similarity measure on the space of features, an edge-weighted graph can be constructed in which nodes representing similar features are connected by highly weighted edges. The unknown labels can then be estimated by space-discretized PDEs on this graph, as in \cite{bertozzi,budd,vanGennipBudd25a,LapRef}. The PDEs often describe gradient flows that minimize a variational functional. 
	Stochastic PDEs appear in this setting if, in addition to finding an estimate for the labels, the uncertainty in the labels is to be quantified as well~\cite{BertozziLuoStuart,SanzAlonsoYang2022}.
	The SPDEs of evolution here either form the basis of Markov chain Monte Carlo sampling algorithms \cite{Cotter,HSV2007, HSVW2005, Roberts} or of a randomized global optimization scheme \cite{Chak,Chiang} for the solution of the variational problem in the deterministic setting.
	In this semi-supervised learning setting, discrete-to-continuum limits are of interest because they establish the consistency of the models in the large-data limit. For deterministic PDEs, the literature has grown to encompass pointwise limits of operators, as in \cite{HeinAudibertvonLuxburg07}, $\Gamma$-limits of the functionals that underlie the dynamics~\cite{GTS2016,LauxLelmi21,ThorpeTheil19,vanGennipBertozzi12,WeihsThorpe23}, and more recently, discrete-to-continuum limits for the dynamics themselves \cite{ElBouchairiFadiliHafieneElmoataz23, FadiliForcadelNguyenZantout23, GigavanGennipOkamoto22,HraivoronskaTse23,LauxLelmi23,MercervanGennip25,WeihsFadiliThorpe23}. For a more in-depth overview of the literature of discrete-to-continuum limits, we refer to \cite{vanGennipBudd25b}.
	\subsection{Main results}
	\label{sec:intro:main-results}
	
	We will now summarize the abstract setting and main discrete-to-continuum convergence results from Sections~\ref{sec:OU-H}--\ref{sec:reaction-diffusion},
    which will already be applied in Section~\ref{sec:graph-WM} to the Whittle--Mat\'ern and stochastic fractional Allen--Cahn equations described in Section~\ref{sec:intro:background-motivation}.
	
	Given a probability space $(\Omega, \cF, \bbP)$ and some $T \in (0,\infty)$ (called the \emph{time horizon}), we consider a sequence of semilinear parabolic stochastic evolution equations
	\begin{equation}\label{eq:intro-SPDE}
		\left\lbrace
		\begin{aligned}
			\rd X_n(t) &= -A_n X_n(t)\rd t + F_n(t, X_n(t)) \rd t + \rd W_n(t),
			\quad 
			t 
			\in 
			(0,T], 
			\\
			X_n(0) &= \xi_n.
		\end{aligned}
		\right.
	\end{equation}
	indexed by $n \in \clos\N \deq \{1,2,\dots\} \cup \{\infty\}$. 
	This problem will be rigorously formulated as a stochastic differential equation taking values in a real Banach space $E_n$ (or a smaller embedded space $B_n \hookrightarrow E_n$) called the \emph{state space}.
	In general, we assume that the terms appearing in~\eqref{eq:intro-SPDE} are as follows:
	\begin{itemize}
		\item $-A_n$ is the generator of a bounded analytic semigroup of bounded linear operators on $E_n$ or $B_n$;
		\item $u_n \mapsto F_n(\omega, t, u_n)$ is a possibly random and nonlinear drift operator on $E_n$ or $B_n$ for all $(\omega, t) \in \Omega \times [0,T]$;
		\item $(W_n(t))_{t\ge0}$ is the projection, onto an appropriate subspace, of an $H$-valued cylindrical Wiener process $(W(t))_{t\ge0}$ for a real and separable Hilbert space $H$ (more details are specified in Theorem~\ref{thm:d2c-conv-informal} below);
		\item $\xi_n$ is a possibly random initial datum with values in $E_n$ or $B_n$.
	\end{itemize}
	The precise assumptions (in particular, whether the operators and initial data are $E_n$-valued or $B_n$-valued) vary throughout Sections~\ref{sec:graph-WM}--\ref{sec:reaction-diffusion};
    an overview is provided in Table~\ref{tab:SPDEs-per-section}.
	Depending on the setting, the \emph{mild solutions} to~\eqref{eq:intro-SPDE} are either well defined on the whole of $[0,T]$ almost surely, or cease to exist at a time $t < T$ with nonzero probability; such solutions are said to be \emph{global} or \emph{local}, respectively.
	
	The aim of this work is to establish conditions on the data of~\eqref{eq:intro-SPDE} under which the corresponding solutions $(X_n)_{n\in\N}$ converge to $X_\infty$ as $n \to \infty$.
	In order to compare processes which take their values in different Banach spaces, we need to assume that each of the families $(E_n)_{n\in\clos\N}$, $(H_n)_{n\in\clos\N}$ and $(B_n)_{n\in\clos\N}$ embeds uniformly into a common space---namely into $E_\infty$, $H_\infty$ and a closed subspace ${\widetilde B \subseteq B_\infty}$, respectively---which they approximate in some appropriate sense as $n \to \infty$.
	In particular, we shall assume that they share a common sequence of (linear) \emph{lifting operators} $(\Lambda_n)_{n\in\clos\N}$ such that each $\Lambda_n$ maps $E_n$ (resp.\ $H_n$, $B_n$) boundedly into $E_\infty$ (resp.\ $H_\infty$, $\widetilde B$), as well as a sequence of \emph{projection operators} $(\Pi_n)_{n\in\clos\N}$ which are left-inverses to the respective lifting operators.
	That is, each sequence satisfies Assumption~\ref{ass:disc3} below with the same lifting and projection operators.
	As an example (see Section~\ref{sec:graph-WM}), one can think of $E_\infty \deq L^q(\cD)$ for $q \in [2,\infty)$, $H_\infty \deq L^2(\cD)$, $\widetilde B \deq L^\infty(\cD)$ (Lebesgue spaces) and $B_\infty \deq C(\cD)$ (continuous functions) for some spatial domain $\cD$, along with $E_n \deq L^q(\cD_n)$, $H_n \deq L^2(\cD_n)$ and $B_n \deq L^\infty(\cD_n)$ for some approximations $(\cD_n)_{n\in\N}$ of the domain $\cD$. 
	
	The projection and lifting operators allow us to compare the ($E_n$- or $B_n$-valued) solution processes $X_n$ by instead considering convergence of the lifted processes $\widetilde X_n \deq \Lambda_n X_n$ to $X_\infty$ as $n\to \infty$, which we call~\emph{discrete-to-continuum convergence}.
	Moreover, they allow us to formulate assumptions under which this occurs in terms of conditions imposed on the lifted resolvents $\widetilde R_n \deq \Lambda_n (A_n + \id_n)^{-1} \Pi_n$ of the linear operators $A_n$, the lifted drift operators $\widetilde F_n(\omega, t, u) \deq \Lambda_n F_n(t,\omega, \Pi_n u)$, and the lifted initial data $\widetilde \xi_n \deq \Lambda_n \xi_n$.
	Roughly speaking, we assume that 
	\begin{itemize}
		\item $\widetilde F_n \to F_\infty$ `pointwise' (see~\ref{item:F-approx} in Section~\ref{sec:globalLip-linGrowth} or~\ref{item:F-approx-B} in Section~\ref{sec:d2c-Nemytskii});
		\item $\widetilde R_n \to R_\infty$ `pointwise' and there exists a small enough $\beta \in [0,\frac{1}{2})$ such that the fractional powers $\widetilde R_n^\beta$ converge to $ R_\infty^\beta$ in an appropriate operator norm (see~\ref{ass:A-conv} in Section~\ref{sec:OU-H} or~\ref{ass:A-conv-B} in Section~\ref{sec:d2c-Nemytskii})
		\item $\widetilde \xi_n \to \xi_\infty$ in $L^p(\Omega; E_\infty)$ or $L^p(\Omega; \widetilde B)$ for some $p \in [1,\infty)$ (see~\ref{ass:IC} in Section~\ref{sec:globalLip-linGrowth} or~\ref{ass:IC-B} in Section~\ref{sec:d2c-Nemytskii}).
	\end{itemize}
	Again we refer to Table~\ref{tab:SPDEs-per-section} for an overview of the different settings and types of solutions, with references to the precise formulations of the corresponding assumptions;
    the setting in the first row (i.e., Section~\ref{sec:conv-graphdiscr-SPDEs}) covers the (fractional) stochastic Allen--Cahn equations announced in Section~\ref{sec:intro:background-motivation}, see Example~\ref{ex:Allen-Cahn} below.
	The following theorem is a summary of the discrete-to-continuum approximation results for solutions to the abstract equations~\eqref{eq:intro-SPDE} in these respective settings.
	
	\begin{theorem}[Discrete-to-continuum convergence---summarized]\label{thm:d2c-conv-informal}
		Let $(\Omega, \cF, \bbP)$ be a probability space and $T \in (0,\infty)$ be a terminal time.
		Consider equations~\eqref{eq:intro-SPDE}, where the state spaces, linear operators, drift operators and initial data are as in one of the rows of Table~\ref{tab:SPDEs-per-section}.
		Let $p \in [1,\infty)$ be the stochastic integrability of the initial data and let $W_n \deq \Pi_n W$, where $(W(t))_{t\ge0}$ is an $H$-valued cylindrical Wiener process.
		For all $n \in \clos\N$, there exists a unique (local or global, see Table~\ref{tab:SPDEs-per-section}) mild solution $X_n$ to~\eqref{eq:intro-SPDE}, and the lifted solution processes $\widetilde X_n \deq \Lambda_n X_n$ satisfy the following:
		\begin{enumerate}
			\item If the solutions are global and $p > 1$, then for all $p^- \in [1, p)$ we have
			\begin{equation}
				\widetilde X_n \to X_\infty
				\quad 
				\text{as } n \to \infty
			\end{equation}
			in $L^{p^-}\!(\Omega; C([0,T]; E_\infty))$ (resp.\ in $L^{p^-}\!(\Omega; C([0,T]; \widetilde B))$).
			In the (semi)linear settings with globally Lipschitz drifts of linear growth, the same in fact holds with $p^- \deq p$ for any $p \in [1,\infty)$.
			\item If the solutions are local, with associated explosion times $\sigma_n \from \Omega \to (0,T]$ (precisely defined in~\eqref{eq:explosion-time} below), then we have
			\begin{equation}
				\widetilde X_n \indic{[0, \sigma_{\infty} \wedge \sigma_n)} \to X_\infty \indic{[0, \sigma_{\infty})}
				\quad 
				\text{as } n \to \infty
			\end{equation}
			in $L^0(\Omega \times [0,T]; E_\infty)$ (resp.\ in $L^0(\Omega \times [0,T]; \widetilde B)$), where $L^0$ indicates convergence in measure.
		\end{enumerate}
	\end{theorem}
	The full convergence statement for each setting is given in the corresponding part of Sections~\ref{sec:OU-H}--\ref{sec:reaction-diffusion}.
	To be precise, Theorem~\ref{thm:d2c-conv-informal} is comprised of the following results, in order of appearance:
	Theorem~\ref{thm:d2c-conv-example-SPDEs},
	Proposition~\ref{prop:d2c-pathwise-conv-Wdelta},
	Theorems~\ref{thm:d2c-semilinear-global-linear} and~\ref{thm:d2c-conv-locLip-unifBound},
	Proposition~\ref{prop:d2c-pathwise-conv-Wdelta2-B},
	Theorems~\ref{thm:d2c-semilinear-global-linear-B} and~\ref{thm:d2c-conv-locLip-unifBound-B} and~Corollary~\ref{cor:d2c-conv-dissB}.

	\begin{table}
		\centering
		\footnotesize
		\begin{tabularx}{\linewidth}{
				>{\hsize=.05\hsize}X
				>{\hsize=.27\hsize}X
				>{\hsize=.5\hsize}X
				>{\hsize=.18\hsize}c}
			\hline
			\textbf{Sec.}&\textbf{Description}&\textbf{Assumptions}& \textbf{Sol.\ type}\\
			\hline 
			\rule{0pt}{3ex}
			\S{\ref{sec:conv-graphdiscr-SPDEs}} & graph-based approximation of Whittle--Mat\'ern operators on a manifold & 
			\begin{minipage}[t]{\linewidth}
				\begin{itemize}[itemsep=3pt, leftmargin=*]
					\item $A_n \deq [\cL_n^{\tau,\kappa}]^s$ (Whittle--Mat\'ern operators)
					\item $[F_n(t, u)](x) \deq f_n(t, u(x))$ (Nemytskii drift)
					\item Assumption~\ref{ass:nonlinearity} (on the functions $(f_n)_{n\in\clos\N}$)
				\end{itemize}
			\end{minipage} 
			& global 
			\\%
			\rule{0pt}{4ex}
			\S{\ref{sec:OU-H}} 
			& 
			$E_n$-valued linear 
			& 
			\begin{minipage}[t]{\linewidth}
				\begin{itemize}[itemsep=3pt, leftmargin=*]
					\item \ref{ass:disc}--\ref{ass:A-conv} (linear operators)
					\item $F_n \deq 0$ and $\xi_n \deq 0$
				\end{itemize}
			\end{minipage} 
			& 
			global 
			\\
			\rule{0pt}{4ex}
			\S{\ref{sec:globalLip-linGrowth}}
			& 
			$E_n$-valued semilinear; \newline globally Lipschitz drifts of linear growth 
			& 
			\begin{minipage}[t]{\linewidth}
				\begin{itemize}[itemsep=3pt, leftmargin=*]
					\item \ref{ass:disc}--\ref{ass:A-conv} (linear operators)
					\item \ref{item:F-globalLip-linearGrowth}--\ref{item:F-approx} (drift operators)
					\item \ref{ass:IC} (initial data)
				\end{itemize}
			\end{minipage} 
			& 
			global 
			\\
			\rule{0pt}{4ex}
			\S{\ref{sec:locLip}} 
			& 
			$E_n$-valued semilinear; 
			\newline 
			locally Lipschitz and locally bounded drifts 
			& 
			\begin{minipage}[t]{\linewidth}
				\begin{itemize}[itemsep=3pt, leftmargin=*]
					\item \ref{ass:disc}--\ref{ass:A-conv} (linear operators)
					\item \ref{item:F-locLip-unifBdd} and~\ref{item:F-approx} (drift operators)
					\item \ref{ass:IC} (initial data)
				\end{itemize}
			\end{minipage}  & local 
			\\
			\rule{0pt}{4ex}
			\S{\ref{sec:d2c-Nemytskii}} & $B_n$-valued semilinear; \newline globally Lipschitz drifts of linear growth & \begin{minipage}[t]{\linewidth}
				\begin{itemize}[itemsep=3pt, leftmargin=*]
					\item \ref{ass:disc2}--\ref{ass:emb-Htheta-B-H2} with $\theta + 2\beta < 1$ (lin.\ ops.)
					\item \ref{item:F-globalLip-linearGrowth-B}--\ref{item:F-approx-B} (drift operators)
					\item \ref{ass:IC-B} (initial data)
				\end{itemize}
			\end{minipage} & global 
			\\
			\rule{0pt}{4ex}
			\S{\ref{sec:locLip-B}} & $B_n$-valued semilinear; \newline locally Lipschitz and locally bounded drifts & \begin{minipage}[t]{\linewidth}
				\begin{itemize}[itemsep=3pt, leftmargin=*]
					\item \ref{ass:disc2}--\ref{ass:emb-Htheta-B-H2} with $\theta + 2\beta < 1$ (lin.\ ops.)
					\item \ref{item:F-locLip-unifBdd-B} and~\ref{item:F-approx-B} (drift operators)
					\item \ref{ass:IC-B} (initial data)
				\end{itemize}
			\end{minipage} & local 
			\\
			\rule{0pt}{4ex}
			\S{\ref{sec:global-wellposed-RDeq}} & $B_n$-valued semilinear; \newline dissipative drifts 
			& 
			\begin{minipage}[t]{\linewidth}
				\begin{itemize}[itemsep=3pt, leftmargin=*]
					\item \ref{ass:disc2}--\ref{ass:emb-Htheta-B-H2} with $\theta + 2\beta < 1$ (lin.\ ops.)
					\item \ref{item:F-diss-B} and~\ref{item:F-approx-B} (drift operators)
					\item \ref{ass:IC-B} (initial data)
				\end{itemize}
			\end{minipage} 
			& global
		\end{tabularx}
		\caption{\label{tab:SPDEs-per-section} Overview of the types of SPDEs considered in the different (sub)sections comprising this work.
			The row in which an assumption appears for the first time also indicates the (sub)section where its definition can be found.}
	\end{table}
	
	\subsection{Contributions}
	\label{sec:intro:contributions}
	
	The abstract discrete-to-continuum approximation theorems for stochastic semilinear parabolic evolution equations driven by additive cylindrical Wiener noise---summarized in Theorem~\ref{thm:d2c-conv-informal} and proved in Sections~\ref{sec:OU-H}--\ref{sec:reaction-diffusion}---complement the results from~\cite{KvN2011, KvN2012}, which establish the \emph{continuous dependence on the coefficients} of semilinear equations driven by multiplicative noise in a state space with unconditional martingale differences (UMD), i.e., convergence in the case where $E_n = E$ (and $B_n = B$) for all $n \in \clos\N$.
	Given the motivating applications and our aim to provide a self-contained exposition of the proofs, we make the simplifying assumptions that the UMD spaces $(E_n)_{n\in\clos\N}$ have Rademacher type 2 (which was also assumed for $E$ in~\cite{KvN2012} but not in~\cite{KvN2011}) and the noise is additive.
	
	Under these conditions, we provide a direct proof of convergence of the $E_n$-valued stochastic convolutions solving the linear parts of~\eqref{eq:intro-SPDE} 
	using a Da Prato--Kwapie\'n--Zabczyk factorization argument~\cite{DaPratoKwapienZabczyk1987} and a discrete-to-continuum analog of the Trotter--Kato approximation theorem~\cite[Theorem~2.1]{ItoKappel1998}, see Proposition~\ref{prop:d2c-pathwise-conv-Wdelta}.
	We extend it to the semilinear $E_n$-valued settings described in Table~\ref{tab:SPDEs-per-section} by adapting the arguments from~\cite[Sections~3 and~4]{KvN2011} and \cite[Subsection~3.1]{KvN2012} to incorporate the discrete-to-continuum projection and lifting operators, yielding Theorems~\ref{thm:d2c-semilinear-global-linear} and~\ref{thm:d2c-conv-locLip-unifBound}, respectively.
	In order to state and prove the analogous Theorems~\ref{thm:d2c-semilinear-global-linear-B} and~\ref{thm:d2c-conv-locLip-unifBound-B} for the $B_n$-valued settings, we impose a uniform ultracontractivity condition on the semigroups which replaces the restriction in~\cite[Section~3]{KvN2012} that the fractional domain spaces $\dot E_n^\alpha \deq \dom{(\id_n + A_n)^{\alpha/2}}$ also coincide for all $n \in \clos\N$.
	
	Theorem~\ref{thm:d2c-conv-example-SPDEs}, regarding the graph discretization of equations whose linear operators are of generalized Whittle--Mat\'ern type on a manifold $\cM$, extends analogous convergence results for linear equations on a spatial domain (cf.~\cite[Theorem~4.2]{SanzAlonsoYang2022} and~\cite[Theorem~7]{SanzAlonsoYang2022} in $L^2(\cM)$ and $L^\infty(\cM)$, respectively) to spatiotemporal and semilinear equations.
	Like the cited theorems, its proof relies on recent spectral convergence results for graph Laplacians (see~\cite{CGT2022} and~\cite{CalderGTLewicka2022} for convergence of eigenfunctions in $L^2$ and $L^\infty$, respectively), which we use to verify that these SPDEs fit into the abstract framework from Sections~\ref{sec:OU-H}--\ref{sec:reaction-diffusion}.
	
	\subsection{Outline}
	\label{sec:intro:outline}
	
	The remainder of this article is structured as follows.
	In Section~\ref{sec:prelims}, we establish some notational conventions and collect preliminaries regarding the (deterministic) discrete-to-continuum Trotter--Kato approximation theorem and stochastic integration in UMD-type-2 Banach spaces.
	We demonstrate in Section~\ref{sec:graph-WM} how the results summarized by Theorem~\ref{thm:d2c-conv-informal} can be applied to graph discretizations of stochastic parabolic evolution equations whose linear part is a generalized Whittle--Mat\'ern operator on a manifold.
	In Section~\ref{sec:OU-H}, we consider the \emph{linear} $E_n$-valued version of~\eqref{eq:intro-SPDE}, whose solutions are also known as infinite-dimensional Ornstein--Uhlenbeck processes.
	These results are extended in Section~\ref{sec:semilinear-L2} to allow for \emph{semilinear} $E_n$-valued drift operators under (local or global) Lipschitz continuity and boundedness assumptions.
	In Section~\ref{sec:reaction-diffusion} we first treat the analogous results in the semilinear \emph{$B_n$-valued} setting, and then establish global well-posedness and convergence for dissipative drifts. 
	Finally, in Section~\ref{sec:outlook} we discuss some potential directions for further research.
	This work is complemented by three appendices: Appendix~\ref{app:proofs-of-intermediate-results} consists of postponed proofs of some intermediate results from Section~\ref{sec:graph-WM}. Appendices~\ref{app:frac-par-int} and~\ref{app:uniform-sectoriality} are concerned with fractional parabolic integration and (uniformly) sectorial linear operators, respectively.
	See Figure~\ref{fig:connections-sections-mainmatter} for a schematic overview of the relations between {Sections~\ref{sec:graph-WM}--\ref{sec:reaction-diffusion}} and the appendices.

	\begin{figure}
		\centering
		\small
		\begin{tikzpicture}[node distance=1cm, every node/.style={draw, text width=2.2cm, align=center, minimum height=.8cm}]
			\node (sec4) {Section~\ref{sec:OU-H}};
			
			\node (appendix-b) [below=of sec4] {Appendix~\ref{app:frac-par-int}};
			
			\node (appendix-c) [above right=-0.2cm and 1.2cm of sec4] {Appendix~\ref{app:uniform-sectoriality}};
			
			\node (subsec51) [below left=0.5cm and 1.5cm of sec4] {Subsection~\ref{sec:globalLip-linGrowth}};
			\node (subsec52) [below=1.8cm of subsec51] {Subsection~\ref{sec:locLip}};
			
			\node (right1) [below right=0.5cm and 1.5cm of sec4] {Subsection~\ref{sec:d2c-Nemytskii}};
			\node (right2) [below=1cm of right1] {Subsection~\ref{sec:locLip-B}};
			\node (right3) [below=1cm of right2] {Subsection~\ref{sec:global-wellposed-RDeq}};
			
			\node (bottom) [below=1.2cm of {$(subsec52)!0.5!(right3)$}] {Section~\ref{sec:graph-WM}};
			
			\node (appendix-a) [above=1.5cm of bottom] {Appendix~\ref{app:proofs-of-intermediate-results}};
			
			\draw[thick, -{latex}] (sec4) -- (subsec51);
			\draw[thick, -{latex}] (sec4) -- (right1);
			
			\draw[thick, -{latex}] (appendix-a) -- (bottom);
			
			\draw[thick, -{latex}] (appendix-b) -- (sec4);
			\draw[thick, -{latex}] (appendix-b) -- (subsec51); 
			\draw[thick, -{latex}] (appendix-b) -- (right1);  
			
			\draw[thick, -{latex}] (appendix-c) -- (sec4);  
			
			\draw[thick, -{latex}] (subsec51) -- (subsec52);
			\draw[thick, -{latex}] (subsec51) -- (bottom);
			
			\draw[thick, -{latex}] (right1) -- (right2);
			\draw[thick, -{latex}] (right2) -- (right3);
			\draw[thick, -{latex}] (right3) -- (bottom);
		\end{tikzpicture}
		\caption{\label{fig:connections-sections-mainmatter} Relations between the sections comprising the main part of this article and the appendix.
			Arrows indicate when the results of one section are applied in another.}
	\end{figure}
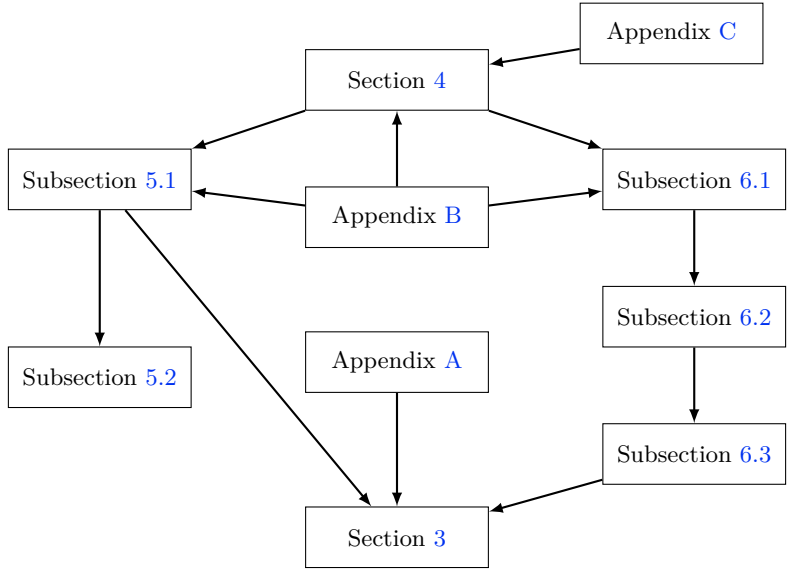
	
	\section{Preliminaries}
	\label{sec:prelims}
	
	\subsection{Notation}
	\label{sec:prelims:notation}
	\begin{table}
		\centering
		\footnotesize
		\begin{minipage}[t]{.48\textwidth}
			\vspace*{0pt}
			\begin{tabularx}{\linewidth}{>{\hsize=.25\hsize}X>{\hsize=.75\hsize}X}
				\multicolumn{2}{l}{\textbf{Elementary sets and operations}} \vspace{.1em} \\
				$\N$ & positive integers \\
				$\N_0$ & nonnegative integers \\
				$\clos\N$ & $\N \cup \{\infty\}$ \\
				$\id_D$ & identity map on a set $D$ \\
				$\mathbf 1_{D_0}$ & indicator map on $D_0 \subseteq D$ \\
				$s\wedge t$ & minimum of $s,t\in\R$ \\
			\end{tabularx}
			\\[.3\baselineskip]
			\begin{tabularx}{\linewidth}{>{\hsize=.25\hsize}X>{\hsize=.75\hsize}X}
				\multicolumn{2}{l}{\textbf{Bounded linear operators}} \vspace{.1em} \\
				$H$, $K$ & \raggedright\arraybackslash separable Hilbert spaces \\
				$E$, $F$ & \raggedright\arraybackslash arbitrary Banach spaces \\
				$\scalar{\,\cdot\,,\,\cdot\,}{H}$ & inner product of $H$ \\
				$\norm{\,\cdot\,}{E}$ & norm of $E$ \\
				$E^*$ & dual space of $E$ \\
				$\LO(E; F)$ & \raggedright\arraybackslash bounded linear operators (from $E$ to $F$) \\
				$\LO(E)$ & abbreviation of $\LO(E; E)$ \\
				$\gamma(H; E)$ & \raggedright\arraybackslash  $\gamma$-radonifying operators \\
				$\LO_2(H; K)$ & \raggedright\arraybackslash  Hilbert--Schmidt operators 
				\\
			\end{tabularx}
			\\[.3\baselineskip]
			\begin{tabularx}{\linewidth}{>{\hsize=.25\hsize}X>{\hsize=.75\hsize}X}
				\multicolumn{2}{l}{\textbf{Function spaces}} \vspace{.1em} \\
				$C(K; E)$ & \raggedright\arraybackslash continuous functions from a compact space $K$ to $E$ \\
				$C(K)$ & \raggedright\arraybackslash abbreviation of $C(K; \R)$ \\
				$L^p(S; E)$ & \raggedright\arraybackslash Bochner space of $p$-integrable functions from a measure space $(S, \mathscr A\!, \nu)$ to $E$ \\
				$L^p(S)$ & Lebesgue space $L^p(S; \R)$ \\
			\end{tabularx}
			\\[.3\baselineskip]
			\begin{tabularx}{\linewidth}{>{\hsize=.25\hsize}X>{\hsize=.75\hsize}X}
				\multicolumn{2}{l}{\textbf{Graph discretization}} \vspace{.1em} \\
				$\cM$ & manifold from Assumption~\ref{ass:manifold} \\
				$\bbT^m$ & $m$-dimensional flat torus \\
				$d_{\cM}$ & geodesic metric on $\cM$ \\
				$\mu$ & volume measure on $\cM$ \\
				$\cM_n$ & point cloud $(x_n^{(j)})_{j=1}^n \subset \cM$ \\
				$(\widetilde \Omega, \widetilde{\cF}, \widetilde{\bbP})$ & \raggedright\arraybackslash probability space of random point cloud from Example~\ref{ex:sampling-manifold} \\
				$\mu_n$ & empirical measure on $\cM_n$ \\
				$T_n$ & transport map from $\cM$ to $\cM_n$ \\
				$\varepsilon_n$ & $\sup_{x\in \cM} d_{\cM}(x, T_n(x))$ \\
				$h_n$ & \raggedright\arraybackslash graph connectivity length scale, see~\eqref{eq:graph-weights} \\
				$\cL_n^{\tau,\kappa}$ & \raggedright\arraybackslash (discretized) Whittle--Mat\'ern operator with coefficient functions $\tau,\kappa \from \cM \to [0,\infty)$ \\
				$(\lambda_n^{(j)})_{j=1}^n$ & eigenvalues of $\cL_n^{\tau,\kappa}$ \\
				$(\psi_n^{(j)})_{j=1}^n$ & \raggedright\arraybackslash $L^2(\cM_n)$-normalized eigenfunctions of $\cL_n^{\tau,\kappa}$ \\
				$M_{\psi, \infty}$ & \raggedright\arraybackslash uniform $L^\infty$-bound of the eigenfunctions, see Assumption~\ref{ass:uniform-Linfty-bdd-eigenfuncs} \\
				$M_{S, q}$ & uniform-ultracontractivity constant, see~\eqref{eq:ultracontractivity}
			\end{tabularx}
		\end{minipage}%
		\begin{minipage}[t]{.01\textwidth}\end{minipage}
		\begin{minipage}[t]{.49\textwidth}
			\vspace*{0pt}
			\begin{tabularx}{\linewidth}{>{\hsize=.3\hsize}X>{\hsize=.7\hsize}X}
				\multicolumn{2}{l}{\textbf{`Discrete-to-continuum' spaces}} \vspace{.1em} \\
				$(E_n)_{n\in\clos\N}$, $\widetilde E$ & \raggedright\arraybackslash Banach spaces from Assumption~\ref{ass:disc3} or~\ref{ass:disc} \\
				$(H_n)_{n\in\clos\N}$, $\widetilde H$ & Hilbert spaces from~\ref{ass:disc} \\
				$(B_n)_{n\in\clos\N}$, $\widetilde B$ & Banach spaces from~\ref{ass:disc2} \\
			\end{tabularx}
			\\[.3\baselineskip]
			\begin{tabularx}{\linewidth}{>{\hsize=.07\hsize}X>{\hsize=.93\hsize}X}
				\multicolumn{2}{l}{\textbf{Projection and lifting}} \vspace{.1em} \\
				$\Pi_n$ & \raggedright\arraybackslash projection operator from $E_n$ (resp.\ $H_n$ or $B_n$) to $\widetilde E$ (resp.\ $\widetilde H$ or $\widetilde B$) \\
				$\Lambda_n$ & \raggedright\arraybackslash lifting operator from $\widetilde E$ (resp.\ $\widetilde H$ or $\widetilde B$) to $E_n$ (resp.\ $H_n$ or $B_n$) \\
				$\widetilde T_n$ & lifted version $\Lambda_n T_n \Pi_n$ on $\widetilde E$ (resp.\ $\widetilde H$ or $\widetilde B$) of operator $T_n$ on $E_n$ (resp.\ $H_n$ or~$B_n$) \\
				$\widetilde Y_n$ & lifted version $\Lambda_n Y_n$ on $\widetilde E$ (resp.\ $\widetilde H$ or $\widetilde B$) of process $Y_n$ on $E_n$ (resp.\ $H_n$ or $B_n$) \\
				$M_\Pi$ & $\sup_{n\in\N} 
				\norm{\Pi_n}{\LO(\widetilde E; E_n)}$ \\
				$\widetilde M_\Pi$ & $\sup_{n\in\N} 
				\norm{\Pi_n}{\LO(\widetilde B; B_n)}$ \\
				$M_\Lambda$ & $\sup_{n\in\N} 
				\norm{\Lambda_n}{\LO(E_n; \widetilde E)}$ \\
				$\widetilde M_\Lambda$ & $\sup_{n\in\N} 
				\norm{\Lambda_n}{\LO(B_n; \widetilde B)}$ \\
			\end{tabularx}
			\\[.3\baselineskip]
			\begin{tabularx}{\linewidth}{>{\hsize=.2\hsize}X>{\hsize=.8\hsize}X}
				\multicolumn{2}{l}{\textbf{Linear operators in evolution equations}} \vspace{.1em} \\
				$A_n$ & \raggedright\arraybackslash linear operator on $E_n$ with domain $\dom{A_n}$ \\
				$S_n$ & semigroup generated by $-A_n$ \\
				\raggedright\arraybackslash $M_S$ & \raggedright\arraybackslash uniform-boundedness constant of $(S_n)_{n\in\clos\N}$ in $(E_n)_{n\in\clos\N}$, see~\eqref{eq:semigroup-estimates} \\
				\raggedright\arraybackslash $\widetilde M_S$ & \raggedright\arraybackslash uniform-boundedness constant of $(S_n)_{n\in\clos\N}$ in $(B_n)_{n\in\clos\N}$, see~\eqref{eq:unif-bdd-Sn-Bn} \\
				$\rho(A_n)$ & resolvent set of $A_n$ \\
				$R_n^\beta$ & $(A_n + \id_n)^{-\beta}$ \\
				$\mathfrak I_{A_n}^s$ & fractional parabolic integration operator, see Appendix~\ref{app:frac-par-int} \\
			\end{tabularx}
			\\[.3\baselineskip]
			\begin{tabularx}{\linewidth}{>{\hsize=.2\hsize}X>{\hsize=.8\hsize}X}
				\multicolumn{2}{l}{\textbf{Stochastic evolution equations}} \vspace{.1em} \\
				$Q_n$ & covariance $\Pi_n \Pi_n^* \in \LO(H_n)$ \\
				$\rd W_n$ & \raggedright\arraybackslash $H_n$-valued $Q_n$-cylindrical Wiener noise on $(\Omega, \cF, \bbP)$ \\
				$W_{A_n}$ & stochastic convolution, see~\eqref{eq:mildsol} \\
				$\xi_n$ & initial datum \\
				$F_n$ & \raggedright\arraybackslash drift operator \\
				$L_F$, $C_F$ & \raggedright\arraybackslash  Lipschitz and growth constants of $E_n$-valued drifts, see~\ref{item:F-globalLip-linearGrowth} \\
				$\widetilde L_F$, $\widetilde C_F$ & \raggedright\arraybackslash  Lipschitz and growth constants of $B_n$-valued drifts, see~\ref{item:F-globalLip-linearGrowth-B} \\
				$L_F^{(r)}$, $C_{F,0}$ & \raggedright\arraybackslash local Lipschitz and growth constants of $E_n$-valued drifts, see~\ref{item:F-locLip-unifBdd} \\
				$f_n$ & \raggedright\arraybackslash real-valued drift coefficient function \\
				$\widetilde L_f$, $\widetilde C_f$& \raggedright\arraybackslash Lipschitz and growth constants of $(f_n)_{n\in\clos\N}$ see Assumption~\ref{ass:nonlinearity}\ref{ass:nonlinearity:globalLip-linGrowth}
			\end{tabularx}
		\end{minipage}
		\caption{\label{tab:notation} A selection of notation which is used repeatedly throughout this work.}
	\end{table}
	Table~\ref{tab:notation} lists some of the most important notation which recurs throughout this work.
	
	The Cartesian product $\prod_{j \in \cI} B_j$ of an 
	indexed family of sets $(B_j)_{j \in \cI}$ is comprised of all 
	functions $f \from \cI \to \bigcup_{j\in\cI} B_j$ satisfying
	$f(j) \in B_j$ for all $j \in \cI$.
	Given two sets  
	$\mathscr{P},\mathscr{Q}$ 
	and maps 
	$\mathscr{F},\mathscr{G}\from \mathscr{P}\times\mathscr{Q}\to\R$, 
	we write
	$\mathscr{F} (p,q) \lesssim_{q} \mathscr{G} (p,q)$ 
	to indicate the existence of some $C \from \mathscr Q \to [0,\infty)$
	such that
	$\mathscr{F} (p,q) \leq C(q) \, \mathscr{G} (p,q)$
	for all $(p,q) \in \mathscr P \times \mathscr{Q}$.
	We write $\mathscr{F} (p,q) \eqsim_q \mathscr{G} (p,q)$ 
	if both
	$\mathscr{F} (p,q) \lesssim_q \mathscr{G}(p,q)$ 
	and $\mathscr{G} (p,q) \lesssim_q \mathscr{F}(p,q)$ hold. 
	
	All normed spaces will be considered over the real or complex scalar field.
	Results concerning spectra and complex interpolation are naturally formulated for complex spaces;
	if the space is real, then we implicitly apply them to the complexification and subsequently restrict back to the original space. 
	
	The space $\LO(E; F)$ of bounded linear operators from $E$ to $F$ shall be equipped with the norm $\norm{T}{\LO(E; F)} \deq \sup_{\norm{x}{E} = 1} \norm{Tx}{F}$.
	We call $T \in \LO(E; F)$ a contraction if $\norm{T}{\LO(E; F)} \le 1$; in particular, the inequality need not be strict.
	An operator $T \in \LO(H)$ is said to be positive definite if there exists $\theta \in (0,\infty)$ such that $\scalar{Tx, x}{H} \ge \theta \norm{x}{H}^2$ for all $x \in H$, and nonnegative definite if $\scalar{Tx, x}{H} \ge 0$.
	Given $T, S \in \LO(H; K)$, we set $\scalar{T, S}{\LO_2(H; K)} \deq \sum_{j=1}^\infty \scalar{Te_j, S e_j}{K}$, where $(e_j)_{j\in\N}$ is any orthonormal basis of~$H$; the space $\LO_2(H; K)$ of Hilbert--Schmidt operators consists of those $T$ for which $\norm{T}{\LO_2(H; K)}^2 \deq \scalar{T, T}{\LO_2(H; K)} < \infty$.
	
	Given a measure space $(S, \mathscr A\!, \nu)$, a function $f \from S \to E$ is said to be strongly measurable if it can be approximated $\nu$-a.e.\ by simple functions.
	For $p \in [1,\infty]$, let $L^p(S; E)\deq L^p(S, \mathscr A \!, \nu; E)$ denote the Bochner space of (equivalence classes of) strongly measurable and $p$-integrable functions from $S$ to $E$, with norm 
	\[
	\norm{f}{L^p(S; E)} \coloneqq
	\begin{cases}
		(\int_S \norm{f(s)}{E}^p \rd 
		\nu(s))^{\frac{1}{p}}, &\text{if } p \in [1,\infty); \\
		\operatorname{ess\,sup}_{s \in 
			S}\norm{f(s)}{E}, &\text{if } p = \infty.
	\end{cases}
	\]
	Sub-intervals $J \subseteq \R$ are equipped with the Lebesgue $\sigma$-algebra and measure. 
	The Banach space of (bounded) continuous functions $u \from J \to E$, endowed with the supremum norm, is denoted by $C(J; E)$.
	
	The meaning of the tensor symbol $\otimes$ will depend on the context: 
	Given a map $\Phi \from (a,b) \to \LO(E; F)$ and some $x \in E$, we define the function $\Phi \otimes x \from (a,b) \to F$ by $[\Phi \otimes x](t) \deq \Phi(t)x$.
	If instead an $h \in H$ is given, we define $h \otimes x \in \LO(H; E)$ to be the rank-one operator $[h \otimes x](u) \deq \scalar{h, u}{H} x$.
	The space of all (finite) linear combinations of such operators is denoted by $H \otimes E$.
	We define the convolution $\Psi * f \from [0,T] \to F$ of the functions $\Psi \from [0,T] \to \LO(E; F)$ and $f \from [0,T] \to E$ by $[\Psi * f](t) \deq \int_0^t \Psi(t-s) f(s) \rd s$.
	
	\subsection{Discrete-to-continuum Trotter--Kato approximation theorem}
	\label{sec:d2c}

	We encode the discrete-to-continuum setting in the following way:
	
	\begin{assumption}\label{ass:disc3}
		Let $(E_n, \norm{\,\cdot\,}{E_n})_{n \in \clos \N}$ and $(\widetilde E, \norm{\,\cdot\,}{\widetilde E})$ be real or complex Banach spaces and suppose that $E_\infty$ is a closed linear subspace of $\widetilde E$. 
		We assume that there exist operators $\Pi_n \in \LO(\widetilde E; E_n)$ and $\Lambda_n \in \LO(E_n; \widetilde E)$ for all $n \in \N$ which satisfy
		\begin{enumerate}
			\item \label{ass:disc:proj-emb-bounds3} 
			$M_\Pi \deq \sup_{n\in\N} 
			\norm{\Pi_n}{\LO(\widetilde E; E_n)} < \infty$
			and 
			$M_\Lambda \deq \sup_{n\in\N} 
			\norm{\Lambda_n}{\LO(E_n; \widetilde E)} < \infty$;
			\item \label{ass:disc:approx-id3} $\Lambda_n \Pi_n x \to x$ in 
			$\widetilde E$ as $n \to 
			\infty$ for all $x \in E_\infty$;
			\item \label{ass:disc:id3} $\Pi_n \Lambda_n = \id_{E_n}$ for 
			all $n \in \N$.
		\end{enumerate}
		In addition, we denote $\Pi_\infty = \Lambda_\infty \deq \id_{\widetilde E}$ for convenience.
	\end{assumption}
	Note that parts~\ref{ass:disc:proj-emb-bounds3} and~\ref{ass:disc:id3} together 
	imply that the lifting operators are continuous embeddings $\Lambda_n \from E_n \hookrightarrow \widetilde E$.
	In applications, they will typically be nested (in the sense that $E_n \hookrightarrow E_{n+1}$ for all $n \in \N$) and finite-dimensional, but neither of these assumptions is strictly necessary in the abstract theory.
	Moreover, we will often have $\widetilde E = E_\infty$, but not always; see Section~\ref{sec:reaction-diffusion}.
	
	Now we consider the following sequence of linear operators on $(E_n)_{n\in\clos\N}$:
	\begin{assumption}\label{ass:operators3}
		For all $n \in \clos \N$, let $-A_n \from \dom{A_n} \subseteq E_n \to E_n$ be a linear operator generating a strongly continuous semigroup $(S_n(t))_{t\ge0} \subseteq \LO(E_n)$, meaning that $S_n(0) = \id_{E_n}$, $S_n(t + s) = S_n(t)S_n(s)$ for all $t, s \ge 0$, and $S_n(t)x \to x$ as $t \downarrow 0$ for all $x \in E_n$.
		Suppose that there exist $M_S \in [1,\infty)$ and $w \in \R$ such that
		\begin{equation}\label{eq:semigroup-estimates}
			\norm{S_n(t)}{\LO(E_n)} 
			\le 
			M_S e^{-w t}
			\quad 
			\text{for all } 
			n \in \clos \N 
			\text{ and } 
			t \in [0,\infty).
		\end{equation}
	\end{assumption}

	Given a linear operator $A \from \dom{A} \subseteq E \to E$ on a 
	Banach space $(E, \norm{\,\cdot\,}{E})$, where $\dom{A}$ denotes its domain,
	we say that $\lambda\in \C$ belongs to the 
	\emph{resolvent set $\rho(A)$ of $A$} if the 
	corresponding \emph{resolvent operator $R(\lambda,A) := (\lambda \id_E - 
		A)^{-1}$}
	exists 
	in $\mathscr L(E)$.
	Given a sequence $(A_n)_{n\in\clos\N}$ of such operators, generating strongly continuous semigroups with uniform growth bounds, the Trotter--Kato approximation theorem (see, e.g.,~\cite[Chapter~III, 
	Theorem~4.8]{EngelNagel2000}) establishes a link between the strong convergence of resolvents and uniform convergence of the semigroups on compact subintervals of~$[0,\infty)$.
	The following discrete-to-continuum analog of this result was proved by Ito and 
	Kappel~\cite[Theorem~2.1]{ItoKappel1998}:
	\begin{theorem}[{Discrete-to-continuum 
			Trotter--Kato approximation}]\label{thm:d2c-TKapprox}
		Let Assumptions~\ref{ass:disc3} and~\ref{ass:operators3} be 
		satisfied, with $w \in \R$.
		The following statements are equivalent:
		\begin{enumerate}[(a)]
			\item\label{thm:d2c-TKapprox:a} There exists a $\lambda \in 
			\bigcap_{n\in\clos\N} \rho(A_n)$ such 
			that, for every $x \in E_\infty$,
			\begin{equation}
				\Lambda_n R(\lambda, A_n) \Pi_n x \to R(\lambda, A_\infty)x
				\quad 
				\text{in } \widetilde E \quad \text{as } n \to \infty.
			\end{equation}
			\item \label{thm:d2c-TKapprox:b}
			For all $x \in E_\infty$ and $T \in (0,\infty)$ it holds that 
			\begin{equation}
				\Lambda_n S_n \Pi_n \otimes x
				\to 
				S_\infty \otimes x
				\quad 
				\text{in } C([0,T]; \widetilde E)
				\quad 
				\text{as } n \to \infty.
			\end{equation}
		\end{enumerate}
		If~\ref{thm:d2c-TKapprox:a} holds for some $\lambda \in 
		\bigcap_{n\in\clos\N} \rho(A_n)$ (or, equivalently, if~\ref{thm:d2c-TKapprox:b} holds), then~\ref{thm:d2c-TKapprox:a} holds in fact for \emph{every} $\lambda \in \C$ such that $\Re \lambda < w$.
	\end{theorem}
	
	\subsection{Stochastic integration in UMD-type-2 Banach spaces}
	\label{sec:stoch-int-UMD-type2}
	
	Given a separable Hilbert space $(H, \scalar{\,\cdot\, , \,\cdot\,}{H})$ over the real scalar field, let $(W(t))_{t\ge0}$ be an $H$-valued cylindrical Wiener process with respect to some filtered probability space $(\Omega, (\cF_t)_{t\ge0}, \cF, \bbP)$.
	A rigorous definition can be found in~\cite[Section~2.5.1]{LiuRockner2015}.
	For our purpose of constructing the stochastic integral below, it suffices to define $\scalar{W(t), h}{H} \deq \sum_{j=1}^\infty \beta_j(t) \scalar{e_j, h}{H}$ for all $t \ge 0$ and $h \in H$, where $(e_j)_{j\in\N}$ is an orthonormal basis of $H$ and $(\beta_j(\,\cdot\,))_{j\in\N}$ is a sequence of independent (real-valued) Brownian motions on $(\Omega, (\cF_t)_{t\ge0}, \cF, \bbP)$. 
	Then $(\scalar{W(t), h}{H})_{t\ge0}$ is a well-defined Brownian motion for every $h \in H$, and intuitively one can think of $(W(t))_{t\ge0}$ as being given by $W(t) = \sum_{j=1}^\infty \beta_j(t) e_j $.
	
	Let $(E, \norm{\,\cdot\,}{E})$ be a real Banach space, and let $(\gamma_j)_{j\in\N}$ be a sequence of independent (real-valued) standard normal random variables on a probability space $(\Omega', \cF', \bbP')$, independent of the probability spaces $(\Omega, \cF, \bbP)$ and $(\widetilde \Omega, \widetilde{\cF}, \widetilde{\bbP})$ used in the rest of this work. 
	We define the space $\gamma(H; E)$ of \emph{$\gamma$-radonifying operators from $H$ to $E$} as the completion of the finite-rank operators $H \otimes E$ with respect to the norm
	$
	\Norm{\sum_{j=1}^n h_j \otimes x_j}_{\gamma(H; E)} 
	\deq 
	\Norm{\sum_{j=1}^n \gamma_j x_j}_{L^2(\Omega'; E)},
	$
	where we assume that the $(h_j)_{j=1}^n$ are $H$-orthonormal.
	This norm is well-defined, i.e., it can be checked that the right-hand side is independent of the choice of representation.
	An important feature of $\gamma(H; E)$ is its \emph{ideal property} (in the algebraic sense)~\cite[Theorem~9.1.10]{HvNVWVolumeII}, which states that for all $T \in \gamma(H; E)$, $U \in \LO(E; F)$ and $S \in \LO(K; H)$, we have 
	\begin{equation}\label{eq:ideal-property}
		UTS \in \gamma(K; F)
		\quad \text{with} \quad 
		\norm{UTS}{\gamma(K; F)} 
		\le 
		\norm{U}{\LO(E; F)} \norm{T}{\gamma(H; E)} \norm{S}{\LO(K; H)}.
	\end{equation}
	For any rank-one operator $h \otimes x \in \LO(H; E)$, we have $h \otimes x \in \gamma(H; E)$ with
	\begin{equation}\label{eq:gamma-norm-pure-tensor}
		\norm{h \otimes x}{\gamma(H; E)} = \norm{h}{H} \norm{x}{E}.
	\end{equation}

	The stochastic integral of an elementary integrand $\Phi \from (0,\infty) \to H \otimes E$, i.e., a function of the form $\Phi(t) = \sum_{j=1}^n \indic{(a_j, b_j]}(t) h_j \otimes x_j$, is defined by 
	\begin{equation}
		\int_0^\infty \Phi(t) \rd W(t) \deq \sum_{j=1}^n \bigl(\scalar{h_j, W(b_j)}{H} - \scalar{h_j, W(a_j)}{H}\bigr) x_j
		\in 
		L^2(\Omega; E).
	\end{equation}
	In order to extend the definition of the stochastic integral beyond elementary integrands, one needs to impose further geometric assumptions on the Banach space~$E$.
	In this article we work in one of the standard settings, namely that of spaces with \emph{unconditional martingale differences} and \emph{Rademacher type} $2$ (abbreviated to \emph{UMD-type-2}).
	Definitions of these notions can be found in~\cite[Section~4.2]{HvNVWVolumeI} and~\cite[Section~7.1]{HvNVWVolumeII}, respectively, but we will only use them to ensure existence of stochastic integrals.
	In this case, one can establish the \emph{It\^o inequality}
	\begin{equation}\label{eq:Ito-inequality}
		\Norm[\bigg]{\int_0^\infty \Phi(t) \rd W(t)}_{L^2(\Omega; E)}
		\lesssim_E 
		\norm{\Phi}{L^2(0,\infty; \gamma(H; E))}
	\end{equation}
	for elementary integrands~\cite[Proposition~4.2]{vNVW2015}, and use it to extend the definition of the stochastic integral to all $\Phi \in L^2(0,\infty; \gamma(H; E))$.
	In fact, since we are only concerned with deterministic integrands in this work, we could suffice with the type $2$ assumption.
	Despite this, we additionally impose the UMD assumption for the sake of compatibility with some of the literature, and because the concrete examples of Banach spaces in which we are interested (such as the Lebesgue $L^q$-spaces for $q \in [2,\infty)$) satisfy both properties.
	For more details on stochastic integration in Banach spaces, we refer to the survey article~\cite{vNVW2015}.
	
	The exponent $2$ in $L^2$ appearing on both sides of~\eqref{eq:Ito-inequality} can be replaced by any other $p \in [1, \infty)$ at the cost of a $p$-dependent constant, see for instance~\cite[Theorem~4.7]{vNVW2015}.
	If $E$ is also a Hilbert space, then $\gamma(H; E)$ is isometrically isomorphic to the space $\LO_2(H; E)$, see~\cite[Proposition~9.1.9]{HvNVWVolumeII}, and instead of the inequality~\eqref{eq:Ito-inequality} we have the It\^o \emph{isometry} between $L^2(0,\infty; \LO_2(H; E))$ and $L^2(\Omega; E)$.
	
	\section{Graph-based semilinear stochastic evolution equations  with Whittle--Mat\'ern linear operators}
	\label{sec:graph-WM}
	
	Before developing the general discrete-to-continuum convergence results summarized by Theorem~\ref{thm:d2c-conv-informal} in the upcoming sections, in this section we demonstrate how they can be applied to the particular case of equations whose linear parts are graph discretizations of a generalized Whittle--Mat\'ern operator on a manifold.
	In the spatial and linear case, such convergence results have been proven in~\cite{SanzAlonsoYang2022, SanzAlonsoYang2022JMLR}.
	We also mention the work~\cite{pmlr-v151-nikitin22a}, in which the statistical properties of the spatiotemporal \emph{linear} equation were investigated for fixed $n \in \N$.
	
	\subsection{Geometric graphs and generalized Whittle--Mat\'ern operators}
	\label{sec:geo-graphs-WMoperators}
	
	\begin{assumption}[Manifold assumption]\label{ass:manifold}
			Given $m, d \in \N$, suppose that $\cM$ is an $m$-dimensional smooth, connected, compact Riemannian manifold without boundary, embedded smoothly and isometrically into $\R^d$.
	   Let $\mu$ and $d_{\cM}$ denote the normalized volume measure and geodesic metric on $\cM$, respectively.
	\end{assumption}
	
	For each $n \in \N$, let a point cloud $\cM_n \coloneqq (x_{n}^{(j)})_{j=1}^{n} \subseteq \cM$ be given. 
	We suppose that $\cM$ can be partitioned into $n$ regions of mass $1/n$, which can be transported to the corresponding $n$ points comprising $\cM_n$, in such a way that the maximal geodetic displacement tends to zero as $n \to \infty$. More precisely, we assume that there exists a sequence $(T_n)_{n\in\N}$ of \emph{transport maps} $T_n \from \cM \to \cM_n$ such that 
	\begin{align}
		\label{eq:transport-maps}
		\mu_n ={}& T_{n\sharp} \mu 
		&
		\text{for all } n &\in \N, \quad \text{and}
		\\
		\label{eq:mesh-vanishes}
		\varepsilon_n \deq{}& \sup_{x\in \cM} d_{\cM}(x, T_n(x)) \to 0 
		& \text{as } n &\to \infty.
	\end{align}
	Here, $\mu_{n} \coloneqq \frac{1}{n} \sum_{j=1}^{n} \delta_{x_{n}^{(j)}}$ is the empirical measure on $\cM$ associated to $\cM_n$, and $T_{n\sharp} \mu$ denotes the pushforward measure $T_{n\sharp} \mu(B) \deq \mu(\{T_n \in B\})$ on $\cM_n$.
	Two different examples in which this assumption is satisfied are presented in Settings~\ref{ex:square-grid} and~\ref{ex:sampling-manifold} below.
	
	Given $u_n \from \cM_n \to \R$ for $n \in\clos\N$, these transport maps enable us to define the functions $\Lambda_u u_n \from \cM \to \R$ and $\Pi_n u \from \cM_n \to \R$ by setting
	\begin{equation}\label{eq:graph-disc-Lambda-Pi}
		\Lambda_n u_n(x) \deq u_n(T_n(x))
		\quad \text{and} \quad 
		\Pi_n u(x_n^{(j)}) 
		\deq
		n \int_{V_n^{(j)}} u(x) \rd \mu(x),
	\end{equation}
	respectively,
	for all $x \in \cM$ and $j \in \{1,\dots,n\}$, 
	where $V_n^{(j)} \deq \{T_n = x_n^{(j)}\} \subseteq \cM$.
	
	It turns out that the operations defined in~\eqref{eq:graph-disc-Lambda-Pi} satisfy Assumption~\ref{ass:disc3} with respect to the following function spaces: 
	Given $q \in [1,\infty]$ and $n \in \N$, we set $E_n \deq L^q(\cM_n) \deq L^q(\cM, \mu_n)$, as well as $\widetilde E \deq L^q(\cM)$ and
	\begin{equation}
		E_\infty \deq 
		\begin{cases}
			L^q(\cM), &\text{if } q \in [1,\infty); \\
			C(\cM), &\text{if } q = \infty.
		\end{cases}
	\end{equation}
	Later on, we will need $q \in [2,\infty)$, so that $E_n$ is a UMD-type-2 space for use in stochastic integration, but the statements here hold for all $q \in [1,\infty]$.
	
	For these spaces, Assumption~\ref{ass:disc3}\ref{ass:disc:proj-emb-bounds3} is satisfied with $M_{\Lambda} = 1$ and $M_\Pi \le 1$.
	Indeed, the fact that $\Lambda_n$ is an isometry follows from~\eqref{eq:transport-maps} if $q \in [1,\infty)$, whereas for $q = \infty$ we see directly from the definition that 
	\[
	\norm{\Lambda_n u_n}{L^\infty(\cM)} 
	= 
	\sup_{x \in \cM} \abs{u_n(T_n(x)) }
	= 
	\max_{j=1}^n \abs{u_n(x_n^{(j)})}
	=
	\norm{u_n}{L^\infty(\cM_n)}.
	\]
	To show that $\Pi_n$ is a contraction, we first apply H\"older's inequality in~\eqref{eq:graph-disc-Lambda-Pi} with $\frac{1}{q} + \frac{1}{q'} = 1$ to find
	$\abs{ \Pi_n u(x_n^{(j)})} \le n \norm{u}{L^q(V_n^{(j)})} \mu(V_n^{(j)})^{\frac{1}{q'}} = n^{\frac{1}{q}} \norm{u}{L^q(V_n^{(j)})}$, so that
	\begin{equation}
		\norm{\Pi_n u}{L^q(\cM_n)}^q 
		=
		\frac{1}{n} \sum_{j=1}^n \abs{\Pi_n u(x_n^{(j)})}^q
		\le 
		\sum_{j=1}^n \norm{u}{L^q(V_n^{(j)}\!\!,\, \mu)}^q
		=
		\norm{u}{L^q(\cM)}^q.
	\end{equation}
	Assumption~\ref{ass:disc3}\ref{ass:disc:approx-id3} is a consequence of~\eqref{eq:mesh-vanishes}, from which it follows that $\Lambda_n \Pi_n u \to u$ in $L^q(\cM)$ for any $u \in C(\cM)$.
	For $q = \infty$, this is what we wanted to show; if $q \in [1,\infty)$, then $C(\cM)$ is dense in $L^q(\cM)$, and the fact that $(\Lambda_n \Pi_n)_{n\in\N}$ is uniformly bounded in $\LO(L^q(\cM))$ by Assumption~\ref{ass:disc3}\ref{ass:disc:proj-emb-bounds3} which we have just proven to hold, yields $\Lambda_n \Pi_n u \to u$ for all $u \in L^q(\cM)$ as desired.
	Assumption~\ref{ass:disc3}\ref{ass:disc:id3} can be verified via direct computation using the definitions. 
	Finally, we have
	\begin{align}
		\int_{\cM} \Lambda_n u_n(x) &v(x) \rd \mu(x)
		=
		\int_{\cM} u_n(T_n(x)) v(x) \rd \mu(x)
		=
		\sum_{j=1}^n \int_{V_n^{(j)}} u_n(x_n^{(j)}) v(x) \rd \mu(x)
		\\&=
		\frac{1}{n} 
		\sum_{j=1}^n u_n(x_n^{(j)}) \Pi_n v(x_n^{(j)})
		=
		\int_{\cM_n} u_n(x) \Pi_n v(x) \rd \mu_n(x).
		\label{eq:Lambda-Pi-adjoint}
	\end{align}
	which shows that the adjoint of $\Pi_n \in \LO(L^q(\cM); L^q(\cM_n))$, where $q \in [1,\infty)$, is given by $\Pi_n^* = \Lambda_n \in \LO(L^{q'}(\cM_n); L^{q'}(\cM))$.
	
	The concrete choices of $\cM$ and their discretizations which we will consider in this section are the following two:
	\begin{setting}[Square grid on $\bbT^m$]\label{ex:square-grid}
		Let $\cM \deq \bbT^m$ be the $m$-dimensional flat torus, which we view as the cube $[0,1]^m$ endowed with periodic boundary conditions. 
		For notational convenience, we will index our sequence of discretizations of $\bbT^m$ only by the natural numbers $n$ such that $n^{1/m} \in \N$, for which we define the following square equidistant grid with mesh size $h_n \deq n^{-1/m}$:
		\[
		\cM_{n} 
		\deq 
		\{\tfrac{1}{2}n^{-1/m}, \tfrac{3}{2}n^{-1/m}, \dots, 1 - \tfrac{1}{2}n^{-1/m}\}^m.
		\]
		Then the grid points of $\cM_n$ can be written as $\mathbf x_n^{(\mathbf j)} = n^{-1/m}(\mathbf j - \frac{1}{2}\mathbf 1)$ for some $m$-tuple $\mathbf j \in \{1, \dots, n^{1/m}\}^m$, where $\mathbf 1 = (1,\dots,1) \in \R^m$. 
		To each of these points $\mathbf x_n^{(\mathbf j)} \in \cM_n$, we associate the half-open cube $U_n^{(\mathbf j)} \deq \prod_{k=1}^m \bigl[n^{-1/m}(j_k - 1), n^{-1/m}j_k\bigr)$.
		Since these cubes form a partition of $\cM$ (recalling that opposite sides are identified), we can define the transport map $T_n \from \cM \to \cM_n$ by $T_n(x) \deq \mathbf x_n^{(\mathbf j)}$ whenever $x \in U_n^{(\mathbf j)}$. 
		It readily follows that~\eqref{eq:transport-maps} holds, as does~\eqref{eq:mesh-vanishes}, since
		$
		\varepsilon_n = \frac{1}{2}\sqrt{m} n^{-1/m}
		$ for all $n \in \N$.
	\end{setting}
	
	\begin{setting}[Randomly sampled point cloud]\label{ex:sampling-manifold}
		Let $\cM$ be any manifold satisfying Assumption~\ref{ass:manifold}, and 
		let $(x^{(n)})_{n\in\N} \subseteq \cM$ be a sequence of points independently sampled from $\mu$.
		This sequence can be viewed as a sample from the product probability space $(\widetilde \Omega, \widetilde \cF, \widetilde \bbP) \deq \prod_{n\in\N} (\cM, \cB(\cM), \mu)$, where $\cB(\cM)$ denotes the Borel $\sigma$-algebra on $\cM$. 
		If we set $\cM_n \deq (x^{(j)})_{j=1}^n$ for all $n \in \N$, then~\cite[Proposition~4.1]{SanzAlonsoYang2022} states that, $\widetilde\bbP$-a.s., there exists a sequence $(T_n)_{n\in\N}$ of transport maps $T_n \from \cM \to \cM_n$ for which~\eqref{eq:transport-maps} holds and
		\begin{equation}\label{eq:mesh-size-rate}
			\varepsilon_n \lesssim_{\cM} (\log n)^{c_m} n^{-1/m} \to 0
			\quad 
			\text{as } n \to \infty,
		\end{equation}
		where $c_m = \frac{3}{4}$ if $m = 2$ and $c_m = \frac{1}{m}$ otherwise.
	\end{setting}
	
	We will now introduce the linear operators which we consider on the domains $\cM$ and $(\cM_n)_{n\in\N}$ as in Settings~\ref{ex:square-grid} and~\ref{ex:sampling-manifold}.
	Given the coefficient functions $\tau \from \cM \to [0,\infty)$ and $\kappa \from \cM \to [0,\infty)$, respectively assumed to be Lipschitz and continuously differentiable, we consider the nonnegative and symmetric second-order linear differential operator $\cL^{\tau, \kappa}_\infty$ formally defined by 
	\begin{equation}\label{eq:base-operator-coeffs}
		\cL^{\tau, \kappa}_\infty u \coloneqq \tau u - \nabla \cdot (\kappa \nabla u), 
	\end{equation}
	for	$u$ belonging to some appropriate domain $\dom{\cL^{\tau, \kappa}_\infty} \subseteq L^2(\cM)$. 
	
	For each $n \in \N$, we endow $\cM_n$ with a (weighted, undirected) graph structure by viewing its points as vertices and defining the weight matrix $\mathbf W_n \in \R^{n \times n}$ by
	\begin{equation}\label{eq:graph-weights}
		(\mathbf W_n)_{ij}
		\coloneqq 
		\frac{2(m+2)}{\nu_m}\frac{1}{n h_n^{m+2}}
		\indic{[0, h_n]}\bigl(\norm{x_n^{(i)} - x_n^{(j)}}{\R^d}\bigr),
	\end{equation}
	where $\nu_m$ denotes the volume of the unit sphere in $\R^m$ and $h_n \in (0,\infty)$ is a given graph connectivity length scale. 
	With these weights, the resulting graph is an example of a \emph{geometric graph} (or in fact a \emph{random geometric graph} if the nodes are sampled randomly as in Setting~\ref{ex:sampling-manifold}).
	The results in this section are likely to remain valid if the indicator function $\indic{[0, h_n]}$ in~\eqref{eq:graph-weights} is replaced by a more general (e.g., Gaussian) cut-off kernel (such as in~\cite{CalderGTLewicka2022}), but we only consider $\eta = \indic{[0, h_n]}$ in order to also cite sources which are not formulated in this generality.
	
	The graph-discretized counterpart $\cL_n^{\tau,\kappa}$ of~\eqref{eq:base-operator-coeffs} is then the operator which acts on a given function $u \from \cM_n \to \R$ as 
	\begin{equation}\label{eq:base-operator-disc}
		\cL_n^{\tau,\kappa} u(x_n^{(i)}) 
		\deq 
		\tau(x_n^{(i)}) u(x_n^{(i)}) 
		+
		\sum_{j=1}^n (\mathbf W_n)_{ij} \sqrt{\kappa(x_n^{(i)})\kappa(x_n^{(j)})}
		\bigl(u(x_n^{(i)}) - u(x_n^{(j)})\bigr).
	\end{equation}
	This can be seen as a generalized version of the (unnormalized) graph Laplacian $\Delta_{\cM_n}$, and in fact reduces to it if $\tau \equiv 0$ and $\kappa \equiv 1$.
	\begin{assumption}[Coefficients of $\cL_{n}^{\tau,\kappa}$]\label{ass:kappa-tau}
		Let $\tau \from \cM \to [0,\infty)$ and $\kappa \from \cM \to [0,\infty)$ be the coefficient functions used to define the base operators $(\cL_{n}^{\tau,\kappa})_{n\in\clos\N}$ in~\eqref{eq:base-operator-coeffs} and~\eqref{eq:base-operator-disc}.
		We shall suppose that
		\begin{enumerate}
			\item\label{ass:kappa-tau:general} 
			$\tau$ is Lipschitz, whereas $\kappa$ is continuously differentiable and bounded below away from zero.
		\end{enumerate}
		For some results, we specialize to the case that
		\begin{enumerate}
			\setcounter{enumi}{1}
			\item\label{ass:kappa-tau:Laplacian} 
			$\tau \equiv 0$ and $\kappa \equiv 1$, i.e., $\cL^{\tau,\kappa}_n = \Delta_{\cM_n}$ and $\cL^{\tau,\kappa}_\infty$ and reduces to the Laplace--Beltrami operator on $\cM$.
		\end{enumerate}
	\end{assumption}
	
	\begin{assumption}[Connectivity length scale of random graph]
		\label{ass:connectivity-regimes}
		Let the manifold $\cM$ and the random point clouds $(\cM_n)_{n\in\N}$ on the probability space $(\widetilde \Omega, \widetilde \cF, \widetilde \bbP)$ be as in Setting~\ref{ex:sampling-manifold}.
		Let $(h_n)_{n\in\N} \subseteq (0,\infty)$ determine the connectivity length scales of the graphs associated to $(\cM_n)_{n\in\N}$ via the weights~\eqref{eq:graph-weights}, and suppose that $s \in (0,\infty)$.
		We will assume one of the following:
		\begin{enumerate}
			\item\label{ass:connectivity-regimes:optimal-range} 
			There exists a $\beta > \frac{m}{4s}$ such that
			$
			(\log n)^{c_m} n^{-\frac{1}{m}} \ll h_n \ll n^{-\frac{1}{4s\beta}}.
			$
			\item\label{ass:connectivity-regimes:suboptimal-range} 
			There exists a $\delta > 0$, so small that $\frac{m}{1-\delta} < m + 4 + \delta$, and a $\beta > \frac{m+4+\delta}{2s}$ such that
			$
			n^{-\frac{1}{m+4+\delta}}
			\lesssim 
			h_n 
			\ll
			n^{-\frac{1}{2s\beta}}.
			$
		\end{enumerate}
		Given two sequences 
		$(a_n)_{n\in\N}$ and $(b_n)_{n\in\N}$ of positive real numbers, the notation 
		$a_n \ll b_n$ means $a_n / b_n \to 0$ as $n \to \infty$.
	\end{assumption}

	Since $\id_n + \cL_n^{\tau,\kappa}$ is self-adjoint, positive definite and has a compact inverse for all $n \in \clos\N$ (cf.~\cite[Chapter~XII]{Taylor1981} in the case $n = \infty$), there exists an orthonormal basis $(\psi^{(j)}_n)_{j = 1}^{n}$ of $L^2(\cM_n)$ and a non-decreasing sequence $(\lambda^{(j)}_n)_{j = 1}^{n} \subseteq [0,\infty)$, accumulating only at infinity for $n = \infty$, such that $\cL_n^{\tau,\kappa} \psi^{(j)}_n = \lambda^{(j)}_n 
	\psi^{(j)}_n$ for all $j \in \{1,\dots,n\}$. 
	We summarize this state of affairs by saying that $(\psi^{(j)}_n, \lambda^{(j)}_n)_{j = 1}^{n}$ is an orthonormal eigenbasis of $L^2(\cM_n)$ associated to $\cL_n^{\tau,\kappa}$.
	The asymptotic behavior of the eigenvalues $(\lambda_\infty^{(j)})_{j\in\N}$ is described by \emph{Weyl's law}, cf.~\cite[Theorem~XII.2.1]{Taylor1981}:
	\begin{equation}\label{eq:weyls-law}
		\lambda^{(j)}_\infty
		\eqsim_{(\cM, \tau, \kappa)}
		j^{2/m} \quad \text{for all } j \in \N.
	\end{equation}

	Given any of the above settings and $n \in \clos \N$, we define the generalized Whittle--Mat\'ern operator $A_n$ on $L^2(\cM_n)$ as a fractional power of the symmetric elliptic operator $\cL_n^{\tau,\kappa}$ given by~\eqref{eq:base-operator-coeffs} and~\eqref{eq:base-operator-disc}.
	That is, we set $A_n \deq (\cL^{\tau, \kappa}_n)^s$ for some $s \in [0,\infty)$, where we use the spectral definition of fractional powers:
	\begin{equation}\label{eq:def-WM-operator}
		A_n u = (\cL^{\tau, \kappa}_n)^s u \deq \sum_{j=1}^n [\lambda^{(j)}_n]^{s} \scalar{u, 
			\psi_n^{(j)}}{L^2(\cM_n)} \psi_n^{(j)}, \quad u \in \dom{A_n} 
		\subseteq 
		L^2(\cM_n).
	\end{equation}
	These will be used as the linear operators $(A_n)_{n\in\clos\N}$ in the stochastic partial differential equations in the next subsection.
	
	Since $A_n$ is a nonnegative definite and self-adjoint operator on $L^2(\cM_n)$ for any $n \in \clos \N$, the Lumer--Phillips theorem~\cite[Theorem~13.35]{Neerven2022} implies that $-A_n$ generates a contractive analytic $C_0$-semigroup $(S_n(z))_{z\in\Sigma_\eta} \subseteq \LO(L^2(\cM_n))$ on the sector 
	\begin{equation}\label{eq:def-sector}
		\Sigma_\eta \deq \{\lambda\in\C\setminus \{0\}: \arg\lambda \in (-\eta,\eta)\}
	\end{equation}
	for every $\eta \in (0, \frac{1}{2}\pi)$.
	Thus, the operators $(A_n)_{n\in\clos\N}$ on $(L^2(\cM_n))_{n\in\clos\N}$ are uniformly sectorial of angle $0$, see Appendix~\ref{app:uniform-sectoriality}.
	
	The following additional assumption(s) on the $L^\infty(\cM_n)$-boundedness of the semigroups will be needed for some of the results in Section~\ref{sec:conv-graphdiscr-SPDEs}:
	\begin{assumption}[Uniform $L^\infty$-boundedness of semigroups]
		\label{ass:uniform-Linfty-bdd-semigroup}
		Suppose that
		\begin{enumerate}
			\item 
			\label{ass:uniform-Linfty-bdd-semigroup:bdd}
			there exists a constant $M_{S,\infty} \in [1,\infty)$ such that 
			\begin{equation}
				\norm{S_n(t)}{\LO(L^\infty(\cM_n))} \le M_{S,\infty}
				\quad 
				\text{for all } 
				n \in \clos \N \text{ and } t \ge 0.
			\end{equation}
		\end{enumerate}
		We may sometimes additionally assume that
		\begin{enumerate}
			\setcounter{enumi}{1}
			\item 
			\label{ass:uniform-Linfty-bdd-semigroup:contractive}
			$(S_n(t))_{t\ge0}$ is $L^\infty(\cM_n)$-contractive for all $n \in \clos\N$, i.e., $M_{S, \infty} = 1$ in~\ref{ass:uniform-Linfty-bdd-semigroup:bdd}.
		\end{enumerate}
	\end{assumption}
	Under this assumption, it follows from~\cite[Proposition~3.12]{Ouhabaz2005} that $(S_n(z))_{z\in\Sigma_{\eta_q}}$ is bounded analytic on $L^q(\cM_n)$ with $\eta_q = \frac{2}{q}\eta$ for all $n \in \clos\N$ and $q \in (2,\infty)$, and its uniform norm bound on the sector $\Sigma_{\eta_q}$ only depends on $q$ and $M_{S,\infty}$.
	Therefore, the sequence of operators $(A_n)_{n\in\clos\N}$ on $(L^q(\cM_n))_{n\in\clos\N}$ is uniformly sectorial of angle at most $(\frac{1}{2} - \frac{1}{q})\pi$.
	
	\subsection{Convergence of graph-discretized semilinear SPDEs}
	\label{sec:conv-graphdiscr-SPDEs}
	
	Let $(W(t))_{t\ge0}$ be an $L^2(\cM)$-valued cylindrical Wiener process with respect to a filtered probability space $(\Omega, \cF, (\cF_t)_{t\in [0,T]}, \bbP)$.
	The spaces $\cM$ and $(\cM_n)_{n\in\N}$ are as in Setting~\ref{ex:square-grid} or~\ref{ex:sampling-manifold} above; in the latter case, note that the space $(\Omega, \cF, (\cF_t)_{t\in [0,T]}, \bbP)$ associated to the Wiener noise is independent of the probability space $(\widetilde \Omega, \widetilde{\cF}, \widetilde{\bbP})$ describing the randomness of the point cloud.
	For every $n \in \clos\N$, we set $W_n \deq \Pi_n W$ and consider the following semilinear stochastic partial differential equation (SPDE):
	\begin{equation}\label{eq:example-SPDE}
		\left\lbrace
		\begin{aligned}
			\rd u_n(t, x) + [\cL^{\tau,\kappa}_n]^s u_n(t,x) \rd t &= f_n(t, u_n(t, x)) \rd t + \rd W_n(t, x),
			\\
			u_n(0, x) &= \xi_n(x),
			\qquad 
			\qquad 
			\qquad
			(t,x) \in (0,T] \times \cM_n,
		\end{aligned}
		\right.
	\end{equation}
	where $s \in (0,\infty)$, $T \in (0,\infty)$ is a finite time horizon, $f_n \from \Omega \times [0,T] \times \R \to \R$ is the nonlinearity, and $\xi_n \from \Omega \times \cM_n \to \R$ is the initial datum.
	Note that $(W_n(t))_{t\ge0}$ is a $Q_n$-cylindrical Wiener process with $Q_n = \Pi_n^* \Pi_n = \Lambda_n \Pi_n = \id_n$, where we recall~\eqref{eq:Lambda-Pi-adjoint} for the second identity. 
	Therefore, $(W_n(t))_{t\ge0}$ is a cylindrical Wiener process on $L^2(\cM_n)$ for all $n \in \clos \N$, and its formal time derivative $\rd W_n$ represents spatiotemporal Gaussian white noise on $[0,T]\times\cM_n$.
	
	Solutions to~\eqref{eq:example-SPDE}---and all the other (semi)linear SPDEs that we consider in this work---are always interpreted in the mild sense. 
	This notion of solutions is defined using the semigroup $(S_n(t))_{t\ge0}$ generated by $-[\cL^{\tau,\kappa}_n]^s$.
	We say that $u_n$ is a \emph{global mild solution} to~\eqref{eq:example-SPDE} if it satisfies the following relation for all $t \in [0,T]$:
	\begin{equation}
		u_n(t) = S_n(t)\xi_n + \int_0^t S_n(t-s)F_n(s, u_n(s)) \rd s + 
		\int_0^t S_n(t-s) \rd W_n(s), 
		\quad 
		\bbP\text{-a.s.}
	\end{equation}
	Here, we interpret $u_n = (u_n(t))_{t\in[0,T]}$ as a process taking its values in an infinite-dimensional Banach space of functions on $\cM_n$ (such as $L^q(\cM_n)$ or $C(\cM_n)$), and we define for every $(\omega, t) \in \Omega \times [0,T]$ the \emph{Nemytskii operator} $u_n \mapsto F_n(\omega, t, u_n)$ on this function space by setting $[F_n(\omega, t, u_n)](\xi) \deq f_n(\omega, t, u_n(\xi))$ for all $\xi \in \cM_n$.
	
	This notion of solution is called ``global'' because it exists on the whole of $[0,T]$, in contrast with ``local'' solutions, which may blow up before time $T$.
	However, we note that global solutions generally grow unbounded as $T \to \infty$.
	We will not consider local solutions in this section, but we do work with them in Section~\ref{sec:locLip}.
	
	In this section, the real-valued functions $f_n$ are supposed to satisfy the following:
	
	\begin{assumption}[Nonlinearities]
		\label{ass:nonlinearity}
		We will assume one of the following conditions:
		\begin{enumerate}
			\item\label{ass:nonlinearity:globalLip-linGrowth}
			The nonlinearities $(f_n)_{n\in\clos\N}$ are globally Lipschitz continuous and grow linearly, both uniformly in $n$.
			I.e., there exist $\widetilde L_f, \widetilde C_f \in [0,\infty)$ such that, for all $n \in \clos \N$ and $x,y\in\R$,
			\begin{equation}
				\abs{f_n(\omega, t, x) - f_n(\omega, t, y)}
				\le 
				\widetilde L_f \abs{x - y}
				\quad \text{and} \quad 
				\abs{f_n(\omega, t, x)}
				\le 
				\widetilde C_f(1 + \abs{x}).
			\end{equation}
			\item\label{ass:nonlinearity:polynomial}
			The nonlinearities $(f_n)_{n\in\clos\N}$ are of the polynomial form
			\begin{equation}\label{eq:fn-polynomial-def}
				f_n(\omega, t, x) \deq 
				-a_{2k+1,n}(\omega, t) x^{2k+1}
				+
				\sum_{j=0}^{2k} a_{j,n}(\omega, t) x^j,
			\end{equation}
			where $k \in \N_0$ and $a_{j, n} \from \Omega \times [0,T] \to \R$ for 
			each $j \in \{0,\dots,2k+1\}$, and there exist constants $c, C \in 
			(0,\infty)$ such that 
			\begin{equation}\label{eq:fn-polynomial-coefficients}
				c \le a_{2k+1,n}(\omega,t) \le C 
				\quad \text{and} 
				\quad 
				\abs{a_{j,n}(\omega,t)} \le C
			\end{equation}
			for all $j \in \{0,\dots,2k\}$, $n \in \clos\N$ and $(\omega,t) \in \Omega \times [0,T]$.
		\end{enumerate}
		In either case, we suppose moreover that $f_n \to f$ uniformly on compact intervals; i.e., for all $r \in [0,\infty)$ and $(\omega, t) \in \Omega \times [0,T]$,
		\begin{equation}\label{eq:nonlinearities-conv-unif-compact}
			\sup_{x \in [-r,r]} 
			\abs{f_n(\omega,t,x) - f_\infty(\omega,t,x)}
			\to 
			0 \quad \text{as } n \to \infty.
		\end{equation}
	\end{assumption}
	\begin{example}\label{ex:Allen-Cahn}
        The cubic polynomial $f_n(\omega, t, x) \deq -x^3 + x$, which turns~\eqref{eq:example-SPDE} into the (fractional) stochastic Allen--Cahn equation in case $\tau \equiv 0$, $\kappa \equiv 1$ and $s \le 1$, 
        is of the form asserted in Assumption~\ref{ass:nonlinearity}\ref{ass:nonlinearity:polynomial}.
    	Note that this is also an example of the important situation where~\eqref{eq:nonlinearities-conv-unif-compact} is trivially satisfied by taking the same function $f_n \deq f$ for all $n \in \clos\N$.
	\end{example}
	
	The final technical assumption that we record before moving on to the main theorem of this section is the following:
	\begin{assumption}[Uniform $L^\infty$-boundedness of eigenfunctions]
		\label{ass:uniform-Linfty-bdd-eigenfuncs}
		There exists a constant $M_{\psi,\infty} \in (0,\infty)$ such that 
		\begin{equation}\label{eq:eigenfuncs-unif-bdd}
			\norm{\psi_n^{(j)}}{L^\infty(\cM_n)} \le M_{\psi, \infty}
			\quad 
			\text{for all } 
			n \in \clos \N \text{ and } j \in \{1,\dots,n\}.
		\end{equation}
	\end{assumption}
	
	The interplay of the various choices of spatial domains $\cM_n$, linear operators $A_n$ and nonlinearity functions $f_n$ determines the class of SPDEs to which~\eqref{eq:example-SPDE} belongs.
	Rigorous definitions of the corresponding mild solution concepts, as well as well-posedness and discrete-to-continuum convergence results can be found in Sections~\ref{sec:OU-H}--\ref{sec:reaction-diffusion}, respectively.
	Applying these results in their respective regimes of applicability yields the following discrete-to-continuum convergence theorem for the solutions to~\eqref{eq:example-SPDE};
    note that the setting of part~\ref{thm:d2c-conv-example-SPDEs:polynomial} covers the stochastic (fractional) Allen--Cahn equation on the one-dimensional torus, see Example~\ref{ex:Allen-Cahn}.
	
	\begin{theorem}\label{thm:d2c-conv-example-SPDEs}
		Let $\cM$ and $(\cM_n)_{n\in\clos\N}$ be as in Setting~\ref{ex:square-grid} or~\ref{ex:sampling-manifold}.
		\begin{enumerate}[(a)]
			\item \label{thm:d2c-conv-example-SPDEs:L2}
			Consider Setting~\ref{ex:sampling-manifold}.
			Let $s > \frac{1}{2}m$ and suppose that Assumption~\ref{ass:connectivity-regimes}\ref{ass:connectivity-regimes:optimal-range} holds with $\beta \in (\frac{m}{4s}, \frac{1}{2})$.
			If Assumptions~\ref{ass:kappa-tau}\ref{ass:kappa-tau:general} and~\ref{ass:nonlinearity}\ref{ass:nonlinearity:globalLip-linGrowth} are satisfied, and $p \in [1,\infty)$ is such that $\Lambda_n \xi_n \to \xi_\infty$ in $L^p(\Omega, \cF_0, \bbP; L^2(\cM))$, then there exists a unique global mild solution $u_n$ in $L^p(\Omega; C([0,T]; L^2(\cM_n)))$ to~\eqref{eq:example-SPDE} for every $n \in \clos\N$, and as $n \to \infty$ we have 
			\begin{equation}
				\Lambda_n u_n \to u_\infty
				\quad \widetilde\bbP\text{-a.s.\ in } 
				L^p(\Omega; C([0,T]; L^2(\cM))).
			\end{equation}
			\item \label{thm:d2c-conv-example-SPDEs:Linfty}
			In Setting~\ref{ex:sampling-manifold},
			let $\delta > 0$ be such that $\frac{m}{1-\delta} < m + 4 + \delta$, suppose $s > m + 4 + \delta$ and Assumption~\ref{ass:connectivity-regimes}\ref{ass:connectivity-regimes:suboptimal-range} holds with $\beta \in (\frac{m+4+\delta}{2s}, \frac{1}{2})$.
			Let Assumptions~\ref{ass:kappa-tau}\ref{ass:kappa-tau:Laplacian}, \ref{ass:uniform-Linfty-bdd-semigroup}\ref{ass:uniform-Linfty-bdd-semigroup:bdd}, \ref{ass:nonlinearity}\ref{ass:nonlinearity:globalLip-linGrowth} and~\ref{ass:uniform-Linfty-bdd-eigenfuncs} be satisfied.
			If $p \in [1,\infty)$ is such that $\Lambda_n \xi_n \to \xi_\infty$ in $L^p(\Omega, \cF_0, \bbP; L^\infty(\cM))$, then there exists a unique global mild solution $u_n$ in $L^p(\Omega; C([0,T]; L^\infty(\cM_n)))$ to~\eqref{eq:example-SPDE} for every $n \in \N$, as well as $u_\infty$ in $L^p(\Omega; C([0,T]; C(\cM)))$ for $n = \infty$, and as $n \to \infty$ we have
			\begin{equation}
				\Lambda_n u_n \to u_\infty
				\quad \text{in } 
				L^0(\widetilde\Omega, L^p(\Omega; C([0,T]; L^\infty(\cM)))).
			\end{equation}
			\item \label{thm:d2c-conv-example-SPDEs:polynomial} 
			Consider Setting~\ref{ex:square-grid} with $\cM \deq \bbT$. Let $s \in (\frac{1}{2}, 1]$ and suppose that Assumptions~\ref{ass:kappa-tau}\ref{ass:kappa-tau:Laplacian} and~\ref{ass:nonlinearity}\ref{ass:nonlinearity:polynomial} are satisfied.
			If $p \in (1,\infty)$ is such that $\Lambda_n \xi_n \to \xi_\infty$ in $L^p(\Omega, \cF_0, \bbP; L^\infty(\bbT))$, then 
			there exists a unique global mild solution $u_n$ in $L^p(\Omega; C([0,T]; L^\infty(\cM_n)))$ to~\eqref{eq:example-SPDE} for every $n \in \N$, as well as $u_\infty$ in $L^p(\Omega; C([0,T]; C(\bbT)))$ for $n = \infty$
			and for all $p^- \in [1,p)$ we have, as $n \to \infty$,
			\begin{equation}
				\Lambda_n u_n \to u_\infty
				\quad 
				\text{in } L^{p^-}\!(\Omega; C([0,T]; L^\infty(\bbT))).
			\end{equation}
		\end{enumerate}
	\end{theorem}

	The proof is presented in Section~\ref{sec:example:proof-of-conv}. 
    In the next section, we list the intermediate results on which it relies.
    The motivations behind the various assumptions listed above, as well as their role in Theorem~\ref{thm:d2c-conv-example-SPDEs}, are discussed in Section~\ref{sec:example:discussion-assumptions}.
	
	\subsection{Intermediate results}
	\label{sec:intermediate-results}
	
	In this subsection, we collect a number of intermediate results which imply that the conditions imposed in Theorem~\ref{thm:d2c-conv-example-SPDEs} are sufficient to fit into the setting of the various convergence theorems in Sections~\ref{sec:OU-H}--\ref{sec:reaction-diffusion}.
	More precisely, depending on the setting, we wish to verify a subset of the following: Conditions~\ref{ass:disc}--\ref{ass:A-conv} from Section~\ref{sec:OU-H} on the linear operators, conditions~\ref{item:F-globalLip-linearGrowth}--\ref{item:F-approx} and \ref{ass:IC} from Section~\ref{sec:semilinear-L2} on the nonlinearities and initial conditions, respectively, as well as their extended counterparts~\ref{ass:disc2}--\ref{ass:emb-Htheta-B-H2}, \ref{ass:IC-B}, \ref{item:F-globalLip-linearGrowth-B}--\ref{item:F-approx-B} and~\ref{item:F-diss-B} from Section~\ref{sec:reaction-diffusion}.
	The proofs of the results in this section are deferred to Appendix~\ref{app:proofs-of-intermediate-results} for ease of exposition.
	
	The necessary convergence of the linear operators, given by~\ref{ass:A-conv} and~\ref{ass:A-conv-B}, will ultimately be derived from the spectral convergence of $(\cL_n^{\tau,\kappa})_{n\in\N}$ to $\cL_\infty^{\tau,\kappa}$, i.e., the convergence of the respective eigenvalues and (lifted) eigenfunctions. 
	In the square grid Setting~\ref{ex:square-grid}, we can argue directly using closed-form expressions of all the eigenvalues and eigenfunctions involved, see Lemma~\ref{lem:spectr-conv-squaregrid} below.
	A subtlety arising in the random graph Setting~\ref{ex:sampling-manifold} is that, for any $n \in \N$, we cannot in general control the errors $\abs{\lambda_n^{(j)} - \lambda_\infty^{(j)}}$ and $\norm{\psi_n^{(j)} \circ T_n - \psi_\infty^{(j)}}{L^q(\cM)}$ for all $j \in \{1,\dots,n\}$, but only for indices $j$ up to a sufficiently small integer $k_n$. We present the precise statements below:
	Theorems~\ref{thm:eigenval-conv-randomgraph} and~\ref{thm:eigenfunc-conv-randomgraph}\ref{thm:eigenfunc-conv-randomgraph:L2}, which cover eigenvalue convergence and $L^2(\cM)$-convergence of eigenfunctions, are respectively taken from~\cite[Theorems~4.6 and~4.7]{SanzAlonsoYang2022}. 
	Theorem~\ref{thm:eigenfunc-conv-randomgraph}\ref{thm:eigenfunc-conv-randomgraph:Linfty}, concerning the $L^\infty(\cM)$-convergence of Laplacian eigenvalues, is a consequence of the main results from~\cite{CalderGTLewicka2022}, as shown in~\cite[Lemma~15]{SanzAlonsoYang2022JMLR} and the discussion preceding it.
	\begin{lemma}[{Spectral convergence---square grid}]\label{lem:spectr-conv-squaregrid}
		Let $\cM = \bbT^m$ be discretized by the sequence of square grids described in Setting~\ref{ex:square-grid}.
		If $\tau \equiv 0$ and $\kappa \equiv 1$, then for all $n \in \N$ such that $n^{1/m} \in \N$, the eigenfunction--eigenvalue pairs $(\psi^{(j)}_n, \lambda^{(j)}_n)_{j = 1}^{n}$ and $(\psi^{(j)}_\infty, \lambda^{(j)}_\infty)_{j \in \N}$ corresponding to the graph Laplacian $\cL_n^{\tau,\kappa} = \Delta_n$ and the Laplace--Beltrami operator $\cL_\infty^{\tau,\kappa} = -\Delta_{\cM}$, respectively, satisfy
		\begin{align}
			\label{eq:eigenval-conv-squaregrid}
			0 \le \lambda_\infty^{(j)} - \lambda_n^{(j)} &\le \tfrac{1}{12} j^4 \pi^4 n^{-\frac{2}{m}}
			\quad \text{for all } j \in \{1, \dots, n\};
			\\
			\label{eq:eigenfunc-conv-squaregrid}
			\norm{\psi_\infty^{(j)} - \psi_n^{(j)} \circ T_n}{L^\infty(\cM)}
			&\le 
			\tfrac{1}{2}\sqrt{2} j \pi n^{-\frac{1}{m}}
			\quad \text{for all } j \in \{1,\dots,n-1\}.
		\end{align}
	\end{lemma}

	\begin{theorem}[{Eigenvalue convergence---random graphs}]\label{thm:eigenval-conv-randomgraph}
		Let the manifold $\cM$ and the random point clouds $(\cM_n)_{n\in\N} \subseteq \cM$ on the probability space $(\widetilde \Omega, \widetilde \cF, \widetilde \bbP)$ be as in Setting~\ref{ex:sampling-manifold}.
		Suppose that $\tau \from \cM \to [0,\infty)$ is Lipschitz, and that $\kappa \from \cM \to [0,\infty)$ is continuously differentiable and bounded below away from zero.
		
		If the graph connectivity length scales $(h_n)_{n\in\N}$ (see~\eqref{eq:graph-weights}) are chosen in such a way that there exist positive integers $(k_n)_{n\in\N}$ satisfying
		\begin{equation}\label{eq:graph-connectivity-range}
			\varepsilon_n \ll h_n \ll [\lambda_\infty^{(k_n)}]^{-\frac{1}{2}},
		\end{equation}
		then there exists a constant $C_{(\cM, \tau, \kappa)} > 0$ such that
		\begin{equation}
			\widetilde \bbP\biggl(
			\frac{\abs{\lambda_n^{(j)} - \lambda_\infty^{(j)}}}{\lambda^{(j)}_\infty + 1}
			\le 
			C_{(\cM,\tau,\kappa)} \; \varepsilon_n h_n^{-1} + h_n [\lambda_\infty^{(j)}]^{\frac{1}{2}}
			\text{ for all } n \in \N, \; j \in \{1,\dots,k_n\}
			\biggr) = 1.
		\end{equation}
	\end{theorem}
	
	\begin{theorem}[{Eigenfunction convergence---random graphs}]\label{thm:eigenfunc-conv-randomgraph}
		Let the manifold $\cM$ and the random point clouds $(\cM_n)_{n\in\N} \subseteq \cM$ on the probability space $(\widetilde \Omega, \widetilde \cF, \widetilde \bbP)$ be as in Setting~\ref{ex:sampling-manifold}.
		Let $(h_n)_{n\in\N} \subseteq (0,\infty)$ be the connectivity length scales of the graphs associated to $(\cM_n)_{n\in\N}$ via the weights~\eqref{eq:graph-weights}, and consider the (graph-discretized) differential operators $\cL_n^{\tau,\kappa}$ with coefficients $\tau \from \cM \to [0,\infty)$ and $\kappa \from \cM \to [0,\infty)$.
		\begin{enumerate}[(a)]
			\item\label{thm:eigenfunc-conv-randomgraph:L2} 
			If Assumption~\ref{ass:kappa-tau}\ref{ass:kappa-tau:general} holds, and there exist integers $(k_n)_{n\in\N}$ such that~\eqref{eq:graph-connectivity-range} is satisfied, then there exists a constant $C_{(\cM, \tau, \kappa)} > 0$ such that, for all $n \in \N$,
			\begin{equation}
				\begin{aligned}
					\widetilde \bbP\Bigl(
					\norm{\psi_n^{(j)} \circ T_n - \psi_\infty^{(j)}}{L^2(\cM)}
					\le 
					C_{(\cM,\tau,\kappa)} \, 
					&j^{\frac{3}{2}}
					\bigl(
					\varepsilon_n h_n^{-1} + h_n [\lambda_\infty^{(j)}]^{\frac{1}{2}}
					\bigr)^{\frac{1}{2}}
					\\&\text{ for all } j \in \{1,\dots,k_n\}
					\Bigr) = 1.
				\end{aligned}
			\end{equation}
			\item \label{thm:eigenfunc-conv-randomgraph:Linfty} 
			Let Assumption~\ref{ass:kappa-tau}\ref{ass:kappa-tau:Laplacian} be satisfied.
			If there exist $(k_n)_{n\in\N}\subseteq\N$ and $\delta > 0$ such that
			\begin{equation}
				n^{-\frac{1}{m+4+\delta}}
				\lesssim_{\cM} 
				h_n 
				\lesssim_{\cM} 
				[\lambda_\infty^{(k_n)}]^{-1}
				\quad \text{and} \quad 
				\lambda_\infty^{(k_n)} \lesssim_\cM n^{\frac{1-\delta}{m}},
			\end{equation}
			then there exists a constant $C_{\cM} > 0$ such that, as $n \to \infty$,
			\begin{equation}
				\begin{aligned}
					\widetilde \bbP\Bigl(
					\norm{\psi_n^{(j)} \circ T_n - \psi_\infty^{(j)}}{L^\infty(\cM)}
					\le 
					C_{\cM} [\lambda_\infty^{(j)}]^{m+1} j^{\frac{3}{2}} &\bigl(\varepsilon_n h_n^{-1} + h_n [\lambda_\infty^{(j)}]^{\frac{1}{2}}\bigr)^{\frac{1}{2}}
					\\&\text{ for all } j \in \{1,\dots,k_n\}
					\Bigr) \to 1.
				\end{aligned}
			\end{equation}
		\end{enumerate}
	\end{theorem}
	
	From the above results, we can derive the following convergence of the sequence $(A_n)_{n\in\clos\N}$. 
	Its proof is analogous to that of~\cite[Theorem~4.2]{SanzAlonsoYang2022}, see Appendix~\ref{app:proofs-of-intermediate-results}.
	
	\begin{theorem}\label{thm:convergence-An}
		Given $\tau \from \cM \to [0,\infty)$, $\kappa \from \cM \to [0,\infty)$ and $s \in [0,\infty)$, consider the generalized Whittle--Mat\'ern operators $A_n \deq (\cL_n^{\kappa, \tau})^s$ defined by~\eqref{eq:def-WM-operator}, and set $\widetilde R_n^\alpha \deq \Lambda_n (\id_n + A_n)^{-\alpha} \Pi_n$, for all $\alpha \in [0,\infty)$ and $n \in \clos\N$. 
		\begin{enumerate}[(a)]
			\item \label{thm:convergence-An:a}
			Suppose that $\cM$ and its discretizations $(\cM_n)_{n\in\N}$ are as in Setting~\ref{ex:sampling-manifold}, Assumption~\ref{ass:connectivity-regimes}\ref{ass:connectivity-regimes:optimal-range} holds with $\beta \in (\frac{m}{4s}, \infty)$ and Assumption~\ref{ass:kappa-tau}\ref{ass:kappa-tau:general} is satisfied.
			Then we have for all $\beta' \in [\beta,\infty)$:
			\begin{equation}
				\widetilde R_n^{\beta'}
				\to 
				R_\infty^{\beta'}
				\quad 
				\widetilde \bbP\text{-a.s.} \text{ in } \LO_2(L^2(\cM))
				\quad
				\text{as } n \to \infty.
			\end{equation}
			\item\label{thm:convergence-An:b} 
			Suppose that $\cM$ and its discretizations $(\cM_n)_{n\in\N}$ are as in Setting~\ref{ex:sampling-manifold}, Assumption~\ref{ass:connectivity-regimes}\ref{ass:connectivity-regimes:suboptimal-range} holds with $\beta \in (\frac{m+4+\delta}{2s}, \infty)$, and Assumptions~\ref{ass:kappa-tau}\ref{ass:kappa-tau:Laplacian} and~\ref{ass:uniform-Linfty-bdd-eigenfuncs} are satisfied.
			Then we have for all $\beta' \in [\beta,\infty)$:
			\begin{equation}
				\widetilde R_n^{\beta'}
				\to 
				R_\infty^{\beta'}
				\quad 
				\text{in } L^0(\widetilde \Omega; \LO(L^2(\cM); L^\infty(\cM)))
				\quad 
				\text{as } n \to \infty.
			\end{equation}
			Here, $L^0(\widetilde\Omega)$ denotes convergence in probability with respect to $(\widetilde \Omega, \widetilde{\cF}, \widetilde{\bbP})$.
			\item\label{thm:convergence-An:c}  Suppose that $\cM \deq \bbT^m$ is discretized using the square grids $(\cM_n)_{n\in\N}$ from Setting~\ref{ex:square-grid}, and that Assumption~\ref{ass:kappa-tau}\ref{ass:kappa-tau:Laplacian} holds.
			For all $\beta \in (\frac{m}{4s}, \infty)$,
			\begin{equation}
				\widetilde R_n^\beta
				\to 
				R_\infty^\beta
				\quad 
				\text{in } \LO(L^2(\cM); L^\infty(\cM))
				\quad 
				\text{as } n \to \infty.
			\end{equation}
		\end{enumerate}
	\end{theorem}
	The following property, which we call the \emph{uniform ultracontractivity} of the semigroups $(S_n)_{n\in\clos\N}$, will be needed in order to obtain the $L^\infty(\cM)$-convergence in Theorem~\ref{thm:d2c-conv-example-SPDEs}\ref{thm:d2c-conv-example-SPDEs:Linfty} and~\ref{thm:d2c-conv-example-SPDEs:polynomial}.
	Its proof relies on Riesz--Thorin interpolation, Assumption~\ref{ass:uniform-Linfty-bdd-semigroup}, and some arguments from Theorem~\ref{thm:convergence-An}.
	\begin{lemma}[Uniform ultracontractivity]\label{lem:uniform-ultracontractivity}
		Let $s \in (0,\infty)$ and consider the generalized Whittle--Mat\'ern operators $A_n \deq (\cL_n^{\kappa, \tau})^s$ defined by~\eqref{eq:def-WM-operator} for all $n \in \clos\N$.
		Assume either of the following statements:
		\begin{enumerate}[(a)]
			\item\label{lem:uniform-ultracontractivity:a} 
			In Setting~\ref{ex:sampling-manifold}, Assumption~\ref{ass:connectivity-regimes}\ref{ass:connectivity-regimes:optimal-range} or~\ref{ass:connectivity-regimes:suboptimal-range} holds with corresponding $\beta$, as well as Assumptions~\ref{ass:kappa-tau}\ref{ass:kappa-tau:general}, \ref{ass:uniform-Linfty-bdd-semigroup}\ref{ass:uniform-Linfty-bdd-semigroup:bdd} and~\ref{ass:uniform-Linfty-bdd-eigenfuncs}.
			\item\label{lem:uniform-ultracontractivity:b} 
			In Setting~\ref{ex:square-grid},  $\beta \in (\frac{m}{4s}, \infty)$ is arbitrary, and Assumptions~\ref{ass:kappa-tau}\ref{ass:kappa-tau:Laplacian} and~\ref{ass:uniform-Linfty-bdd-semigroup}\ref{ass:uniform-Linfty-bdd-semigroup:bdd} hold.
		\end{enumerate}
		Then, for every $q \in [2,\infty]$, there exists $M_{S,q} \in [1,\infty)$ such that
		\begin{equation}\label{eq:ultracontractivity}
			\norm{S_n(t)}{\LO(L^q(\cM_n); L^\infty(\cM_n))}
			\le 
			M_{S, q} t^{-\frac{2}{q}\beta}
			\quad 
			\text{for all } n \in \clos\N \text{ and } t > 0.
		\end{equation}
		In case of~\ref{lem:uniform-ultracontractivity:a}, \eqref{eq:ultracontractivity} holds $\widetilde{\bbP}$-a.s.
	\end{lemma}
	%
	
	\subsection{Proof of convergence}
	\label{sec:example:proof-of-conv}
	
	Using the intermediate results from Subsection~\ref{sec:intermediate-results}, we can now prove Theorem~\ref{thm:d2c-conv-example-SPDEs}:
	
	\begin{proof}[Proof of Theorem~\ref{thm:d2c-conv-example-SPDEs}]
		To prove parts~\ref{thm:d2c-conv-example-SPDEs:L2}--\ref{thm:d2c-conv-example-SPDEs:polynomial}, we will apply Theorems~\ref{thm:d2c-semilinear-global-linear},~\ref{thm:d2c-semilinear-global-linear-B} and Corollary~\ref{cor:d2c-conv-dissB}, respectively, which are the rigorous counterparts of the discrete-to-continuum Theorem~\ref{thm:d2c-conv-informal} in the respective settings.
		
		The argument preceding Setting~\ref{ex:square-grid} shows that~\ref{ass:disc} and~\ref{ass:disc2} hold in any of the given situations, with $H_n \deq L^2(\cM_n)$ and $E_n \deq L^q(\cM_n)$ for $n \in \clos \N$ and $q \in [2,\infty)$, as well as $B_n \deq L^\infty(\cM_n)$ for all $n \in \N$, $B_\infty \deq C(\cM)$ and $\widetilde B \deq L^\infty(\cM)$.
		Moreover, note that~\ref{ass:IC} (or~\ref{ass:IC-B}) is explicitly assumed in each case.
		
		\ref{thm:d2c-conv-example-SPDEs:L2}
		Here, we take $q = 2$, i.e., $E_n = H_n = L^2(\cM_n)$ for all $n \in \clos \N$.
		As discussed at the end of Subsection~\ref{sec:geo-graphs-WMoperators},
		the operators $(A_n)_{n\in\clos\N} \deq ([\cL_n^{\tau,\kappa}]^s)_{n\in\clos\N}$ are uniformly sectorial of angle $0$ on $(L^2(\cM_n))_{n\in\clos\N}$.
		Letting $\beta \in (\frac{m}{4s}, \frac{1}{2})$ be as in Assumption~\ref{ass:connectivity-regimes}\ref{ass:connectivity-regimes:optimal-range}, it follows from Theorem~\ref{thm:convergence-An}\ref{thm:convergence-An:a} that $\widetilde R_n^{\beta'} \to R_\infty^{\beta'}$, $\widetilde{\bbP}$-a.s., in $\LO_2(L^2(\cM))$ as $n \to \infty$, for all $\beta' \ge \beta$.
		Applying this with $\beta' \deq \beta \in (0,\frac{1}{2})$ and $\beta' \deq 1$ yields~\ref{ass:operators} and~\ref{ass:A-conv}.
		Setting, for all $(\omega,t) \in \Omega\times[0,T]$, $u \in L^2(\cM_n)$ and $x \in \cM_n$,
		\begin{equation}\label{eq:Fn-equals-fn}
			[F_n(\omega, t, u)](x) \deq f_n(\omega, t, u(x)),
		\end{equation}
		it is immediate from Assumption~\ref{ass:nonlinearity}\ref{ass:nonlinearity:globalLip-linGrowth} that~\ref{item:F-globalLip-linearGrowth} is satisfied.
		Moreover, combining the definition of $\widetilde F_n$ from~\eqref{eq:def-Fntilde} with~\eqref{eq:Fn-equals-fn} yields 
		\begin{equation}
			[\widetilde F_n(\omega, t, u)](x) = [\Lambda_n F_n(\omega, t, \Pi_n u)](x) = f_n(t,\omega, \Lambda_n \Pi_n u(x)),
		\end{equation}
		so that
		\begin{equation}
			\begin{aligned}
				&\norm{\widetilde F_n(\omega, t, u) - F_\infty(\omega, t, u)}{L^2(\cM)}
				=
				\norm{f_n(\omega, t, \Lambda_n \Pi_n u(\,\cdot\,)) - f_\infty(\omega, t, u(\,\cdot\,))}{L^2(\cM)}
				\\&\qquad 
				\le 
				\widetilde L_f
				\norm{\Lambda_n \Pi_n u - u}{L^2(\cM)}
				+ 
				\norm{f_n(\omega, t, u(\,\cdot\,)) - f_\infty(\omega, t, u(\,\cdot\,))}{L^2(\cM)}.
			\end{aligned}
		\end{equation}
		As $n \to \infty$, the first term vanishes by Assumption~\ref{ass:disc}, and the second term by dominated convergence using~\eqref{eq:nonlinearities-conv-unif-compact} and the uniform linear growth condition in Assumption~\ref{ass:nonlinearity}\ref{ass:nonlinearity:globalLip-linGrowth}.
		Therefore, condition~\ref{item:F-approx} is also satisfied.
		
		\ref{thm:d2c-conv-example-SPDEs:Linfty}
		Now we need Assumption~\ref{ass:uniform-Linfty-bdd-semigroup}\ref{ass:uniform-Linfty-bdd-semigroup:bdd} in order for $([\cL_n^{\tau,\kappa}]^s)_{n\in\clos\N}$ to be uniformly sectorial of angle less than $\frac{1}{2}\pi$ on $(L^q(\cM_n))_{n\in\clos\N}$ for all $q \in [2,\infty)$.
		Letting $\delta > 0$ and $\beta \in (\frac{m + 4 + \delta}{2s}, \frac{1}{2})$ be as in Assumption~\ref{ass:connectivity-regimes}\ref{ass:connectivity-regimes:suboptimal-range}, it follows from Theorem~\ref{thm:convergence-An}\ref{thm:convergence-An:b} that $\widetilde R_n^{\beta'} \to R_\infty^{\beta'}$ in $L^0(\widetilde \Omega; \LO(L^2(\cM); L^\infty(\cM)))$ as $n \to \infty$, for all $\beta' \ge \beta$, under Assumptions~\ref{ass:kappa-tau}\ref{ass:kappa-tau:Laplacian} and~\ref{ass:uniform-Linfty-bdd-eigenfuncs}.
		In particular, we have $\widetilde R_n^{\beta} \to R_\infty^{\beta}$ in $L^0(\widetilde \Omega; \gamma(L^2(\cM); L^q(\cM)))$ for all $q \in [1,\infty)$ by~\cite[Corollary~9.3.3]{HvNVWVolumeII}, and $\widetilde R_n \to R_\infty$ in $L^0(\widetilde \Omega; \LO(L^\infty(\cM)))$.
		This shows~\ref{ass:semigroup-B2} and~\ref{ass:A-conv-B}.
		By Lemma~\ref{lem:uniform-ultracontractivity}\ref{lem:uniform-ultracontractivity:a}, we have~\ref{ass:emb-Htheta-B-H2} with $\theta = \frac{4}{q}\beta$.
		Thus, choosing $q > \frac{4\beta}{1-2\beta}$ yields $\theta + 2\beta < 1$.
		Conditions~\ref{item:F-globalLip-linearGrowth-B} and~\ref{item:F-approx-B} follow similarly to part~\ref{thm:d2c-TKapprox:a}.
		
		\ref{thm:d2c-conv-example-SPDEs:polynomial}
		As in part~\ref{thm:d2c-conv-example-SPDEs:Linfty}, we need to verify conditions~\ref{ass:disc2}--\ref{ass:emb-Htheta-B-H2}, now with contractive semigroups $(S_n(t))_{t\ge0}$, i.e., Assumption~\ref{ass:uniform-Linfty-bdd-semigroup}\ref{ass:uniform-Linfty-bdd-semigroup:contractive}. 
		For $s = 1$, $(S_\infty(t))_{t\ge0}$ is $L^\infty(\bbT)$-contractive since the $L^1(\bbT)$-norm of its heat kernel coincides with the $L^1(\R)$-norm of the Gauss--Weierstrass kernel, which is equal to $1$.
		For finite $n$, the $L^\infty(\cM_n)$-contractivity of $S_n(t) = e^{-tA_n}$ is equivalent to $A_n$ being diagonally dominant with positive diagonal by~\cite[Lemma~6.1]{Mugnolo2006}, which holds for Laplacian matrices.
		Since these assertions can be extended to all $s \in (0,1]$ by subordination, see for instance~\cite[Theorem~15.2.17]{HvNVWVolumeIII}, we indeed find that Assumption~\ref{ass:uniform-Linfty-bdd-semigroup}\ref{ass:uniform-Linfty-bdd-semigroup:contractive} holds.
		Thus, we can proceed to argue as in~\ref{thm:d2c-conv-example-SPDEs:Linfty}, using Theorem~\ref{thm:convergence-An}\ref{thm:convergence-An:c} and Lemma~\ref{lem:uniform-ultracontractivity}\ref{lem:uniform-ultracontractivity:b} for an arbitrary $\beta \in (\frac{1}{4s}, \frac{1}{2})$, to obtain~\ref{ass:disc2}--\ref{ass:emb-Htheta-B-H2} with $\widetilde M_S = 1$ and $\theta + 2\beta < 1$ for $q \in [2,\infty)$ large enough.
		
		It remains to establish that the nonlinearities from Assumption~\ref{ass:nonlinearity}\ref{ass:nonlinearity:polynomial} are such that~\ref{item:F-diss-B} holds.
		This is done in Example~\ref{ex:polynomial-nonlinearity-satisfies-FdissB}, noting that the space $L^\infty(\cM_n)$ coincides with $C(\cM_n)$ if we equip $\cM_n$ with the discrete topology.
	\end{proof}
	
	\subsection{Discussion of the assumptions}
	\label{sec:example:discussion-assumptions}
	
	In this subsection, we comment on the various assumptions made in Theorem~\ref{thm:d2c-conv-example-SPDEs}, the extent to which they are necessary, and how one might check them in practice. \medskip
	
	\noindent The distinction between parts~\ref{ass:kappa-tau:general} and~\ref{ass:kappa-tau:Laplacian} of Assumption~\ref{ass:kappa-tau}, i.e., whether to allow for spatially varying coefficient functions $\tau$ and $\kappa$ in the second-order symmetric base operators $(\cL_n^{\tau,\kappa})_{n\in\clos\N}$ instead of merely considering Laplacians, is mainly due to the availability of spectral convergence theorems in the respective situations.
	Most of the literature on eigenfunction convergence of graph-discretized second-order operators is focused on the Laplacian case, see for instance~\cite{CGT2022, GTGHS2020} for $L^2$-convergence and~\cite{CalderGTLewicka2022, DWW2021, WR2021} for $L^\infty$-convergence.
	However, the authors of~\cite{SanzAlonsoYang2022} show how the $L^2$-convergence results can be extended to coefficient functions satisfying Assumption~\ref{ass:kappa-tau}\ref{ass:kappa-tau:general}.
	We expect that most spectral convergence results for graph Laplacians can be extended to allow for varying coefficients, but doing so requires significant effort, hence we sometimes make Assumption~\ref{ass:kappa-tau}\ref{ass:kappa-tau:Laplacian} for the sake of convenience.
	
	Similarly, the difference between the bounds on the graph connectivity length scales in the two parts of Assumption~\ref{ass:connectivity-regimes} is a result of the current availability of spectral convergence literature.
	Eigenfunction convergence of $(\cL_n^{\tau,\kappa})_{n\in\clos\N}$ in $L^2$ has for instance been proved in~\cite{SanzAlonsoYang2022} under Assumption~\ref{ass:connectivity-regimes}\ref{ass:connectivity-regimes:optimal-range}, but for graphs and manifolds as in our setting, the optimal available $L^\infty$-convergence results (for graph Laplacians) seem to be those of~\cite{CalderGTLewicka2022}, which require Assumption~\ref{ass:connectivity-regimes}\ref{ass:connectivity-regimes:suboptimal-range}.
	However, according to~\cite[Remark~2.7]{CalderGTLewicka2022}, it is plausible that the $L^\infty(\cM)$-convergence of Laplacian eigenfunctions can be proved under the same assumptions as the $L^2(\cM)$-convergence, with the same rate. 
	Some recent results in this direction can be found in~\cite{ArmstrongVenkatraman2023}, where the authors show $L^\infty$-convergence of Laplacian eigenvectors with optimal rates and loose lower bounds on the connectivity lengths, using homogenization theory, for point clouds on less general spatial domains. \medskip
	
	\noindent Assumption~\ref{ass:uniform-Linfty-bdd-semigroup} is natural in the sense that the results regarding $L^\infty$-convergence in space (for instance Theorem~\ref{thm:d2c-conv-example-SPDEs}\ref{thm:d2c-conv-example-SPDEs:Linfty} and~\ref{thm:d2c-conv-example-SPDEs:polynomial}) rely on uniform $L^\infty$-convergence of semigroup orbits on compact time intervals. 
	The latter necessitates that Assumption~\ref{ass:uniform-Linfty-bdd-semigroup}\ref{ass:uniform-Linfty-bdd-semigroup:bdd} is satisfied, at least for $t \in [0,T]$ with arbitrarily large $T \in (0,\infty)$.
	
	Moreover, for typical choices of differential operators, one can often check that Assumption~\ref{ass:uniform-Linfty-bdd-semigroup}\ref{ass:uniform-Linfty-bdd-semigroup:contractive} holds, meaning that the semigroups are in fact $L^\infty$-contractive.
	One such example is outlined in the proof of Theorem~\ref{thm:d2c-conv-example-SPDEs}\ref{thm:d2c-conv-example-SPDEs:polynomial}: 
	Matrix exponentials $(e^{-tL_n})_{t\ge0}$ are $L^\infty$-contractive if and only if $L_n \in \R^{n\times n}$ is diagonally dominant with nonnegative diagonal entries~\cite[Lemma~6.1]{Mugnolo2006}. 
	Sufficient conditions for the $L^\infty$-contractivity of the semigroup $(S_\infty(t))_{t\ge0}$ generated by the negative of a uniformly elliptic second-order differential operator on a Euclidean domain $\cD\subsetneq \R^d$, subject to appropriate boundary conditions, can be found in~\cite[Section~4.3]{Ouhabaz2005}.
	Likewise, the heat semigroup associated to the Laplace--Beltrami operator on a compact Riemannian manifold $\cM$ is $L^\infty$-contractive, cf.~\cite[p.~148]{Davies1989}.
	
	As mentioned in the proof of Theorem~\ref{thm:d2c-conv-example-SPDEs}\ref{thm:d2c-conv-example-SPDEs:polynomial}, all of the above $L^\infty$-contractivity results for second-order differential operators can be extended to fractional powers $s \in (0,1)$ by using a subordination formula such as~\cite[Theorem~15.2.17]{HvNVWVolumeIII}, noting that the definition of fractional power operators in this reference coincides with ours (more details are given in the first half of the proof of Lemma~\ref{lem:strong-conv-AnalphaS} below).
	The semigroups generated by higher-order differential operators, however, are in general not contractive on $L^\infty$ (or any $L^q$ for $q \neq 2$, see for instance~\cite{LangerMazya1999}); this is closely related to their lack of positivity preservation.
	As an example, the fractional heat kernel associated to $(-\Delta)^s$ on $\R^d$ with $s \in (0,\infty)$ at time $t \in (0,\infty)$ is given by the inverse Fourier transform of $\xi \mapsto \exp(-t \norm{\xi}{\R^d}^{2s})$, which is positive for $s \le 1$ but fails to be sign-definite if $s > 1$, see~\cite[p.~626 and pp.~632--633]{EOR1991}, respectively.%
	
	Thus, for operators $([\cL_n^{\tau,\kappa}]^s)_{n\in\clos\N}$ with $s > 1$, we have to content ourselves with uniform $L^\infty$-boundedness of the semigroups $(S_n(t))_{t\ge0}$ in $n$ and $t$, i.e., Assumption~\ref{ass:uniform-Linfty-bdd-semigroup}\ref{ass:uniform-Linfty-bdd-semigroup:bdd}.
	In the absence of positivity preservation, one route to verifying such uniformity is through Gaussian upper bounds on the integral kernels corresponding to the semigroups, cf.~\cite[Proposition~7.1]{Ouhabaz2005}.
	Such bounds have been established for higher-order differential operators on Euclidean domains, as well as Laplacian operators on more general domains such as manifolds, graphs and fractals (see~\cite[pp.~194--196]{Ouhabaz2005} and the references therein).
	While it may be possible to unify these results in the setting of graph-discretized higher-order differential operators on manifolds, and thus obtain the uniform $L^\infty$-bounds required by Assumption~\ref{ass:uniform-Linfty-bdd-semigroup}\ref{ass:uniform-Linfty-bdd-semigroup:bdd}, this appears to be highly nontrivial and outside the scope of this work.
	
	We also remark that certain higher-order operators have been shown to exhibit \emph{(local) eventual positivity}, meaning that for every nonnegative initial datum $u_0 \ge 0$ and subset $\cD^*$ of the spatial domain $\cD$, there exists $t^* > 0$ such that $S(t)u_0 \ge 0$ on $\cD^*$ for all $t \ge t^*$. 
	For instance, in~\cite{GG2008}, this was shown for the bi-Laplacian $\Delta^2$ on $\cD = \R^d$. 
	In~\cite{GregorioMugnolo2020}, the authors apply the theory of~\cite{DGK2016a, DGK2016b} to treat the squared graph Laplacian $\Delta_n^2$, deduce that $(e^{-\Delta_n^2 t})_{t\ge0}$ is eventually $L^\infty$-contractive~\cite[Proposition~6.7]{GregorioMugnolo2020}, and note that this implies $L^\infty$-boundedness uniformly in $t \ge 0$~\cite[Remark~6.8]{GregorioMugnolo2020}.
	However, as $n \to \infty$, their upper bound $\norm{e^{-\Delta_n^2 t}}{\infty} \le \exp(\norm{\Delta_n}{\infty}^2 t^*)$ blows up in our setting. Hence, these results do not appear to be directly useful for our purpose of verifying Assumption~\ref{ass:uniform-Linfty-bdd-semigroup}\ref{ass:uniform-Linfty-bdd-semigroup:bdd}.
	\medskip
	
	\noindent If we restrict ourselves to nonlinearities of the form $[F_n(\omega, t, u)](x) \deq f_n(\omega, t, u(x))$ (see~\eqref{eq:Langevin-equation-infinite-dim-A} and~\eqref{eq:reaction-diffusion}), then the conditions in Assumption~\ref{ass:nonlinearity} are the natural ones to ensure the global (in time) convergence results formulated in Theorem~\ref{thm:d2c-conv-example-SPDEs}.
	Convergence results for more general nonlinearities, possibly formulated only in terms of local-in-time convergence and in weaker norms, can be found in Sections~\ref{sec:semilinear-L2} and~\ref{sec:reaction-diffusion}. \medskip 
	
	\noindent Assumption~\ref{ass:uniform-Linfty-bdd-eigenfuncs} was used (explicitly or implicitly) to establish the $L^\infty$-convergence asserted in Theorem~\ref{thm:convergence-An} and the uniform $L^2$--$L^\infty$-ultracontractivity in Lemma~\ref{lem:uniform-ultracontractivity}.
	The $L^\infty$-norms of the $L^2$-normalized eigenfunctions of the Laplace--Beltrami operator on a general compact Riemannian manifold $\cM$ of dimension $m$ satisfy the upper bound $\norm{\psi^{(j)}}{L^\infty(\cM)} \lesssim [\lambda^{(j)}]^{\frac{m-1}{4}}$ due to Hörmander~\cite{Hormander1968}. 
	This bound is sharp in the sense that it is attained by the symmetric spherical harmonics on the sphere.
	On the other hand, the $L^\infty$-norms are uniformly bounded (i.e., satisfy Assumption~\ref{ass:uniform-Linfty-bdd-eigenfuncs}) if $m = 1$ or if $\cM = \bbT^m$ is a flat torus. 
	Some results relating the $L^\infty$-growth rate of these eigenfunctions to the geometry of the manifold can be found in~\cite{Donnelly2006, SoggeZelditch2002, TothZelditch2002}.
	
	These observations indicate that Assumption~\ref{ass:uniform-Linfty-bdd-eigenfuncs} poses strong restrictions on the curvature of the manifold, which raises the question whether this assumption could be removed.
	Its central role in the proofs of Theorem~\ref{thm:convergence-An}\ref{thm:convergence-An:b},~\ref{thm:convergence-An:c} and Lemma~\ref{lem:uniform-ultracontractivity} is due to the $L^2$--$L^\infty$-norm bounds~\eqref{eq:L2-Linfty-estimate} of operators which are defined in terms of eigenvalue expansions, such as the fractional powers defined by~\eqref{eq:def-WM-operator}.
	This suggests that disposing of Assumption~\ref{ass:uniform-Linfty-bdd-eigenfuncs} would involve techniques which are not based on spectral representations and spectral convergence of the operators involved.
	For Lemma~\ref{lem:uniform-ultracontractivity} in particular, one indication that this should be possible is the fact that, like $L^\infty$-boundedness, the $L^2$--$L^\infty$-ultracontractivity of $(S_\infty(t))_{t\ge0}$ follows from certain upper bounds on its heat kernel~\cite[Theorem~3.2]{GHL2014}.
	For the Laplace--Beltrami operator on a compact Riemannian manifold, we indeed have such bounds by~\cite[Proposition~5.5.1 and Theorem~5.5.2]{Davies1989}, which imply~\eqref{eq:ultracontractivity} with $\beta = \frac{m}{4}$ (for $n = \infty$ and $s = 1$).
	
	\section{Infinite-dimensional Ornstein--Uhlenbeck process}
	\label{sec:OU-H}
	
	This section and the subsequent Sections~\ref{sec:semilinear-L2} and~\ref{sec:reaction-diffusion} are devoted to proving the abstract discrete-to-continuum approximation results which are applied to the Whittle--Mat\'ern graph discretization setting in Section~\ref{sec:graph-WM}.
	Thus, from this point onwards we no longer necessarily work with graphs or Whittle--Mat\'ern operators.
	Instead, we have the following abstract setting.
	
	Let the filtered probability space $(\Omega, \cF, (\cF_t)_{t\in [0,T]}, \bbP)$ be given.
	For any $n \in \clos\N$, we consider the following \emph{linear} stochastic 
	evolution equation, whose state space is a real and separable UMD-type-2 Banach 
	space $(E_n, \norm{\,\cdot\,}{E_n})$:
	\begin{equation}\label{eq:OU}
		\left\lbrace
		\begin{aligned}
			\mathrm d X_n(t) &= -A_n X_n(t) \rd t + \mathrm dW_n(t), \quad t \in 
			(0,T], 
			\\
			X_n(0) &= 0.
		\end{aligned}
		\right.
	\end{equation}
	Here, $A_n \from \dom{A_n} \subseteq E_n \to E_n$ is a linear operator and $T \in (0,\infty)$ is a time horizon.
	Moreover, we take $W_n \deq \Pi_n W_\infty$, where $(W_\infty(t))_{t\ge0}$ denotes a cylindrical Wiener process on $(\Omega, \cF, (\cF_t)_{t\in [0,T]}, \bbP)$ taking values in a separable Hilbert space $(H_\infty, \scalar{\,\cdot\,,\,\cdot\,}{H_\infty})$ and the operator $\Pi_n \in \LO(H_\infty; H_n)$ is as in assumption~\ref{ass:disc} below.
	Thus, the formal time derivative $\dot W_\infty$ of $W_\infty$ represents space--time Gaussian white noise and $(W_n)_{t\ge0}$ is an $H_n$-valued $Q_n$-cylindrical Wiener process colored in space by the covariance operator $Q_n \deq \Pi_n \Pi_n^* \in \LO(H_n)$.
	
	We impose the following uniformity 
	assumptions 
	on the spaces $(E_n)_{n\in\clos\N}$, $(H_n)_{n\in\clos\N}$ and the operators $(A_n)_{n\in\clos \N}$:
	\begin{enumerate}[label=(A{\arabic*}), series=ass-A]
		\item\label{ass:disc} 
		Assumption~\ref{ass:disc3} holds for the UMD-type-2 Banach spaces $(E_n)_{n\in\clos\N}$ and the Hilbert spaces $(H_n)_{n\in\clos\N}$, with $\widetilde E \deq E_\infty$ and $\widetilde H \deq H_\infty$, both with the same sequence $(\Lambda_n, \Pi_n)_{n\in\clos\N}$ of lifting and projection operators.
		\item\label{ass:operators} 
		The operators $(A_n)_{n\in\clos\N}$ on $(E_n)_{n\in\clos\N}$ are uniformly sectorial of angle less than $\frac{1}{2}\pi$ (see Appendix~\ref{app:uniform-sectoriality} for the definition of this concept). 
		In particular, their negatives generate bounded analytic $C_0$-semigroups $(S_n(t))_{t\ge0} \subseteq \LO(E_n)$ which satisfy Assumption~\ref{ass:operators3} with $w = 0$.
		Moreover, there exists a $\beta \in [0, \frac{1}{2})$ such that $R_n^\beta \deq (\id_n + A_n)^{-\beta} \in \gamma(H_n; E_n)$ for every $n \in \clos\N$.
	\end{enumerate}
	In general, the fractional powers of the sectorial operators $A_n$ appearing in~\ref{ass:operators} can be defined using any of the equivalent definitions in~\cite[Chapter~3]{Haase2006}.
	If, as in Section~\ref{sec:graph-WM}, the operator $A_n$ given as the restriction of an operator whose eigenvalues form an orthonormal basis on some Hilbert space containing $E_n$, then one can use the spectral definition~\eqref{eq:def-WM-operator} of fractional powers of $A_n$.
	
	The solution concept which we consider for all of the equations in this work is that of a \emph{mild solution}.
	For the linear equation~\eqref{eq:OU}, it is given by the following stochastic convolution:
	
	\begin{proposition}\label{prop:existence-stoch-WAn}
		Let $n \in \clos\N$ and $T \in (0,\infty)$.
		Under Assumptions~\ref{ass:disc}--\ref{ass:operators}, the stochastic convolution
		\begin{equation}\label{eq:mildsol}
			W_{A_n}(t) \coloneqq \int_0^t S_n(t-s) \rd W_n(s), \quad t \in [0,T],
		\end{equation}
		is a well-defined process in $C([0,T]; L^p(\Omega; E_n))$ for every $p \in [1, \infty)$.
	\end{proposition}
	\begin{proof}
		For every $p \in [1,\infty)$ and $t \in [0,T]$, we have by the It\^o inequality~\eqref{eq:Ito-inequality}:
		\begin{equation}
			\norm{W_{A_n}(t)}{L^p(\Omega; E_n)}^2
			\lesssim_{(p, E_n)}
			\int_0^t \norm{S_n(t-s)}{\gamma(H_n; E_n)}^2 \rd s
			\le
			\norm{S_n}{L^2(0,T; \gamma(H_n; E_n))}^2
		\end{equation}
		To thow that the right-hand side is bounded, we use the ideal property~\eqref{eq:ideal-property} of $\gamma(H_n; E_n)$ and the estimate~\eqref{eq:app-analytic-semigroup-est} for analytic semigroups (in conjuction with Assumptions~\ref{ass:disc} and~\ref{ass:operators}) to see that
		\begin{equation}
			\begin{aligned}
				\norm{S_n}{L^2(0,T; \gamma(H_n; E_n))}^2
				&=
				\int_0^T \norm{(\id_n + A_n)^\beta S_n(t) R_n^\beta}{\gamma(H_n; E_n)}^2 \rd t
				\\
				&\le 
				\norm{R_n^\beta}{\gamma(H_n; E_n)}^2
				\int_0^T \norm{(\id_n + A_n)^\beta S_n(t)}{\LO(E_n)}^2 \rd t
				\\&\lesssim_{\beta} 
				\norm{R_n^\beta}{\gamma(H_n; E_n)}^2 \int_0^T t^{-2\beta} \rd t
				=
				\norm{R_n^\beta}{\gamma(H_n; E_n)}^2 \frac{T^{1-2\beta}}{1-2\beta} < \infty.
			\end{aligned}
		\end{equation}
		Note that $\norm{R_n^\beta}{\gamma(H_n; E_n)}^2$ is finite by~\ref{ass:operators}.
		Next, applying the It\^o inequality~\eqref{eq:Ito-inequality} to the difference $W_{A_n}(t + h) - W_{A_n}(t)$ for small enough $h \in \R$ yields
		\begin{equation}
			\norm{W_{A_n}(t + h) - W_{A_n}(t)}{L^p(\Omega; E_n)}
			\lesssim_{(p, E_n)}
			\norm{S_n(\,\cdot\, + h) - S_n}{L^2(0,T; \gamma(H_n; E_n))}
			\to 0,
		\end{equation}
		as $h \to 0$, by the strong continuity of translation operators on the Bochner space $L^2(0,T; \gamma(H_n; E_n))$.
		This shows $W_{A_n} \in L^p(\Omega; C([0,T]; E_n))$.
	\end{proof}
	
	Two stochastic processes $(X(t))_{t\in[0,T]}$ and $(Y(t))_{t\in[0,T]}$ are said to be \emph{modifications} of each other if $\bbP(X(t) = Y(t)) = 1$ for all $t \in [0,T]$.
	
	\begin{definition}\label{def:OU-mildsol}
		An $E_n$-valued stochastic process $X_n = (X_n(t))_{t\in[0,T]}$ belonging to $C([0,T]; L^p(\Omega; E_n))$ for some $p \in [1,\infty)$ is said to be a \emph{mild solution} to~\eqref{eq:OU} if it is a modification of the process $W_{A_n}$ defined in~\eqref{eq:mildsol}.
	\end{definition}
	
	The existence and uniqueness (up to modification) of the mild solution to~\eqref{eq:OU} in $C([0,T]; L^p(\Omega; E_n))$ is then immediate from Definition~\ref{def:OU-mildsol}.
	
	As remarked in Section~\ref{sec:d2c}, we mainly have applications in mind where the problem corresponding to $n = \infty$ is interpreted as a spatiotemporal stochastic partial differential equation.
    In the present linear setting, its solution is also known as an \emph{infinite-dimensional Ornstein--Uhlenbeck process}, and solutions to~\eqref{eq:OU} for $n \in \N$ are spatially discretized approximations. 
	Therefore, it is natural to ask whether we can identify the right mode of 
	convergence of the operators $(A_n)_{n\in\N}$ to $A_\infty$ as $n \to \infty$ 
	to 
	ensure the 
	convergence of the processes
	$(W_{A_n})_{n\in\N}$ to $W_{A_\infty}$. 
	
	The answer is provided by Proposition~\ref{prop:d2c-meansquare-conv-Wdelta} 
	below, which is a stochastic counterpart of the discrete-to-continuum 
	Trotter--Kato approximation theorem for 
	strongly continuous semigroups recalled in Theorem~\ref{thm:d2c-TKapprox}. 
	In fact, with an eye towards the proof of 
	Proposition~\ref{prop:d2c-pathwise-conv-Wdelta} below, we consider a more 
	general class of auxiliary processes, see equation~\eqref{eq:def-WAndelta}.
	
	Before stating any discrete-to-continuum results, let us introduce some 
	convenient notation for this goal.
	Using the operators 
	$(\Lambda_n)_{n\in\N}$ and~$(\Pi_n)_{n\in\N}$, we can take a mapping which has 
	$E_n$ as its domain or state space, 
	and turn it into an analogous mapping from or to $E_\infty$. 
	For instance, we define the $E_\infty$-valued processes $\widetilde W_{A_n} \deq \Lambda_n W_{A_n}$, as well as the operators $\widetilde R_n^\alpha \deq \Lambda_n R_n^{\alpha} \Pi_n$ and $\widetilde S_n(t) \deq \Lambda_n S_n(t) \Pi_n$ in $\LO(E_\infty)$ for $\alpha, t \in [0,\infty)$.
	Now we can formulate our notion of the convergence $A_n \to A_\infty$ as $n \to \infty$ as follows:
	\begin{enumerate}[label=(A{\arabic*}), resume=ass-A]
		\item \label{ass:A-conv}
		For every $x \in E_\infty$, we have $\widetilde R_n^1 x \to R_\infty^1 x$ in 
		$E_\infty$ as $n \to \infty$. Moreover, given $\beta \in [0, \frac{1}{2})$ as in~\ref{ass:operators}, we have $\widetilde R_n^{\beta} \to R_\infty^\beta$ in $\gamma(H_\infty; E_\infty)$.
	\end{enumerate}
	Proposition~\ref{prop:d2c-meansquare-conv-Wdelta} states that this type of convergence of the operators is 
	sufficient to ensure 
	convergence of the solutions to the linear stochastic evolution 
	equation~\eqref{eq:OU}. 
	Its proof is based on Lemma~\ref{lem:strong-conv-AnalphaS} in Appendix~\ref{app:uniform-sectoriality}, as well as the following general approximation lemma for square-integrable functions with values in the space of $\gamma$-radonifying operators.
	It is a simpler analog to~\cite[Lemma~2.6]{KvN2011}, which was only necessary to allow for stochastic equations in UMD Banach spaces without type $2$.
	
	\begin{lemma}\label{lem:approx-L2gamma}
		Let $(E, \norm{\,\cdot\,}{E})$ and $(F, \norm{\,\cdot\,}{F})$ be Banach spaces, and let $(H, \scalar{\,\cdot\, , \,\cdot\,}{H})$ be a Hilbert space.
		Given $a, b \in \R$ with $a < b$, let $M_n \from (a,b) \to \LO(E; F)$ for all $n \in \clos\N$ and suppose that
		\begin{enumerate}
			\item\label{lem:approx-L2gamma:strong-conv}
			$M_n \otimes x \to M \otimes x$ uniformly on compact subsets of $(a,b)$ for all $x \in E$, and
			\item\label{lem:approx-L2gamma:bdd}
			$\sup_{n\in\clos\N} \sup_{t \in (a,b)} \norm{M_n(t)}{\LO(E; F)} < \infty$.
		\end{enumerate}
		For all $R \in L^2(a,b; \gamma(H; E))$ and $n \in \clos\N$, we have $M_n \otimes R \in L^2(a,b; \gamma(H; F))$ and
		\begin{equation}
			M_n R \to M R
			\quad \text{in } L^2(a,b; \gamma(H; F))
			\quad \text{as } n \to \infty.
		\end{equation}
	\end{lemma}
	\begin{proof}
		Arguing as in the proof of Proposition~\ref{prop:strong-continuity-IAntilde} and using~\ref{lem:approx-L2gamma:bdd}, it follows that we only need to prove the claim for all $R$ belonging to some dense subset $D$ of $L^2(a,b; \gamma(H; E))$.
		Note that every $R \in L^2(a,b; \gamma(H; E))$ can be approximated by a step function $\sum_{j=1}^N \mathbf 1_{(a_j, b_j')} \otimes T_j$ with $a < a_j' < b_j' < b$ and $T_j \in \gamma(H; E)$, and by definition of $\gamma(H; E)$ the latter can be approximated by finite-rank operators.
		By linearity, it thus suffices to prove the statement for $R$ of the form
		\[
		R(t) = \mathbf 1_{(a',b')}(t) \, h \otimes x, 
		\quad \text{where } a < a' < b' < b \text{ and } (h,x) \in H \times E.
		\]
		Substituting this representation, using~\eqref{eq:gamma-norm-pure-tensor} and~\ref{lem:approx-L2gamma:strong-conv}, we find as $n \to \infty$:
		\begin{equation}
			\begin{aligned}
				\norm{M_n R - MR}{L^2(a,b; \gamma(H; F))}^2
				&=
				\int_{a'}^{b'} 
				\norm{h \otimes [M_n(t)x - M(t)x]}{\gamma(H; F)}^2 \rd t
				\\
				&= 
				\norm{h}{H}^2
				\int_{a'}^{b'} \norm{M_n(t)x - M(t)x}{F}^2 \rd t
				\\&\le 
				\norm{h}{H}^2
				(b' - a')
				\sup_{t \in (a',b')} \norm{M_n(t)x - M(t)x}{F}^2
				\to 0. \qedhere
			\end{aligned}
		\end{equation}
	\end{proof}
	
	\begin{proposition}\label{prop:d2c-meansquare-conv-Wdelta}
		Suppose that Assumptions~\ref{ass:disc} and~\ref{ass:operators} hold. 
		Let us define the auxiliary processes
		\begin{equation}\label{eq:def-WAndelta}
			W^\delta_{A_n}(t) \deq \frac{1}{\Gamma(\delta)} \int_0^t 
			(t-s)^{\delta-1} 
			S_n(t-s) 
			\rd W_n(s), 
			\quad 
			\delta \in (\nicefrac{1}{2}, \infty), \; t \in [0,\infty),
		\end{equation}
		where $\Gamma$ denotes the Gamma function~\cite[Section~5.2]{OlverBoisvertClark2010}.
		Then, for every $\beta' \in (\beta, \infty)$, $T \in (0,\infty)$ and $p \in [1,\infty)$, we have $W_{A_n}^{\beta' + \frac{1}{2}} \in C([0,T]; L^p(\Omega; E_n))$ for all $n \in \clos\N$. 
		If we suppose in addition that Assumption~\ref{ass:A-conv} is satisfied, then
		\begin{equation}
			\widetilde W_{A_n}^{\beta' + \frac{1}{2}} 
			\to W_{A_\infty}^{\beta' + \frac{1}{2}}
			\quad 
			\text{in } C([0,T]; L^p(\Omega; E_\infty))
			\quad 
			\text{as } n \to \infty.
		\end{equation}
	\end{proposition}
	\begin{proof}
		The fact that $W_{A_n}^{\beta' + \frac{1}{2}} \in C([0,T]; L^p(\Omega; E_n))$ for all $n \in \clos\N$ can be established by arguing as in Proposition~\ref{prop:existence-stoch-WAn}, thus using Assumptions~\ref{ass:disc} and~\ref{ass:operators}.
		For all $t \in [0,T]$, the It\^o inequality~\eqref{eq:Ito-inequality} yields
		\begin{align}
			&\norm{\widetilde W_{A_n}^{\beta' + \frac{1}{2}}(t) - W_{A_\infty}^{\beta' + \frac{1}{2}}(t)}{L^p(\Omega; E_\infty)}
			\\&\quad \lesssim_{(p, E_\infty)}
			\frac{1}{\Gamma(\beta' + \frac{1}{2})}
			\biggl(
			\int_0^t (t-s)^{2\beta' - 1} 
			\norm{\widetilde S_n(t-s) - S_\infty(t-s)}{\gamma(H_\infty; E_\infty)}^2
			\rd s\biggr)^{\frac{1}{2}}
			\\&\quad \le
			\frac{1}{\Gamma(\beta' + \frac{1}{2})}
			\biggl(\int_0^T s^{2\beta' - 1} 
			\norm{\widetilde S_n(s) - S_\infty(s)}{\gamma(H_\infty; E_\infty)}^2
			\rd s\biggr)^{\frac{1}{2}}.
		\end{align}
		Since semigroups commute with fractional powers of their infinitesimal generators, we can write the difference between the semigroups as follows:
		\begin{align}
			\widetilde S_n(s) - S_\infty(s)
			&=
			\Lambda_n (\id_n + A_n)^\beta S_n(s)\Pi_n \widetilde R_n^{\beta} 
			-
			(\id_\infty + A_\infty)^\beta S_\infty(s) R_\infty^\beta
			\\
			&=
			\begin{aligned}[t]
				&\Lambda_n (\id_n + A_n)^\beta S_n(s) \Pi_n (\widetilde R_n^{\beta} - R_\infty^{\beta})
				\\&+
				(\Lambda_n (\id_n + A_n)^\beta S_n(s) \Pi_n  - (\id_\infty + A_\infty)^{\beta} S_\infty) R_\infty^{\beta}.
			\end{aligned}
		\end{align}
		Thus, by the triangle inequality, it suffices to show that 
		\begin{equation}
			\begin{aligned}
				\text{(I)} &\deq \int_0^T s^{2\beta' - 1} 
				\norm{\Lambda_n (\id_n + A_n)^\beta S_n(s) \Pi_n (\widetilde R_n^{\beta} - R_\infty^{\beta})}{\gamma(H_\infty; E_\infty)}^2
				\rd s \quad \text{ and } \\
				\text{(II)} &\deq \int_0^T s^{2\beta' - 1} 
				\norm{[\Lambda_n (\id_n + A_n)^\beta S_n(s) \Pi_n  - (\id_\infty + R_\infty)^{\beta} S_\infty]R_\infty^{\beta}}{\gamma(H_\infty; E_\infty)}^2
				\rd s
			\end{aligned}
		\end{equation}
		tend to zero as $n \to \infty$.
		Applying the ideal property~\eqref{eq:ideal-property} of $\gamma(H_\infty; E_\infty)$, followed by the analytic semigroup estimate~\eqref{eq:app-analytic-semigroup-est} in conjunction with Assumptions~\ref{ass:disc} and~\ref{ass:operators}, we find
		\begin{equation}\label{eq:d2c-meansquare-conv-Wdelta-estI}
			\begin{aligned}
				\text{(I)}
				&\le 
				\norm{\widetilde R_n^{\beta} - R_\infty^{\beta}}{\gamma(H_\infty; E_\infty)}^2
				\int_0^T s^{2\beta' - 1} \norm{\Lambda_n (\id_n + A_n)^\beta S_n(s) \Pi_n}{\LO(E_\infty)}^2 \rd s
				\\&\lesssim_{(\beta, M_\Lambda, M_\Pi)} 
				\norm{\widetilde R_n^{\beta} - R_\infty^{\beta}}{\gamma(H_\infty; E_\infty)}^2
				\int_0^T s^{2(\beta' - \beta) - 1} \rd s
				\\&=
				\frac{T^{2(\beta' - \beta)}}{2(\beta' - \beta)}
				\norm{\widetilde R_n^{\beta} - R_\infty^{\beta}}{\gamma(H_\infty; E_\infty)}^2
				\to 0,
			\end{aligned}
		\end{equation}
		where the convergence on the last line follows from the second part of~\ref{ass:A-conv}.
		The convergence $\mathrm{(II)} \to 0$ follows by applying Lemma~\ref{lem:approx-L2gamma} with 
		\begin{equation}
			M_n(t) \deq t^\beta \Lambda_n (\id_n + A_n)^\beta S_n(t) \Pi_n
			\quad \text{and} \quad 
			R(t) \deq t^{\beta' - \beta - \frac{1}{2}} R_\infty^{\beta},
		\end{equation}
		Indeed, this is justified since $R \in L^2(0,T; \gamma(H_\infty; E_\infty))$ with 
		\begin{equation}
			\norm{R}{L^2(0,T; \gamma(H_\infty; E_\infty))}^2
			=
			\frac{T^{2(\beta' - \beta)}}{2(\beta' - \beta)} \norm{R^\beta_\infty}{\gamma(H_\infty; E_\infty)}^2
			<
			\infty,
		\end{equation}
		condition~\ref{lem:approx-L2gamma:bdd} is verified by applying~\eqref{eq:app-analytic-semigroup-est} to $\norm{M_n(t)}{\LO(E_\infty)}$ combined with Assumptions~\ref{ass:disc} and~\ref{ass:operators} as in~\eqref{eq:d2c-meansquare-conv-Wdelta-estI}, and hypothesis~\ref{lem:approx-L2gamma:strong-conv} holds by Lemma~\ref{lem:strong-conv-AnalphaS}.
	\end{proof}
	
	We will show that there exist modifications of $W_{A_\infty}$ and 
	$(\widetilde W_{A_n})_{n\in\N}$ which, for all $p \in [1, \infty)$ and $T \in 
	(0,\infty)$, belong to $L^p(\Omega; C([0,T]; E_\infty))$ and converge in 
	this norm. 
	In 
	particular, as $n \to \infty$, their trajectories converge uniformly on 
	bounded time intervals, $\bbP$-a.s.
	
	The proof is based on the Da Prato--Kwapień--Zabczyk factorization method, 
	first formulated in~\cite{DaPratoKwapienZabczyk1987} for Hilbert spaces (see also~\cite[Section~5.3]{DaPrato2014}), and later extended to UMD-type-2 Banach spaces in~\cite[Theorem~3.2]{Brzezniak1997}.
	The general idea is to express the process $W_{A_n}$ as the `product' $\fI_{A_n}^{\frac{1}{2} - \beta'} W^{\frac{1}{2} + \beta'}_{A_n}$ of a fractional parabolic integral operator $\fI_{A_n}^{\frac{1}{2} - \beta'}$ as in Appendix~\ref{app:frac-par-int} and auxiliary process $W^{\frac{1}{2} + \beta'}_{A_n}$ as in~\eqref{eq:def-WAndelta}, and using the smoothing properties of the former.
	
	\begin{proposition}\label{prop:d2c-pathwise-conv-Wdelta}
		Let $p \in [1,\infty)$ and $T \in (0,\infty)$.
		If Assumptions~\ref{ass:disc}--\ref{ass:operators} 
		hold, then for every $n \in \clos \N$ there exists a modification of $W_{A_n}$ which belongs to $L^p(\Omega; C([0,T]; E_n)$, which we will identify with $W_{A_n}$ itself.
		
		If, in addition, Assumption~\ref{ass:A-conv} holds, then the sequence $(\widetilde W_{A_n})_{n\in\clos\N}$ satisfies
		\begin{equation}
			\widetilde W_{A_n} \to W_{A_\infty} \quad \text{in } L^p(\Omega; 
			C([0,T]; E_\infty))
			\quad \text{as } n \to \infty.
		\end{equation}
	\end{proposition}
	\begin{proof}
		Let $\beta' \in (\beta, \frac{1}{2})$, where $\beta \in [0, \frac{1}{2})$ is as in~\ref{ass:operators}.
		Since $L^p(\Omega)$-spaces with higher exponents are embedded in those with lower ones contractively (because $\bbP(\Omega) = 1$), we assume without loss of generality that $p \in ((\frac{1}{2} - \beta)^{-1}, \infty)$.
		By the first part of Proposition~\ref{prop:d2c-meansquare-conv-Wdelta} (thus using Assumptions~\ref{ass:disc} and~\ref{ass:operators}) and the Fubini theorem, we have for all $n \in \clos\N$:
		\begin{equation}
			W^{\frac{1}{2} + \beta'}_{A_n} 
			\in 
			C([0,T]; L^p(\Omega; E_n))
			\emb 
			L^p(0,T; L^p(\Omega; E_n))
			\cong 
			L^p(\Omega; L^p(0,T; E_n)),
		\end{equation}
		where the constants for the first embedding depend only on $p$ and $T$.
		In particular, there exists an event $\Omega_1 \subseteq \Omega$ with $\bbP(\Omega_1) = 1$ such that $W^{\frac{1}{2} + \beta'}_{A_n}(\omega, \,\cdot\,)$ belongs to $L^p(0,T; E_n)$ for all $\omega \in \Omega_1$.
		It then follows from Proposition~\ref{prop:app-basicPropsIA}\ref{prop:app-basicPropsIA:Lp-CHdot} that $\fI_{A_n}^{\frac{1}{2} - \beta'} W^{\frac{1}{2} + \beta'}_{A_n}(\omega, \,\cdot\,)$ belongs to $C([0,T]; E_n)$, where $(\mathfrak I_{A_n}^s)_{s\in[0,\infty)}$ are the fractional parabolic integral operators defined by~\eqref{eq:app-def-fracParab} in Appendix~\ref{app:frac-par-int}.
		In this case, 
		the process 
		$\fI_{A_n}^{\frac{1}{2} - \beta'} W^{\frac{1}{2} + \beta'}_{A_n}$ (set to 
		zero outside of $\Omega_1$) belongs 
		to $L^p(\Omega; C([0,T]; E_n))$, 
		and	by the factorization theorem~\cite[Theorem~3.2]{Brzezniak1997} it is a modification of $W_{A_n}$. 
		
		For the lifted processes, the 
		properties of the embeddings and 
		projections from assumption~\ref{ass:disc} imply that
		$\widetilde{\mathfrak I}_{A_n}^{\frac{1}{2} - \beta'} \widetilde
		W_{A_n}^{\frac{1}{2} + \beta'} \in L^p(\Omega; C([0,T]; E_\infty))$ is a 
		continuous 
		modification of 
		$\widetilde W_{A_n}$, where
		$\widetilde{\mathfrak I}_{A_n}^s \deq \Lambda_n \mathfrak I_{A_n}^s \Pi_n$.
		Identifying $\widetilde W_{A_n}$ with its factorized continuous 
		modification for every $n \in \clos \N$, we can estimate as follows:
		\begin{equation}
			\begin{aligned}
				\norm{\widetilde W_{A_n} - W_{A_\infty}}{L^p(\Omega; C([0,T]; 
					E_\infty))}
				&=
				\norm{\widetilde{\mathfrak I}_{A_n}^{\frac{1}{2} - \beta'} 
					\widetilde W_{A_n}^{\frac{1}{2} + \beta'} 
					- 
					{\mathfrak I}_{A_\infty}^{\frac{1}{2} - \beta'}
					W_{A_\infty}^{\frac{1}{2} + \beta'}
				}{L^p(\Omega; C([0,T]; E_\infty))}
				\\
				&\le 
				\norm{\widetilde{\mathfrak I}_{A_n}^{\frac{1}{2} - \beta'} 
					(\widetilde W_{A_n}^{\frac{1}{2} + \beta'} - 
					W_{A_\infty}^{\frac{1}{2} + \beta'})}{L^p(\Omega; C([0,T]; 
					E_\infty))}
				\\&\quad +
				\norm{(\widetilde{\mathfrak I}_{A_n}^{\frac{1}{2} - \beta'} - 
					\mathfrak I_{A_\infty}^{\frac{1}{2} - \beta'}) 
					W_{A_\infty}^{\frac{1}{2} + \beta'}}{L^p(\Omega; C([0,T]; 
					E_\infty))}.
			\end{aligned}
		\end{equation}
		Since $\frac{1}{2} - \beta' > \frac{1}{p}$, we can apply	Corollary~\ref{cor:unifbb-IAtilde}\ref{cor:unifbb-IAtilde:b} to find that $\widetilde{\mathfrak I}_{A_n}^{\frac{1}{2} - \beta'}$ is a bounded 
		linear operator from 
		$L^p(0,T; E_\infty)$ 
		to $C([0,T]; E_\infty)$ whose norm can be bounded independently of $n$. Thus, by the 
		above 
		discussion and the second part of Proposition~\ref{prop:d2c-meansquare-conv-Wdelta} (which uses Assumption~\ref{ass:A-conv}), we find
		\begin{equation}\begin{aligned}
				&\norm{\widetilde{\mathfrak I}_{A_n}^{\frac{1}{2} - \beta'} 
					(\widetilde W_{A_n}^{\frac{1}{2} + \beta'} - 
					W_{A_\infty}^{\frac{1}{2} + \beta'})}{L^p(\Omega; C([0,T]; 
					E_\infty))}
				\\&\lesssim
				\norm{\widetilde  W_{A_n}^{\frac{1}{2} + \beta'} - 
					W_{A_\infty}^{\frac{1}{2} + \beta'}}{L^p(\Omega \times (0,T); 
					E_\infty)}
				\le T^{\frac{1}{p}} \norm{\widetilde  
					W_{A_n}^{\frac{1}{2} + \beta'} - 
					W_{A_\infty}^{\frac{1}{2} + \beta'}}{C([0,T]; L^p(\Omega; 
					E_\infty))}
				\to 0.
			\end{aligned}
		\end{equation}
		Now we note that, for all $\omega \in \Omega$, 
		Proposition~\ref{prop:strong-continuity-IAntilde}\ref{prop:strong-continuity-IAntilde:a}
		implies that
		\begin{align}
			\norm{
				\widetilde{\mathfrak I}_{A_n}^{\frac{1}{2} - \beta'} 
				W_{A_\infty}^{\frac{1}{2} + \beta'}(\omega, \,\cdot\,) 
				- 
				\mathfrak I_{A_\infty}^{\frac{1}{2} - \kappa} 
				W_{A_\infty}^{\frac{1}{2} +	\beta'}(\omega, \,\cdot\,) 
			}{C([0,T]; E_\infty)} \to 0.
		\end{align}
		Again by Corollary~\ref{cor:unifbb-IAtilde}\ref{cor:unifbb-IAtilde:b},
		we moreover have 
		\begin{align}
			\norm{
				\widetilde{\mathfrak I}_{A_n}^{\frac{1}{2} - \beta'} 
				W_{A_\infty}^{\frac{1}{2} + \beta'}(\omega, \,\cdot\,) 
				- 
				\mathfrak I_{A_\infty}^{\frac{1}{2} - \beta'} 
				W_{A_\infty}^{\frac{1}{2} + \beta'}(\omega, \,\cdot\,) 
			}{C([0,T]; E_\infty)}
			\lesssim
			2 \norm{W_{A_\infty}^{\frac{1}{2} + \beta'}}{L^p(0,T; E_\infty)}
		\end{align}
		with constant independent of $n \in \N$, and since 
		$W_{A_\infty}^{\frac{1}{2} + \beta'} 
		\in L^p((0,T) \times \Omega; E_\infty)$, 
		the dominated convergence theorem yields
		\[
		\norm{(\widetilde{\mathfrak I}_{A_n}^{\frac{1}{2} - \beta'} - 
			\mathfrak I_{A_\infty}^{\frac{1}{2} - \beta'}) 
			W_{A_\infty}^{\frac{1}{2} + \beta'}}%
		{L^p(\Omega; C([0,T]; E_\infty))}
		\to 
		0. \qedhere
		\]
	\end{proof}
	
	\section{Approximation of semilinear stochastic evolution equations with additive cylindrical 
		Wiener noise}
	\label{sec:semilinear-L2}
	
	In this section, we shall extend the results from Section~\ref{sec:OU-H} regarding the \emph{linear} $E_n$-valued equation~\eqref{eq:OU} to the \emph{semilinear case}.
	As before, let the spaces $(E_n)_{n\in\clos\N}$, $(H_n)_{n\in\clos \N}$ and the operators $(A_n)_{n\in\clos\N}$ satisfy assumptions~\ref{ass:disc} and~\ref{ass:operators}, respectively, and suppose that $W_n \deq \Pi_n W_\infty$ is an $H_n$-valued $Q_n$-cylindrical Wiener process (with $Q_n \deq \Pi_n \Pi_n^*$), supported on $(\Omega, \cF, (\cF_t)_{t\in[0,T]}, \bbP)$.
	Let $T \in (0,\infty)$ be a finite time horizon. 
	In this section, we suppose moreover that we are given a drift coefficient function $F_n \from \Omega \times [0,T]  \times E_n \to E_n$ and initial datum $\xi_n \from \Omega \to E_n$. 
	We will consider the following semilinear stochastic evolution equation:
	\begin{equation}\label{eq:Langevin-equation-infinite-dim-A}
		\left\lbrace
		\begin{aligned}
			\rd X_n(t) &= -A_n X_n(t)\rd t + F_n(t, X_n(t)) \rd t + \rd W_n(t), 
			\quad 
			t 
			\in 
			(0,T], 
			\\
			X_n(0) &= \xi_n.
		\end{aligned}
		\right.
	\end{equation}
	Note that $F_n(\omega, t, \,\cdot\,)$ is a (nonlinear) operator on $E_n$ for all $(\omega, t) \in \Omega \times [0,T]$; in Section~\ref{sec:graph-WM}, we considered the specific case $[F_n(\omega, t, u_n)](x) \deq f_n(\omega, t, u_n(x))$ for some real-valued nonlinearity $f_n$.
	
	In what follows, we shall impose more precise conditions on the $F_n$ and 
	$\xi_n$ to ensure the well-posedness 
	of~\eqref{eq:Langevin-equation-infinite-dim-A} for every fixed $n \in \clos 
	\N$ and to obtain discrete-to-continuum convergence of the respective solutions 
	as $n \to 
	\infty$. 
	
	In Section~\ref{sec:globalLip-linGrowth} we will assume in particular 
	that the drifts $(F_n)_{n\in\clos{\N}}$ are uniformly globally Lipschitz and of linear growth to obtain the existence of unique global solutions $(X_n(t))_{t\in[0,T]}$, whose lifted counterparts $\widetilde X_n$ converge to $X_\infty$ in 
	$L^p(\Omega; C([0,T]; E_\infty))$ as $n \to \infty$, where $p \in [1,\infty)$ is the stochastic integrability of the initial data.  
	These assumptions 
	are relaxed 
	in 
	Section~\ref{sec:locLip}, where we suppose that the drifts are uniformly 
	locally Lipschitz and uniformly bounded near zero. In general, this comes at 
	the  
	cost of obtaining merely local solutions, converging in a weaker norm.
	However, if one can show independently that the solutions are global and the 
	$L^p(\Omega; C([0,T]; E_n))$-norms of $X_n$ are uniformly bounded in $n \in 
	\N$, then we recover the stronger sense of convergence.
	
	\subsection{Globally Lipschitz drifts of linear growth}
	\label{sec:globalLip-linGrowth}
	
	In this section we suppose that the drift coefficients 
	$F_n \from \Omega \times [0,T] \times E_n \to E_n$ 
	in~\eqref{eq:Langevin-equation-infinite-dim-A} for $n \in 
	\clos \N$ are uniformly globally Lipschitz continuous and of linear growth. More 
	precisely:
	\begin{enumerate}[label=(F{\arabic*}), series=Fconditions]
		\item\label{item:F-globalLip-linearGrowth} There exist $L_F, C_F \in 
		(0,\infty)$ such that, for 
		all 
		$t \in [0,T]$, $\omega \in \Omega$, $n \in \clos \N$ 
		and $x_n, y_n \in E_n$,
		\begin{align}
			\norm{F_n(\omega, t, x_n) - F_n(\omega, t, y_n)}{E_n}
			&\le 
			L_F \norm{x_n - y_n}{E_n}; \\ 
			\norm{F_n(\omega, t, x_n)}{E_n} &\le C_F (1 + \norm{x_n}{E_n}).
		\end{align}
		Moreover, the process 
		$(\omega, t) \mapsto F_n(\omega, t, x_n)$ is 
		strongly measurable and adapted to the filtration $(\cF_t)_{t\in[0,T]}$.
	\end{enumerate}
	Now let us comment on the existence and uniqueness of solutions 
	to~\eqref{eq:Langevin-equation-infinite-dim-A} for fixed $n \in \clos\N$. 
	We will use the following concept of global mild solutions, see~\cite[pp.~969--970]{vNVW2008}. 
	In Subsection~\ref{sec:locLip}, we also introduce the concept of local solutions, which may blow up in finite time. 
	In particular, a local solution which exists $\bbP$-a.s.\ on the whole of $[0,T]$ is in fact global.
	
	Recall that $(S_n(t))_{t\ge0}$ denotes the $C_0$-semigroup on $E_n$ generated by $-A_n$.
	\begin{definition}\label{def:mildsol}
		An $E_n$-valued stochastic process $X_n = (X_n(t))_{t \in [0,T]}$ is a 
		global mild solution 
		to~\eqref{eq:Langevin-equation-infinite-dim-A} with coefficients $(A_n, 
		F_n, \xi_n)$ if 
		\begin{enumerate}
			\item $X_n \from \Omega \times [0,T] \to E_n$ is strongly measurable 
			and 
			$(\cF_t)_{t\in[0,T]}$-adapted; 
			\item $s \mapsto S_n(t-s) F_n(s, X_n(s)) \in L^0(\Omega; L^1(0,t; 
			E_n))$ 
			for every $t \in [0,T]$; 
			\item $s \mapsto S_n(t-s) \in L^2(0,t; \gamma(H_n; E_n))$ for 
			every $t \in [0,T]$;
			\item for all $t \in [0,T]$, we have 
			\begin{equation}
				X_n(t) = S_n(t)\xi_n + \int_0^t S_n(t-s)F_n(s, X_n(s))\rd s + 
				W_{A_n}(t), 
				\quad 
				\bbP\text{-a.s.}
			\end{equation}
		\end{enumerate}
	\end{definition}
	In the present framework, existence and uniqueness can be proven by 
	showing that 
	the operator 
	$\Phi_{n,T}$ given by
	\begin{equation}\label{eq:fixed-point-operator}
		[\Phi_{n,T}(u_n)](t) 
		\deq 
		S_n(t)\xi_n + \int_0^t S_n(t-s)F_n(s, u_n(s))\rd s + W_{A_n}(t)
	\end{equation}
	is a well-defined and Lipschitz-continuous mapping on $L^p(\Omega; C([0,T]; E_n))$, whose Lipschitz constant tends to zero as $T \downarrow 0$ (see~\cite[Proposition~6.1]{vNVW2008} or~\cite[Theorem~3.7]{KvN2011} for more general results):
	\begin{proposition}\label{prop:existence-uniqueness-globLip-linGrowth}
		Suppose that Assumptions~\ref{ass:disc},~\ref{ass:operators} and~\ref{item:F-globalLip-linearGrowth} are satisfied. 	
		Let $n \in \clos\N$, $p \in [1, \infty)$, $\xi_n \in L^p(\Omega, \cF_0, \bbP; E_n)$ and $T \in (0,\infty)$. 
		The operator $\Phi_{n,T}$ given by~\eqref{eq:fixed-point-operator} is well defined and Lipschitz continuous on $L^p(\Omega; C([0,T]; E_n))$. 
		Its Lipschitz constant is independent of $\xi_n$, depends on $A_n$ and $F_n$ only through $M_S$ and $L_F$, and tends to zero as $T \downarrow 0$.
	\end{proposition}
	\begin{proof}
		The fact that $S_n \otimes \xi_n \in L^p(\Omega; C([0,T]; E_n))$ is immediate from~\ref{ass:disc}--\ref{ass:operators} and $\xi_n \in L^p(\Omega, \cF_0, \bbP; E_n)$.
		We also have $W_{A_n} \in L^p(\Omega; C([0,T]; E_n))$ by the first part of Proposition~\ref{prop:d2c-pathwise-conv-Wdelta}.
		Given $u_n \in L^p(\Omega; C([0,T]; E_n))$, it follows from~\ref{item:F-globalLip-linearGrowth} that $\norm{s \mapsto F_n(s, u_n(s))}{L^\infty(0,T; E_n)} \le C_F (1 + \norm{u_n}{C([0,T]; E_n)})$, so that $S_n * F_n(\,\cdot\,, u_n)$ belongs to $L^p(\Omega; C([0,T]; E_n))$ with
		\begin{equation}
			\norm{S_n * F_n(\,\cdot\,, u_n)}{L^p(\Omega; C([0,T]; E_n))}
			\le 
			C_F (1 + \norm{u_n}{L^p(\Omega; C([0,T]; E_n))})
		\end{equation}
		by Proposition~\ref{prop:app-basicPropsIA}\ref{prop:app-basicPropsIA:Lp-CHdot} with $E = F = E_n$, $\alpha = 0$ and $s = 1$ (noting that $S_n * f = \mathfrak I_{A_n}^1 f$, see Appendix~\ref{app:frac-par-int}).
		This shows that $\Phi_{n,T}$ is well-defined.
		
		Now let $u_n, v_n \in L^p(\Omega; C([0,T]; E_n))$ and observe that
		\begin{equation}
			\Phi_{n,T}(u_n) - \Phi_{n,T}(v_n)
			=
			\int_0^t S_n(t-s) [F_n(s, u_n(s)) - F_n(s, v_n(s))] \rd s.
		\end{equation}
		A straightforward estimate involving Assumptions~\ref{ass:disc},~\ref{ass:operators} and~\ref{item:F-globalLip-linearGrowth} then yields
		\[
		\norm{\Phi_{n,T}(u_n) - \Phi_{n,T}(v_n)}{L^p(\Omega; C([0,T]; E_n))}
		\le M_S L_F T \, \norm{u_n - v_n}{L^p(\Omega; C([0,T]; E_n))}. \qedhere
		\]
	\end{proof}
	
	Under the conditions of 
	Proposition~\ref{prop:existence-uniqueness-globLip-linGrowth}, it follows from 
	the Banach fixed-point theorem 
	that~\eqref{eq:Langevin-equation-infinite-dim-A}
	has a unique solution on a small enough time interval $[0, T_0]$, which can be 
	extended to a unique global mild solution on any $[0,T]$ by ``patching together'' solutions on small time 
	intervals:
	\begin{proposition}\label{prop:global-well-posedness-En}
		Suppose that Assumptions~\ref{ass:disc},~\ref{ass:operators} and~\ref{item:F-globalLip-linearGrowth} are satisfied, and let $n \in \clos\N$, $p \in [1, \infty)$, $\xi_n \in L^p(\Omega, \cF_0, \bbP; E_n)$ and $T \in (0,\infty)$. 
		Then~\eqref{eq:Langevin-equation-infinite-dim-A} has a unique global mild solution $X_n \in L^p(\Omega; C([0,T]; E_n))$.
	\end{proposition}
	\begin{proof}
		By Proposition~\ref{prop:existence-uniqueness-globLip-linGrowth}, there exists $T_0 \in (0,\infty)$ such that $\Phi_{n,T_0}$ is a strict contraction on $L^p(\Omega; C([0,T_0]; E_n))$, and thus has a unique fixed point $X_n$.
		Since the bound on the Lipschitz constant of $\Phi_{n,T}$ only depended on $M_S$, $L_F$ and $T$, we can repeat the previous argument to obtain a unique solution $Y \in L^p(\Omega; C([0,T_0]; E_n))$ for~\eqref{eq:Langevin-equation-infinite-dim-A} with initial datum $\eta_n \deq X_n(\frac{1}{2}T_0)$, drift $G_n(\,\cdot\,, u_n) \deq F_n(\,\cdot\, + \frac{1}{2}T_0, u_n)$ and noise $\widehat W_n \deq W_n(\,\cdot\, + \frac{1}{2}T_0)$.
		It can then be argued directly using Definition~\ref{def:mildsol} that the concatenation of the processes $X_n$ and $Y_n(\,\cdot\,+\frac{1}{2}T_0)$ is the unique mild solution to~\eqref{eq:Langevin-equation-infinite-dim-A} with the original data on $[0,\frac{3}{2}T_0]$.
		Proceeding inductively, we find the same conclusion for all intervals $[0, (k + \frac{1}{2})T_0]$ with $k \in \N$ and thus for $[0,T]$.
	\end{proof}
	
	For every $n \in \N$, we analogously define the lifted initial datum $\widetilde \xi_n \from \Omega \to E_\infty$ by~$\widetilde \xi_n \deq \Lambda_n \xi_n$ and the lifted drift coefficient $\widetilde F_n \from \Omega \times [0,T] \times E_\infty \to E_\infty$ by
	\begin{equation}\label{eq:def-Fntilde}
		\widetilde F_n(\omega, t, x) \deq \Lambda_n F_n(\omega, t, \Pi_n x), 
		\quad 
		(\omega, t, x) \in \Omega \times [0,T] \times E_\infty.
	\end{equation}
	We will assume that the initial data and drift coefficients are approximated in 
	the following way:
	
	\begin{enumerate}[label=(IC)]
		\item\label{ass:IC}%
		There exists $p \in [1,\infty)$ such that 
		$(\xi_n)_{n\in\clos\N} \in \prod_{n\in\clos\N} L^p(\Omega, \cF_0, \bbP; E_n)$ and
		\begin{equation}
			\widetilde \xi_n \to \xi_\infty 
			\quad \text{in } L^p(\Omega; E_\infty) \text{ as } n \to \infty.
		\end{equation}
	\end{enumerate}
	\begin{enumerate}[label=(F{\arabic*}), resume=Fconditions]
		\item\label{item:F-approx} 
		For a.e.\ $(\omega, t) \in \Omega \times [0,T]$ and every $x \in E_\infty$, we have
		\begin{equation}
			\widetilde F_n(\omega, t, x) 
			\to 
			F_\infty(\omega, t, x)
			\quad 
			\text{ in } E_\infty 
			\text{ as } n \to \infty.
		\end{equation}
	\end{enumerate}
	Under these assumptions, we obtain the main result of this section, namely the 
	following discrete-to-continuum convergence theorem in the context of uniformly 
	globally 
	Lipschitz 
	nonlinearities of linear growth.
	It is analogous to~\cite[Theorem~4.3]{KvN2011}.
	\begin{theorem}\label{thm:d2c-semilinear-global-linear}
		Suppose that~\ref{ass:disc}--\ref{ass:A-conv},~\ref{item:F-globalLip-linearGrowth}--\ref{item:F-approx} and~\ref{ass:IC} are satisfied, with $p \in [1,\infty)$.
		For all $n \in \clos \N$ and $T \in 
		(0,\infty)$, let $X_n = (X_n(t))_{t\in[0,T]}$ denote the unique global 
		solution to~\eqref{eq:Langevin-equation-infinite-dim-A}, and let 
		$\widetilde X_n \deq \Lambda_n X_n$.
		Then we have
		\begin{equation}
			\widetilde X_n \to X_\infty
			\quad 
			\text{in } 
			L^p(\Omega; C([0,T]; E_\infty))
			\quad 
			\text{as } 
			n \to \infty.
		\end{equation}
	\end{theorem}
	Its proof involves the lifted counterparts of $\Phi_{n,T}$, defined by
	\begin{equation}
		\widetilde \Phi_{n,T} \deq \Lambda_n \Phi_{n,T} \circ \Pi_n 
		\from 
		L^p(\Omega; C([0,T]; E_\infty))
		\to 
		L^p(\Omega; C([0,T]; E_\infty)),
	\end{equation}
	i.e., for all $u \in L^p(\Omega; C([0,T]; E_\infty))$ and $t \in [0,T]$, we have 
	\begin{align}
		[\widetilde \Phi_{n, T}(u)](t) 
		\deq{}& 
		\Lambda_n S_n(t)\xi_n + \int_0^t \Lambda_n S_n(t-s) F_n(s, 
		\Pi_n u(s))\rd s + \Lambda_n W_{A_n}(t)
		\\
		={}&
		\widetilde S_n(t) \widetilde \xi_n + \int_0^t \widetilde S_n(t-s) 
		\widetilde F_n(s, 
		u(s))\rd s + \widetilde W_{A_n}(t),
		\quad 
		\bbP\text{-a.s.}, \label{eq:lifted-fixedPoint}
	\end{align}
	where the second identity is due to 
	Assumption~\ref{ass:disc3}\ref{ass:disc:id3}. 
	Using the tensor and convolution notations from Section~\ref{sec:prelims:notation}, it can be expressed even more concisely as
	\begin{equation}\label{eq:lifted-fixedPoint-shorter}
		\widetilde \Phi_{n, T}(u)
		=
		\widetilde S_n \otimes \widetilde \xi_n + \widetilde S_n * \widetilde F_n(\,\cdot\,, u) + \widetilde W_{A_n}
	\end{equation}
	In particular, we will show that all three terms of~\eqref{eq:lifted-fixedPoint-shorter} converge to their ``continuum'' counterparts; they are addressed by Lemmas~\ref{lem:ICterm-convs}--\ref{lem:deterministic-convs} below (which are analogous to~\cite[Lemma~4.4, 4.5(1) and 4.5(3)]{KvN2011}), as well as Proposition~\ref{prop:d2c-pathwise-conv-Wdelta}.
	\begin{lemma}\label{lem:ICterm-convs}
		If~\ref{ass:disc}--\ref{ass:A-conv}
		and~\ref{ass:IC} hold with $p \in [1,\infty)$,
		then we have 
		\begin{equation}
			\widetilde S_n \otimes \xi_n
			\to 
			S_\infty \otimes \xi_\infty
			\quad 
			\text{in } L^p(\Omega; C([0,T]; E_\infty))
			\quad 
			\text{as } n \to \infty.
		\end{equation}
	\end{lemma}
	\begin{proof}
		As in the beginning of the proof of Proposition~\ref{prop:existence-uniqueness-globLip-linGrowth}, it follows from~\ref{ass:disc}--\ref{ass:operators} and~\ref{ass:IC} that $S_n \otimes \xi_n \in L^p(\Omega; C([0,T]; E_n))$ for all $n \in \clos\N$.
		Applying the projection and lifting operators from~\ref{ass:disc}, we thus find $\widetilde S_n \otimes \widetilde \xi_n \in L^p(\Omega; C([0,T]; E_\infty))$.
		
		The triangle inequality implies
		\begin{align}
			&\norm{\widetilde S_n \otimes \widetilde \xi_n - S_\infty \otimes \xi_\infty}{L^p(\Omega; C([0,T]; E_\infty))}
			\\&\quad \le 
			\norm{\widetilde S_n \otimes (\widetilde \xi_n - \xi_\infty)}{L^p(\Omega; C([0,T]; E_\infty))}
			+
			\norm{(\widetilde S_n - S_\infty) \otimes \xi_\infty}%
			{L^p(\Omega; C([0,T]; E_\infty))}.
		\end{align}
		By~\ref{ass:disc}--\ref{ass:operators} and~\ref{ass:IC}, for the first term we have, as $n \to \infty$:
		\begin{equation}
			\norm{\widetilde S_n \otimes (\widetilde \xi_n - \xi_\infty)}{L^p(\Omega; C([0,T]; E_\infty))}
			\le 
			M_\Lambda M_S M_\Pi 
			\norm{\widetilde \xi_n - \xi_\infty}{L^p(\Omega; E_\infty)} \to 0
		\end{equation}
		For the second term, first note that 
		$\widetilde S_n \otimes \xi_\infty(\omega) \to 
		S_\infty \otimes \xi_\infty(\omega)$ in 
		$C([0,T]; E_\infty)$, $\bbP$-a.s., by Theorem~\ref{thm:d2c-TKapprox}, where we now also use~\ref{ass:A-conv}. 
		Since we moreover have 
		\begin{equation}
			\norm{\widetilde S_n \otimes \xi_\infty(\omega) - 
				S_\infty \otimes \xi_\infty(\omega)}{C([0,T]; E_\infty)}
			\le 
			M_S(M_\Lambda M_\Pi + 1) \norm{\xi_\infty(\omega)}{E_\infty},
		\end{equation}
		and the right-hand side belongs to $L^p(\Omega)$ by assumption, we deduce that also 
		\[
		\norm{(\widetilde S_n - S_\infty) \otimes \xi_\infty}{L^p(\Omega; C([0,T]; E_\infty))} \to 0
		\quad \text{as } 
		n \to \infty. \qedhere
		\]
	\end{proof}
	\begin{lemma}\label{lem:deterministic-convs}
		Suppose that Assumptions~\ref{ass:disc},~\ref{ass:operators},~\ref{item:F-globalLip-linearGrowth} and~\ref{item:F-approx} are satisfied. 
		Let $p \in [1,\infty)$ and $u \in L^p(\Omega; C([0,T]; E_\infty))$ be given.
		Then we have
		\begin{equation}
			\widetilde S_n * \widetilde F_n(\,\cdot\,, u)
			\to 
			S_\infty * F_\infty(\,\cdot\,, u )
			\quad 
			\text{in } L^p(\Omega; C([0,T]; E_\infty))
			\text{ as } n \to \infty.
		\end{equation}
	\end{lemma}
	\begin{proof}
		Similarly to the proof of Proposition~\ref{prop:existence-uniqueness-globLip-linGrowth}, it follows from~\ref{ass:disc}--\ref{ass:operators} and~\ref{item:F-globalLip-linearGrowth} that $\widetilde S_n * \widetilde F_n(\,\cdot\, , u) \in L^p(\Omega; C([0,T]; E_\infty))$ for all $n \in \clos\N$.
		By the triangle inequality, we can split up the statement into the following two assertions:
		\begin{enumerate}
			\item \label{lem:deterministic-convs:item1}
			$\widetilde S_n * \widetilde F_n(\,\cdot\, , u) - \widetilde S_n 
			* F_\infty(\,\cdot\, , u) \to 0$ in $L^p(\Omega; C([0,T]; E_\infty))$ 
			as $n \to \infty$; 
			\item \label{lem:deterministic-convs:item2}
			$\widetilde S_n * F_\infty(\,\cdot\, , u) \to S_\infty * 
			F_\infty(\,\cdot\, , u)$ in 
			$L^p(\Omega; C([0,T]; E_\infty))$  
			as $n \to \infty$.
		\end{enumerate}
		For almost every $(\omega, t) \in \Omega \times [0,T]$, we have by~\ref{item:F-globalLip-linearGrowth} and~\ref{ass:disc}:
		\begin{equation}\label{eq:Ftilde-n-unifBdd}
			\begin{aligned}
				&\norm{\widetilde F_n(\omega, t, u(\omega, t)) - F_\infty(\omega, t, u(\omega, t))}{E_\infty}
				\\&\quad 
				\le
				C_F(M_\Lambda + 1 + (M_\Pi M_\Lambda + 1)\norm{u(\omega, t)}{E_\infty}).
			\end{aligned}
		\end{equation}
		It follows that
		\begin{equation}\label{eq:Ftilde-n-unifBdd-LpOmegaZeroT}
			\begin{aligned}
				\norm{\widetilde F_n(\,\cdot\,, u) - F_\infty(\,\cdot\,, u)}{L^p(\Omega \times (0,T); E_\infty)}
				&\lesssim_{(C_F, M_\Lambda, M_\Pi)}
				\norm{u}{L^p(\Omega \times (0,T); E_\infty)} \\
				&\lesssim_{(p, T)} 
				\norm{u}{L^p(\Omega; C([0,T]; E_\infty))} < \infty.
			\end{aligned}
		\end{equation}
		Since $\widetilde S_n * f = \widetilde \fI_{A_n}^1 f$ for all $f \in L^p(0,T; E_\infty)$, we can apply Proposition~\ref{prop:app-basicPropsIA}\ref{prop:app-basicPropsIA:Lp-CHdot} with $E = F = E_\infty$, $\alpha = 0$ and $s = 1$ to find that 
		\begin{equation}
			\begin{aligned}
				&\norm{\widetilde S_n * (
					\widetilde F_n(\,\cdot\,, u) - F_\infty(\,\cdot\,, u)
					)
				}{C([0,T]; E_\infty)} 
				\\&\quad 
				\lesssim_{(s,r,T,M_S)}
				\norm{\widetilde F_n(\,\cdot\,, u) - F_\infty(\,\cdot\,, u)}{L^p(\Omega \times (0,T); E_\infty)}.
			\end{aligned}
		\end{equation}
		The latter tends to zero as $n \to \infty$ by the dominated convergence theorem, which applies in view of~\ref{item:F-approx} and~\eqref{eq:Ftilde-n-unifBdd}--\eqref{eq:Ftilde-n-unifBdd-LpOmegaZeroT}.  
		This shows~\ref{lem:deterministic-convs:item1}.
		
		For~\ref{lem:deterministic-convs:item2}, we derive in the same way that, for almost every $\omega \in \Omega$,
		\begin{equation}
			t \mapsto F_\infty(\omega, t, u(\omega, t)) \in L^p(0,T; E_\infty),
		\end{equation}
		which implies, cf.\ 
		Proposition~\ref{prop:strong-continuity-IAntilde}\ref{prop:strong-continuity-IAntilde:a} with $\widetilde E \deq E_\infty$, that
		\begin{equation}
			\widetilde S_n * F_\infty(\omega, \,\cdot\,, u(\omega, \,\cdot\,))
			\to 
			S_\infty * F_\infty(\omega, \,\cdot\,, u(\omega, \,\cdot\,))
			\quad 
			\text{in } C([0,T]; E_\infty) \quad \text{as } n \to \infty. 
		\end{equation}
		The conclusion follows by using the uniform boundedness of the operators $(\widetilde{\fI}_{A_n}^1)_{n\in\clos\N}$ in $\LO(L^p(0,T; E_\infty); C([0,T]; E_\infty))$, asserted in Corollary~\ref{cor:unifbb-IAtilde}\ref{cor:unifbb-IAtilde:b} (with $\widetilde E \deq E_\infty$ once more), and finishing the dominated convergence argument as in part~\ref{lem:deterministic-convs:item1}.
	\end{proof}
	\begin{proof}[Proof of Theorem~\ref{thm:d2c-semilinear-global-linear}]
		By Proposition~\ref{prop:existence-uniqueness-globLip-linGrowth} and~\ref{ass:disc}, for small enough $T_0 \in (0,\infty)$ there exists a constant $c \in [0, 1)$, depending only on $L_F$, $M_S$, $M_\Lambda$ and $M_\Pi$, such that, for all $u, v \in L^p(\Omega; C([0,T_0]; E_\infty))$,
		\begin{equation}
			\sup\nolimits_{n\in\N} \, \norm{\widetilde \Phi_{n,T_0}(u) - \widetilde 
				\Phi_{n, T_0}(v)}{L^p(\Omega; C([0, T_0]; E_\infty))}
			\le 
			c \norm{u - v}{L^p(\Omega; C([0, T_0]; E_\infty))}.
		\end{equation}
		Moreover, by Proposition~\ref{prop:global-well-posedness-En}, for every $n \in \clos \N$, there exists a unique global solution $X_n \in L^p(\Omega; C([0,T]; E_n))$ to~\eqref{eq:Langevin-equation-infinite-dim-A}, which in particular satisfies $X_n = \Phi_{n,T_0}(X_n)$ when restricted to $[0,T_0]$. 
		By~\ref{ass:disc}, this implies $\widetilde X_n = \widetilde \Phi_{n,T_0}(\widetilde X_n)$.
		Hence,
		\begin{align}
			&\norm{X_\infty - \widetilde X_n}{L^p(\Omega; C([0, T_0]; E_\infty))}
			=
			\norm{\Phi_{\infty, T_0}(X_\infty) - \widetilde \Phi_{n, 
					T_0}(\widetilde X_n)}{L^p(\Omega; C([0, T_0]; E_\infty))}
			\\
			&\quad \le 
			\norm{\Phi_{\infty, T_0}(X_\infty) - \widetilde 
				\Phi_{n, T_0}(X_\infty)}{L^p(\Omega; C([0, T_0]; E_\infty))}
			\\&\qquad\quad +
			\norm{\widetilde \Phi_{n, T_0}(X_\infty) - \widetilde 
				\Phi_{n,T_0}(\widetilde 
				X_n)}{L^p(\Omega; C([0, T_0]; E_\infty))}
			\\
			&\quad \le 
			\norm{\Phi_{\infty, T_0}(X_\infty) - \widetilde 
				\Phi_{n, T_0}(X_\infty)}{L^p(\Omega; C([0, T_0]; E_\infty))}
			\\&\qquad \quad +
			c \norm{X_\infty - \widetilde X_n}{L^p(\Omega; C([0, T_0]; E_\infty))},
		\end{align}
		so that Lemmas~\ref{lem:ICterm-convs}--\ref{lem:deterministic-convs} 
		and Proposition~\ref{prop:d2c-pathwise-conv-Wdelta} yield (using all of the Assumptions~\ref{ass:disc}--\ref{ass:A-conv},~\ref{item:F-globalLip-linearGrowth}--\ref{item:F-approx} and~\ref{ass:IC}):
		\begin{equation}\label{eq:Thm32-1}
			\begin{aligned}
				&\norm{X_\infty - \widetilde X_n}{L^p(\Omega; C([0, T_0]; E_\infty))}
				\\&\quad \le 
				\frac{1}{1-c} \, \norm{\Phi_{\infty, T_0}(X_\infty) - \widetilde \Phi_{n, T_0}(X_\infty)}{L^p(\Omega; C([0, T_0]; E_\infty))} \to 0
			\end{aligned}
		\end{equation}
		as $n \to \infty$.
		In order to extend the convergence to arbitrary time horizons, we write
		\begin{align}
			\norm{X_\infty - \widetilde X_n}{L^p(\Omega; C([0, \frac{3}{2}T_0]; E_\infty))}
			&\le 
			\norm{X_\infty - \widetilde X_n}{L^p(\Omega; C([0, \frac{1}{2}T_0]; E_\infty))} \\
			&\quad +
			\norm{X_\infty - \widetilde X_n}{L^p(\Omega; C([\frac{1}{2}T_0, \frac{3}{2}T_0]; E_\infty))}.
		\end{align}
		The first term tends to zero as $n \to 
		\infty$ by~\eqref{eq:Thm32-1}. As for the second term, we note that 
		$X_n \vert_{[\frac{1}{2}T_0,\frac{3}{2}T_0]}$
		are the respective solutions to~\eqref{eq:Langevin-equation-infinite-dim-A} with shifted drift functions 
		$(F_n(\,\cdot\, + \frac{1}{2}T_0, \,\cdot\,))_{n\in\clos\N}$ and 
		initial values $(X_n(\frac{1}{2}T_0))_{n\in\clos\N}$. Since the Lipschitz 
		constants of the 
		fixed point operators defined above did not depend on the initial datum and 
		only 
		depended on $F$ through its (time-independent) Lipschitz constant $L_F$, 
		we can repeat the same argument to find that the second term tends to zero. 
		Proceeding by 
		induction, we obtain the convergence 
		$\widetilde X_n \to X_\infty$ in $L^p(\Omega; C([0, (1 + \frac{k}{2})T_0]; E_\infty))$
		for any $k \in \N$, and thus in $L^p(\Omega; C([0,T]; E_\infty))$ for any $T \in (0,\infty)$.
	\end{proof}
	
	\subsection{Locally Lipschitz nonlinearities}
	\label{sec:locLip}
	
	In this section, we work under the weaker assumption that the drift coefficients 
	$(F_n)_{n\in\clos\N}$ are \emph{locally} Lipschitz continuous, with local-Lipschitz constants uniformly bounded in $t \in [0,T]$, $\omega \in \Omega$ and $n \in \clos\N$.
	Moreover, we replace the uniform linear growth condition from~\ref{item:F-globalLip-linearGrowth} by the assumption that $F_n(t, \omega, 0)$ is bounded, again uniformly in $t$, $\omega$ and $n$; we will also call this notion of boundedness local since it only involves $u = 0$.
	Thus, we assume that the $(F_n)_{n\in\clos\N}$ are locally uniformly Lipschitz and locally uniformly bounded:
	\begin{enumerate}[label=(F{\arabic*}$'$)]
		\item\label{item:F-locLip-unifBdd} 
		For every $r \in (0,\infty)$ there exists a constant $L_F^{(r)} \in 
		(0,\infty)$ such that
		for almost every $(\omega, t) \in \Omega \times [0,T]$, all 
		$n \in \clos\N$ and every $x_n, y_n \in E_n$ such that $\norm{x_n}{E_n}, 
		\norm{y_n}{E_n} \le r$, we have 
		\begin{equation}
			\norm{F_n(\omega, t, x_n) - F_n(\omega, t, y_n)}{E_n}
			\le 
			L_F^{(r)} \norm{x_n - y_n}{E_n}.
		\end{equation}
		Moreover, for every $x_n \in E_n$, $n \in \clos\N$ the process $(\omega, t) \mapsto F_n(\omega, t, x_n)$ is strongly measurable and adapted, and there exists a constant $C_{F, 0}$ such that 
		\begin{equation}
			\norm{F_n(\omega, t, 0)}{E_n} \le C_{F,0}
			\quad \text{for all }
			n \in \clos\N.
		\end{equation}
	\end{enumerate}
	Under these conditions, we can in general not expect to obtain global solutions of~\eqref{eq:Langevin-equation-infinite-dim-A} in the sense of Definition~\ref{def:mildsol}. 
	Instead, we need to work with locally defined $E_n$-valued stochastic processes, i.e., with mappings of the form 
	\begin{equation}\label{eq:locally-defined-process}
		Y \from \{ (\omega, t) \in \Omega \times [0,T] : t \in [0, \tau(\omega)) 
		\}
		\to E_n
	\end{equation}
	for some stopping time $\tau \from \Omega \to [0,T]$. We denote such a process 
	by~$Y = (Y(t))_{t \in [0, \tau)}$. 
	If the half-open interval $[0, \tau(\omega))$ in~\eqref{eq:locally-defined-process} is replaced by $[0, \tau(\omega)]$, then we write $Y = (Y(t))_{t\in[0,\tau]}$ instead. 
	We say that $(Y(t))_{t \in [0, \tau)}$ is \emph{admissible} if 
	\begin{itemize}[series=mildsol-loc-ass]
		\item for all $t \in [0,T]$, the mapping $\{\omega \in \Omega : t < \tau(\omega) \} \ni \omega \mapsto Y(\omega, t) \in E_n$ is $\cF_t$-measurable;
		\item the mapping $[0, \tau(\omega)) \ni t \mapsto Y(\omega, t) \in E_n$ is continuous, $\bbP$-a.s.
	\end{itemize}
	We denote by $V^{\mathrm{loc}}([0,\tau) \times \Omega; E_n)$ the space of admissible $E_n$-valued processes $(Y(t))_{t \in [0,\tau)}$ for which there exists a sequence $(\tau_m)_{m\in\N}$ of stopping times such that, for $\bbP$-a.e.\ $\omega \in \Omega$, we have $\tau_m(\omega) \uparrow \tau(\omega)$ as $m \to \infty$ and $\norm{Y}{C([0, \tau_m(\omega)]; E_n)} < \infty$ for all $m \in \N$.
	As in~\cite[Section~8]{vNVW2008}, we define local solutions
	to~\eqref{eq:Langevin-equation-infinite-dim-A} as follows:
	\begin{definition}\label{def:mildsol-loc}
		An admissible $E_n$-valued stochastic process $X_n = (X_n(t))_{t \in [0,\tau)}$ is said to be a local solution to~\eqref{eq:Langevin-equation-infinite-dim-A} with coefficients $(A_n, F_n, \xi_n)$ if there exists a sequence $(\tau_m)_{m\in\N}$ of stopping times such that 
		$\tau_m \uparrow \tau$ as $m \to \infty$, $\bbP$-a.s., and for all $m \in 
		\N$ we have
		\begin{enumerate}
			\item for every $t \in [0,T]$, the process $(\omega, s) \mapsto S_n(t-s) F_n(\omega, s, X_n(\omega, s)) \indic{[0,\tau_m]}(s)$ belongs to $L^0(\Omega; L^1(0,t; E_n))$;
			\item 
			for every $t \in [0,T]$, $s \mapsto S_n(t-s)\indic{[0,\tau_m]}(s) \in L^2(0,t; \gamma(H_n; E_n))$ ;
			\item it holds $\bbP$-a.s.\ that for all $t \in [0,\tau_m]$, we have 
			\begin{align}
				X_n(t) = S_n(t)\xi_n &+ \int_0^t S_n(t-s)F_n(s, X_n(s)) 
				\indic{[0,\tau_m]}(s) \rd s \\&+ 
				\int_0^t S_n(t-s) \indic{[0,\tau_m]}(s) \rd W_n(s).
			\end{align}
		\end{enumerate}
		We say that a local solution $(X_n(t))_{t\in[0,\tau)}$ 
		to~\eqref{eq:Langevin-equation-infinite-dim-A} is \emph{maximal} if for any 
		other 
		local solution $(\overline X_n(t))_{t \in [0, \overline\tau)}$ it holds 
		$\bbP$-a.s.\ that 
		$\overline \tau \le \tau$ and $X_n|_{[0,\overline\tau)} \equiv \overline 
		X_n$. It is called \emph{global} if $\tau = T$ holds $\bbP$-a.s.\ and there 
		exists a 
		solution $(\widehat X_n(t))_{t\in[0,T]}$ 
		to~\eqref{eq:Langevin-equation-infinite-dim-A} in the sense of 
		Definition~\ref{def:mildsol} such that $\widehat X_n|_{[0, \tau)} 
		\equiv X_n$, $\bbP$-a.s.
		The stopping time $\tau$ is called an 
		\emph{explosion time} if 
		\begin{equation}\label{eq:explosion-time}
			\limsup\nolimits_{t \uparrow \tau(\omega)}
			\norm{X_n(\omega, t)}{E_n} = \infty
			\quad 
			\text{for a.e.\ } \omega \in \Omega \text{ such that } 
			\tau(\omega) < T.
		\end{equation}
	\end{definition}
	The following local well-posedness result then follows from~\cite[Theorem~8.1]{vNVW2008}:
	\begin{theorem}[{\cite[Theorem~8.1]{vNVW2008}}]\label{thm:local-well-posedness-En}
		Suppose that Assumptions~\ref{ass:disc},~\ref{ass:operators} and~\ref{item:F-locLip-unifBdd} are satisfied, and let $n \in \clos\N$, $p \in [1, \infty)$, $\xi_n \in L^p(\Omega, \cF_0, \bbP; E_n)$ and $T \in (0,\infty)$ be given. 
		Then~\eqref{eq:Langevin-equation-infinite-dim-A} has a unique maximal local mild solution $(X_n(t))_{t \in [0,\sigma_n)}$ in $V^{\mathrm{loc}}([0, \sigma_n) \times \Omega; E_n)$, where $\sigma_n \from \Omega \to [0,T]$ is an explosion time. 
	\end{theorem}
	
	Combined with the convergence assumptions~\ref{ass:IC} and~\ref{item:F-approx}, we can argue analogously to~\cite[Theorem~3.3 and~Corollary~3.4]{KvN2012} to derive the following extension of Theorem~\ref{thm:d2c-semilinear-global-linear} to the present setting.
	
	\begin{theorem}\label{thm:d2c-conv-locLip-unifBound}
		Suppose that Assumptions~\ref{ass:disc},~\ref{ass:operators},~\ref{ass:IC},~\ref{item:F-locLip-unifBdd}
		and~\ref{item:F-approx} are satisfied. 
		For $n \in \clos\N$, let $(X_n(t))_{t \in [0,\sigma_n)}$ be the maximal 
		local solution to~\eqref{eq:Langevin-equation-infinite-dim-A}
		with explosion time $\sigma_n \from \Omega \to [0,T]$, and set
		$\widetilde X_n \deq \Lambda_n X_n$.
		Then the following is true:
		\begin{enumerate}[series=d2c-conv-locLip-unifBound]
			\item\label{item:d2c-conv-locLip-unifBound-sigmaConv}%
			We have $\widetilde X_n \indic{[0, \sigma_\infty \wedge \sigma_n)} \to X_\infty \indic{[0,\sigma_\infty)}$ in $L^0(\Omega \times [0,T]; E_\infty)$ as $n \to \infty$.
		\end{enumerate}
		If, moreover, $\sigma_n = T$ holds
		$\bbP$-a.s.\ for all $n \in \N$ and $p \in [1,\infty)$ is such 
		that 
		\begin{equation}\label{eq:Xn-unifBdd-LpOmegaCHn}
			\sup_{n\in\N} \norm{X_n}{L^p(\Omega; C([0,T]; E_n))} < \infty, 
		\end{equation}
		then the following assertions also hold:
		\begin{enumerate}[resume=d2c-conv-locLip-unifBound]
			\item\label{item:d2c-conv-locLip-unifBound-globalInfty} We have 
			$\sigma_\infty = T$, $\bbP$-a.s.
			\item\label{item:d2c-conv-locLip-unifBound-convInCty} If $p \in 
			(1,\infty)$, then for all $p^- \in [1, p)$ we have 
			\begin{equation}
				\widetilde X_n \to X_\infty 
				\quad 
				\text{ in } 
				L^{p^-}(\Omega; C([0,T]; E_\infty))
				\text{ as } 
				n \to \infty.
			\end{equation}
		\end{enumerate}
	\end{theorem}
	Similarly to~\cite[Theorem~3.3 and~Corollary~3.4]{KvN2012}, the proof of Theorem~\ref{thm:d2c-conv-locLip-unifBound} relies on the following general approximation results for locally defined processes:
	\begin{theorem}[{\cite[Theorem~2.1 and Corollary~2.5]{KvN2012}}]\label{thm:approx-local-processes}
		Let $(E, \norm{\,\cdot\,}{E})$ be a real and separable Banach space and $T \in (0,\infty)$.
		For every $n \in \clos\N$, suppose that $(Y_n(t))_{t\in[0,\sigma_n)}$ is a continuous and adapted $E$-valued locally defined process with explosion time $\sigma_n \from \Omega \to (0, T]$, and define the stopping times $\rho_n^{(r)} \from \Omega \to [0,T]$ by
		\begin{equation}\label{eq:def-rho-n-r}
			\rho_n^{(r)} \deq 
			\inf \{ t \in (0, \sigma_n) : \norm{Y_n(t)}{E} > r 
			\}, 
			\quad 
			r \in (0,\infty),
		\end{equation}
		with the convention that $\inf \emptyset \deq T$. 
		Moreover, suppose that for each $r \in (0,\infty)$ there exists a (globally defined) continuous and adapted $E$-valued process $(Y_n^{(r)}(t))_{t\in[0,T]}$ which satisfies the following two conditions:
		\begin{enumerate}[label=\normalfont(\alph*)]
			\item\label{item:thm:approx-local-processes-a}%
			For all $n \in \clos \N$ and $r \in (0,\infty)$, it holds $\bbP$-a.s.\ that 
			\begin{equation}
				Y_n^{(r)} \indic{[0,\rho_n^{(r)}]}
				\equiv 
				Y_n \indic{[0,\rho_n^{(r)}]}
				\quad 
				\text{ on } 
				[0,T];
			\end{equation}
			\item\label{item:thm:approx-local-processes-b}%
			For all $r \in (0,\infty)$ we have 
			\begin{equation}
				Y_n^{(r)} \to Y_\infty^{(r)}
				\quad 
				\text{ in } 
				L^0(\Omega; C([0,T]; E))
				\text{ as } 
				n \to \infty.
			\end{equation}
		\end{enumerate}
		Then the following assertions hold:
		\begin{enumerate}[series=approx-local-processes-assertions]
			\item\label{item:thm:approx-local-processes-i}%
			For all $r \in (0,\infty)$ and $\varepsilon > 0$ it holds $\bbP$-a.s.\ that
			\begin{equation}
				\liminf_{n\to\infty} \rho_n^{(r)}
				\le 
				\rho_\infty^{(r)}
				\le 
				\limsup_{n\to\infty} \rho_n^{(r + \varepsilon)}.
			\end{equation}
			\item\label{item:thm:approx-local-processes-ii}%
			For all $r \in (0,\infty)$ and $\varepsilon > 0$, we have
			\begin{equation}
				Y_n \indic{[0, \rho_\infty^{(r)} \wedge \rho_n^{(r + \varepsilon)})} 
				\to Y_\infty \indic{[0, \rho_\infty^{(r)})}
				\quad 
				\text{ in }
				L^0(\Omega; B_{\mathrm b}([0,T]; E))
				\text{ as } 
				n \to \infty,
			\end{equation}
			where $B_{\mathrm b}([0,T]; E)$ denotes the space of bounded and strongly measurable functions from $[0,T]$ to $E$.
			\item\label{item:thm:approx-local-processes-iii}%
			We have
			\begin{equation}
				Y_n \indic{[0, \sigma_\infty \wedge \sigma_n)}
				\to 
				Y_\infty \indic{[0,\sigma_\infty)} \quad 
				\text{ in }
				L^0(\Omega \times [0,T]; E)
				\text{ as } 
				n \to \infty.
			\end{equation}
		\end{enumerate}
		If, in addition, we have $\bbP(\sigma_n = T) = 1$ for all $n \in \N$, and $p \in [1,\infty)$ is such that 
		\begin{equation}
			\sup_{n\in\N} \norm{Y_n}{L^p(\Omega; C([0,T]; E))} < \infty,
		\end{equation}
		then the following assertions also hold:
		\begin{enumerate}[resume=approx-local-processes-assertions]
			\item\label{item:thm:approx-local-processes-iv}%
			We have $\bbP(\sigma_\infty = T) = 1$ and $X_\infty \in L^q(\Omega; C([0,T]; E))$.
			\item\label{item:thm:approx-local-processes-v}%
			If $p \in (1,\infty)$, then for all $p^- \in [1, p)$ we have
			\begin{equation}
				Y_n \to Y_\infty 
				\quad 
				\text{in } L^{p^-}(\Omega; C([0,T]; E)) \text{ as } n \to \infty.
			\end{equation}
		\end{enumerate}
	\end{theorem}
	\begin{proof}[Proof of Theorem~\ref{thm:d2c-conv-locLip-unifBound}]
		For every $n \in \clos \N$ and $r \in (0,\infty)$, let us define the mapping
		$F_n^{(r)} \from \Omega \times [0,T] \times E_n \to E_n$ by
		\begin{equation}\label{eq:globLip-retract}
			F_n^{(r)}(\omega, t, x_n) \deq 
			\begin{cases}
				F_n(\omega, t, x_n), &\text{if } \norm{\Lambda_n x_n}{E_\infty} \le 
				r, \\
				F_n\Bigl(\omega, t, \frac{r x_n}{\norm{\Lambda_n x_n}{E_\infty}}\Bigr), 
				&\text{otherwise}.
			\end{cases}
		\end{equation}
		For any fixed $r > 0$, the sequence 
		$(F_n^{(r)})_{n\in \clos \N}$ satisfies 
		conditions~\ref{item:F-globalLip-linearGrowth} and~\ref{item:F-approx}.
		Indeed, to establish the former, we first note that $F_n^{(r)}$ can be 
		written as 
		\begin{equation}\label{eq:Fnr-identity}
			F_n^{(r)}(\omega, t, x_n) 
			=
			F_n(\omega, t, \Pi_n R_r(\Lambda_n x_n)),
		\end{equation}
		where $R_r \from E_\infty \to E_\infty$ denotes the canonical retraction of $E_\infty$ onto the closed ball around $0 \in E_\infty$ with radius $r$:
		\begin{equation}
			R_r(x) \deq \begin{cases}
				x, &\text{if } \norm{x}{E_\infty} \le r; \\
				\frac{r x}{\norm{x}{E_\infty}}, &\text{otherwise}.
			\end{cases}
		\end{equation}
		An elementary estimate shows that $R_r$ is Lipschitz with constant $2$.
		It follows that $(F_n^{(r)})_{n\in\clos \N}$ is uniformly 
		globally Lipschitz with constant $L_{F^{(r)}} \le L_F^{(r M_\Pi)} M_\Pi 
		M_\Lambda$, and thus of linear growth with uniform constant 
		$C_{F^{(r)}} \le \max\{ L_{F^{(r)}}, C_{F,0}\}$, so that it 
		satisfies~\ref{item:F-globalLip-linearGrowth}.
		
		In order to show that $(F_n^{(r)})_{n\in 
			\clos \N}$ satisfies~\ref{item:F-approx},
		first note that for every sequence $(y_n)_{n\in\N} \subseteq E_\infty$ converging to some $y$ 
		in~$E_\infty$, we have $\widetilde F_n(\omega, t, y_n) \to F_\infty(\omega, t, y)$.
		Indeed, by the triangle inequality it suffices to note that
		\begin{equation}
			\norm{\widetilde F_n(\omega, t, y_n) - 
				\widetilde F_n(\omega, t, y)}{E_\infty} \to 0
			\quad 
			\text{and}
			\quad
			\norm{\widetilde F_n(\omega, t, y) - F_\infty(\omega, t, y)}{E_\infty}
			\to 
			0  
		\end{equation}
		as $n \to \infty$, respectively because $(\widetilde 
		F_n)_{n\in\clos\N}$ is 
		uniformly locally 
		Lipschitz (with constants $L_{\widetilde F}^{(r)} \le M_\Pi M_\Lambda 
		L_F^{(r)}$) and since~\ref{item:F-approx} was assumed for $(F_n)_{n\in\N}$.
		Writing
		\begin{equation}
			\widetilde F_n^{(r)}(\omega, t, x)
			=
			\widetilde F_n(\omega, t, R_r(\Lambda_n \Pi_n x)),
		\end{equation}
		see~\eqref{eq:Fnr-identity}, we can apply the above observation to the 
		sequence $y_n \deq R_r(\Lambda_n \Pi_n x)$, which converges 
		to $y \deq R_r(x)$ in $E_\infty$ as $n \to \infty$ in view of 
		Assumption~\ref{ass:disc3}\ref{ass:disc:approx-id3} and the (Lipschitz) 
		continuity of $R_r$. Therefore, we find 
		$\widetilde F_n^{(r)}(\omega, t, x) \to F^{(r)}_\infty(\omega, t, x)$, thus 
		proving 
		the claim that~\ref{item:F-approx} holds 
		for~$(F_n^{(r)})_{n\in\clos\N}$ as well.
		
		For each $r > 0$, condition~\ref{item:F-globalLip-linearGrowth} for 
		$(F_n^{(r)})_{n\in\clos \N}$ yields the existence of a unique global 
		solution
		$(X_n^{(r)}(t))_{t\in[0,T]}$ to~\eqref{eq:Langevin-equation-infinite-dim-A} 
		with coefficients $(A_n, F_n^{(r)}, \xi_n)$.
		In order to establish statement~\ref{item:d2c-conv-locLip-unifBound-sigmaConv} of the present theorem, we will apply the corresponding parts~\ref{item:thm:approx-local-processes-i}--\ref{item:thm:approx-local-processes-iii} of Theorem~\ref{thm:approx-local-processes} to the processes $Y_n \deq \widetilde X_n$; hence we need to verify its conditions~\ref{item:thm:approx-local-processes-a} and~\ref{item:thm:approx-local-processes-b} for $(\widetilde X_n)_{n\in\clos\N}$.
		First note that we have
		$\rho_n^{(r)} \le \sigma_n$, where $\rho_n^{(r)}$ is defined by~\eqref{eq:def-rho-n-r}, and that the restrictions of
		$X_n^{(r)}$
		and $X_n$ to $[0, \rho_n^{(r)})$ are local solutions 
		to~\eqref{eq:Langevin-equation-infinite-dim-A} with coefficients 
		$(A_n, F_n^{(r)}, \xi_n)$ and $(A_n, F_n, \xi_n)$, respectively. Since it 
		holds $\bbP$-a.s.\ that
		$\norm{\Lambda_n X_n(t)}{E_\infty} \le r$ for $t \in [0, \rho_n^{(r)})$, 
		we find
		\[
		F_n(\,\cdot\, , X_n) \equiv F_n^{(r)}(\,\cdot\, , X_n)
		\quad 
		\text{ on } 
		[0,\rho_n^{(r)}), 
		\quad 
		\bbP\text{-a.s.},
		\]
		hence $(X_n(t))_{t\in[0,\rho_n^{(r)})}$ is 
		in fact also a local solution of the equation with coefficients
		$(A_n, F_n^{(r)}, \xi_n)$. Therefore, the local uniqueness 
		of~\eqref{eq:Langevin-equation-infinite-dim-A} 
		(cf.~\cite[Lemma~8.2]{vNVW2008}) implies that 
		$X_n^{(r)} \equiv X_n$ on $[0, \rho_n^{(r)}]$ holds $\bbP$-a.s., and 
		applying $\Lambda_n$ on both sides 
		verifies~\ref{item:thm:approx-local-processes-a}.
		Condition~\ref{item:thm:approx-local-processes-b} follows by applying Theorem~\ref{thm:d2c-semilinear-global-linear} to $(F_n^{(r)})_{n\in\clos \N}$, proving~\ref{item:d2c-conv-locLip-unifBound-sigmaConv}.
		
		Finally, since $(\Lambda_n)_{n\in\N}$ is uniformly bounded in view of 
		Assumption~\ref{ass:disc3}\ref{ass:disc:proj-emb-bounds3}, we see 
		that~\eqref{eq:Xn-unifBdd-LpOmegaCHn} implies that the conditions
		of Theorem~\ref{thm:approx-local-processes}\ref{item:thm:approx-local-processes-iv} and~\ref{item:thm:approx-local-processes-v} are satisfied, which directly yields
		the remaining assertions~\ref{item:d2c-conv-locLip-unifBound-globalInfty} 
		and~\ref{item:d2c-conv-locLip-unifBound-convInCty}.
	\end{proof}
	
	\section{Reaction--diffusion type equations}
	\label{sec:reaction-diffusion}
	
	In this section, we introduce another family of Banach spaces $(B_n)_{n\clos\N}$ such that each $B_n$ embeds into $E_n$ and $\widetilde B \supseteq B_\infty$ (along with other assumptions, given in Subsection~\ref{sec:d2c-Nemytskii}), and consider the $B_n$-valued counterparts of~\eqref{eq:Langevin-equation-infinite-dim-A}.
	The main purpose of this setting is to eventually specialize to the class of stochastic reaction--diffusion 
	type equations which are formally given, for any $n \in \clos \N$, by
	\begin{equation}\label{eq:reaction-diffusion}
		\left\lbrace
		\begin{aligned}
			\rd X_n(t, x) &= -A_n X_n(t, x)\rd t + f_n(t, X_n(t, x)) \rd t + \rd 
			W_n(t, x),
			\\
			X_n(0, x) &= \xi_n(x),
		\end{aligned}
		\right.
	\end{equation}
	where $t \in (0,T]$, $T \in (0,\infty)$, $x \in \cD_n \subseteq \R^d$ and $f_n \from \Omega \times [0,T] \times \R \to \R$ is a locally Lipschitz (real-valued) function. 
	This problem amounts to letting the drift $F_n$ in~\eqref{eq:Langevin-equation-infinite-dim-A} be a Nemytskii operator (also called a \emph{superposition operator}), i.e., defining it by
	\begin{equation}\label{eq:Nemytskii}
		[F_n(\omega, t, u)](x) \deq f_n(\omega, t, u(x)), 
		\quad 
		(\omega,t,x) \in \Omega \times [0,T] \times \cD_n,
	\end{equation}
	for a given function $u \from \cD_n \to \R$. However, in order for a Nemytskii operator $F_n$ to 
	inherit the local Lipschitz continuity from $f_n$, we cannot view it as acting on the UMD-type-2 Banach space $E_n = L^q(\cD_n)$ with $q \in [1,\infty)$. 
	In fact, by~\cite[Theorem~3.9]{AppellZabrejko1990}, the operator $u \mapsto F_n(\omega, t, u)$ defined by~\eqref{eq:Nemytskii} is weakly continuous on $L^q(\cD_n)$ (meaning that it maps weakly convergent sequences to weakly convergent sequences) if and only if it is affine in $u$, i.e., there exist coefficients $a_n(\omega, t), b_n(\omega, t) \in \R$ such that $[F_n(\omega, t, u)](x) = a_n(\omega, t) + b_n(\omega, t)u(x)$ for all $x \in \cD_n$.
	In particular, if $(F_n)_{n\in\clos\N}$ is a family of Nemytskii operators which is uniformly locally Lipschitz in the sense of~\ref{item:F-locLip-unifBdd} on the spaces $E_n = L^q(\cD_n)$, then it is in fact globally Lipschitz and of linear growth in the sense of~\ref{item:F-globalLip-linearGrowth}.
	
	Thus, we will instead view $F_n$ as an operator on $B_n \deq C(\clos{\cD_n})$ (which coincides with $B_n = L^\infty(\cD_n)$ if $\cD_n$ is discrete).
	This, in turn, poses a difficulty for stochastic evolution equations, since there is no theory for the stochastic integration of integrands taking their values in a space of continuous functions. 
	This is due to the poor geometric properties of $(C(\clos{\cD_n}), \norm{\,\cdot\,}{\infty})$ as a Banach space:
	The most general notion of stochastic integration in Banach spaces (see~\cite{vNVW2015}) requires at least the UMD property; such spaces are, in particular, reflexive~\cite[Theorem~4.3.3]{HvNVWVolumeI}, which $C(\clos{\cD_n})$ fails to be.
	
	One way to circumvent this issue is to proceed as in~\cite[Section~3.2]{KvN2012}; 
	namely, defining the fractional domain spaces
	\begin{equation}
		\dot E_n^\alpha \deq \dom{A_n^{\alpha/2}}, 
		\quad 
		\norm{x_n}{\dot E_n^\alpha} \deq \norm{(\id_n + A_n)^{\alpha/2} x_n}{E_n},
	\end{equation}
	and supposing that 
	$\dot E_n^{\theta} \emb C(\clos{\cD_n}) \emb L^q(\cD_n)$
	continuously and densely for some $\theta \in [0, 1)$,
	one can carry out the stochastic 
	integration in the space $\dot E_n^{\theta}$, while working with 
	$C(\clos{\cD_n})$-valued processes for the fixed-point arguments.
	
	In applications, we typically assume that $A_n$ is an (unbounded) linear 
	differential 
	operator on $L^q(\cD_n)$, where $q \in [2,\infty)$, augmented with some boundary conditions 
	(``b.c.'')\ such 
	that $\dot E^{\alpha}_n$ is the fractional Sobolev space 
	$W^{\alpha, q}_{\text{b.c.}}(\cD_n)$ of order $\alpha$. We 
	then suppose that $\theta$ is chosen large enough in 
	relation to the dimension $d$ that we have the continuous and dense Sobolev 
	embedding $W^{\alpha, q}_{\text{b.c.}}(\cD_n) \hookrightarrow  C_{\text{b.c.}}(\clos{\cD_n})$.
	
	In Section~\ref{sec:d2c-Nemytskii} we will specify the abstract 
	formulation of the setting outlined above, as well as some additional 
	uniformity conditions with respect to $n \in \clos\N$. 
	These will be used to, as a first step, derive $B_n$-valued counterparts to the $E_n$-valued discrete-to-continuum approximation results for globally Lipschitz drifts of linear growth from Subsection~\ref{sec:globalLip-linGrowth}; in Subsection~\ref{sec:locLip-B} we do the same for the $B_n$-valued setting with locally Lipschitz and locally bounded drifts.
	In the latter case, the solution are local in general, but in 
	Section~\ref{sec:global-wellposed-RDeq} we state an extra dissipativity 
	assumption on $F_n$ under which the existence of global solutions 
	to~\eqref{eq:reaction-diffusion} has been proven in~\cite[Section~4]{KvN2012}.
	These processes then also converge in an improved sense, and we can apply this to Section~\ref{sec:graph-WM}.
	\subsection{Setting and convergence for globally Lipschitz drifts}
	\label{sec:d2c-Nemytskii}
	We start by specifying the abstract setting for the treatment of 
	reaction--diffusion type equations which was outlined at the beginning of this 
	section. 
	That is, we complement the Hilbert spaces $(H_n)_{n\in\clos\N}$ and UMD-type-2 Banach spaces $(E_n)_{n\in\clos\N}$ from 
	Section~\ref{sec:semilinear-L2} with a sequence of real separable Banach spaces 
	$(B_n, \norm{\,\cdot\,}{B_n})_{n\in\clos\N}$, embedded continuously and densely 
	into $E_n$ for each $n \in \clos \N$.
	Moreover, we introduce the real Banach space $(\widetilde B, \norm{\,\cdot\,}{\widetilde B})$, containing $B_\infty$ as a closed subspace, and we suppose that all the $B_n$ are embedded by into $\widetilde B$ by the lifting operators $(\Lambda_n)_{n\in\N}$ from~\ref{ass:disc}. 
	The spaces $B_\infty \subseteq \widetilde B$ should respectively be thought of as $C(\clos \cD) \subseteq L^\infty(\cD)$. 
	More precisely, we will work with the following extensions of assumptions~\ref{ass:disc}--\ref{ass:A-conv}:
	\begin{enumerate}[label=(A{\arabic*}-B), leftmargin=1.4cm, series=ass-A-B]
		\item\label{ass:disc2}
		Assumption~\ref{ass:disc} holds, and Assumption~\ref{ass:disc3} is 
		satisfied for 
		$(B_n, \norm{\,\cdot\,}{B_n})_{n \in \clos \N}$ 
		and $(\widetilde B, \norm{\,\cdot\,}{B_\infty})$, with the same projection 
		and lifting 
		operators $(\Pi_n)_{n\in\clos\N}$, $(\Lambda_n)_{n\in\clos\N}$ 
		from~\ref{ass:disc}, for which we set
		$\widetilde M_{\Pi} \deq \sup_{n\in\N} 
		\norm{\Pi_n}{\LO(\widetilde B; B_n)}$
		and 
		$\widetilde M_{\Lambda} \deq \sup_{n\in\N} 
		\norm{\Lambda_n}{\LO(B_n; \widetilde B)}$.
		\item\label{ass:semigroup-B2}%
		Assumption~\ref{ass:operators} holds, the semigroup $(S_n(t))_{t\ge0} \subseteq \LO(E_n)$ leaves $B_n$ invariant for all $n \in \clos\N$, and its restriction $(S_n(t)|_{B_n})_{t\ge0}$ to $B_n$ is a strongly continuous semigroup in $\LO(B_n)$.
		Moreover, there exists an $\widetilde M_{S} \in [1,\infty)$ such that
		\begin{equation}\label{eq:unif-bdd-Sn-Bn}
			\norm{S_n(t)}{\LO(B_n)}
			\le 
			\widetilde M_{S}
			\quad 
			\text{for all }
			n \in \clos\N \text{ and }
			t \in [0,\infty).
		\end{equation}
		\item \label{ass:A-conv-B} Assumption~\ref{ass:A-conv} holds, and $\widetilde R_n x \to R_\infty x$ in 
		$\widetilde B$ as $n \to \infty$ for all $x\in B_\infty$.
	\end{enumerate}
	By~\cite[Chapter~II, Proposition~2.3]{EngelNagel2000}, assumptions~\ref{ass:disc2} and~\ref{ass:semigroup-B2} imply that the generator of $(S_n(t)|_{B_n})_{t\ge0}$ is the operator $-A_n \vert_{B_n} \from \dom{A_n \vert_{B_n}} \subseteq B_n \to B_n$ 
	defined by 
	\[
	-A_n \vert_{B_n} x_n \deq -A_n x_n
	\quad \text{on } 
	\dom{A_n \vert_{B_n}} \deq \{ x_n \in B_n \cap \dom{A_n} : A_n x_n \in 
	B_n\},
	\]
	which is known as \emph{the part of $-A_n$ in~$B_n$}.
	Therefore, by Theorem~\ref{thm:d2c-TKapprox}, assumption~\ref{ass:A-conv-B} implies $\widetilde S_n|_{B_n} \otimes x \to S|_{B_\infty} \otimes x$ in $C([0,T]; \widetilde B)$, as $n \to \infty$, 
	for all $x \in B_\infty$ and $T \in (0,\infty)$.
	
	The following \emph{uniform ultracontractivity} assumption is necessary in order to circumvent the aforementioned 
	problem 
	regarding stochastic integration in 
	arbitrary separable Banach spaces $B_n$. 
	It replaces assumption (A3) in~\cite{KvN2012}, which forces all the spaces $(\dot E_n^{\alpha})_{n \in \clos\N}$ to essentially be the same, which is not satisfied in applications such as discrete-to-continuum approximation, where each $\dot E_n^{\alpha}$ consists of functions defined on a different domain.
	\begin{enumerate}[label=(A{\arabic*}-B), resume=ass-A-B, leftmargin=1.35cm]
		\item\label{ass:emb-Htheta-B-H2}%
		There exist $\theta \in [0, 1)$ and $M_\theta \in [0,\infty)$ such that, for all $n \in \clos \N$, we have
		$\dot E^{\theta}_n \emb B_n \emb E_n$ continuously and densely, with 
		\begin{equation}
			\norm{S_n(t)x_n}{B_n} \le M_{\theta} t^{-\theta / 2} \norm{x_n}{E_n}
			\quad 
			\text{for all }
			x_n \in E_n \text{ and } t \in [0,\infty).
		\end{equation}
	\end{enumerate}
	The uniformity in $n$ of the constant $M_\theta$ enables us to prove the following discrete-to-continuum approximation result for the stochastic convolutions $(W_{A_n})_{n\in\clos\N}$ as $B_n$-valued processes (i.e., a $B_n$-valued counterpart to Proposition~\ref{prop:d2c-pathwise-conv-Wdelta}):
	\begin{proposition}\label{prop:d2c-pathwise-conv-Wdelta2-B}
		Let $p \in [1,\infty)$ and $T \in (0,\infty)$ be given, and suppose that Assumptions~\ref{ass:disc2}, \ref{ass:semigroup-B2} and~\ref{ass:emb-Htheta-B-H2} hold, with $\theta + 2\beta < 1$. 
		For every $n \in \clos\N$, there exists a modification of $W_{A_n}$ which belongs to $L^p(\Omega; C([0,T]; B_n))$, and we identify these modifications with the processes $(W_{A_n})_{n\in\clos\N}$ themselves. 
		
		Under the additional Assumption~\ref{ass:A-conv-B}, we have
		\begin{equation}
			\widetilde W_{A_n} \to W_{A_\infty}
			\quad \text{in } L^p(\Omega; C([0,T]; \widetilde B))
			\text{ as } n \to \infty.
		\end{equation}
	\end{proposition}
	\begin{proof}
		Fix a $\beta' \in (\beta, \frac{1}{2})$ such that $\theta + 2\beta' < 1$, where $\beta \in [0, \frac{1}{2})$ is as in~\ref{ass:operators}.
		Without loss of generality, we may assume that $p \in (2, \infty)$ is so large 
		that, in fact, $\theta + 2\beta' < 1 - \frac{2}{p}$. 
		As in the proof of Proposition~\ref{prop:d2c-pathwise-conv-Wdelta}, we see that $W_{A_n}^{\frac{1}{2} + \beta'}(\omega, \,\cdot\,)$, where $W_{A_n}^{\frac{1}{2} + \beta'}$ is the auxiliary process defined by~\eqref{eq:def-WAndelta}, belongs to $L^p(0,T; E_n)$ for $\bbP$-a.e.\ $\omega \in \Omega$, hence $\fI_{A_n}^{\frac{1}{2} - \beta'}W_{A_n}^{\frac{1}{2} + \beta'}(\omega, \,\cdot\,) \in C([0,T]; B_n)$	by applying Proposition~\ref{prop:app-basicPropsIA}\ref{prop:app-basicPropsIA:Lp-CHdot} with $E \deq E_n$, $F \deq B_n$ and $\alpha \deq \theta$ from assumption~\ref{ass:emb-Htheta-B-H2}. 
		Moreover, one finds that the process $\fI_{A_n}^{\frac{1}{2} - \beta'}W_{A_n}^{\frac{1}{2} + \beta'}$ is a continuous modification of $W_{A_n}$, belonging to $L^p(\Omega; C([0,T]; B_n))$.
		It now suffices to establish the following, as $n \to \infty$:
		\begin{align}
			\norm{\widetilde{\mathfrak I}_{A_n}^{\frac{1}{2} - \kappa}
				(\widetilde W_{A_n}^{\frac{1}{2} + \kappa} - 
				W_{A_\infty}^{\frac{1}{2} + \kappa})}{L^p(\Omega; 
				C([0,T];\widetilde B))}
			&\to 0, \label{eq:conv-1sttermFact-CB}
			\\
			\norm{
				\widetilde{\mathfrak I}_{A_n}^{\frac{1}{2} - \kappa} 
				W_{A_\infty}^{\frac{1}{2} + \kappa} 
				- 
				\mathfrak I_{A_\infty}^{\frac{1}{2} - \kappa} 
				W_{A_\infty}^{\frac{1}{2} + 
					\kappa}
			}{L^p(\Omega; C([0,T];\widetilde B))}
			&\to 0. \label{eq:conv-2ndtermFact-CB}
		\end{align}
		These convergences follow by arguing as in the proof of 
		Proposition~\ref{prop:d2c-pathwise-conv-Wdelta}, where we now need 
		Corollary~\ref{cor:unifbb-IAtilde}\ref{cor:unifbb-IAtilde:c}
		and 
		Proposition~\ref{prop:strong-continuity-IAntilde}\ref{prop:strong-continuity-IAntilde:b}
		instead of Corollary~\ref{cor:unifbb-IAtilde}\ref{cor:unifbb-IAtilde:b}
		and 
		Proposition~\ref{prop:strong-continuity-IAntilde}\ref{prop:strong-continuity-IAntilde:a},
		respectively.
	\end{proof}
	For the initial data, we consider the following analog to~\ref{ass:IC}:
	\begin{enumerate}[label=(IC-B), leftmargin=1.3cm]
		\item\label{ass:IC-B}%
		There exists $p \in [1,\infty)$ such that $(\xi_n)_{n\in\clos\N} \in
		\prod_{n\in\clos\N} L^p(\Omega; B_n)$ and
		\begin{equation}
			\widetilde \xi_n \to \xi_\infty \quad 
			\text{in } L^p(\Omega; \widetilde B) \text{ as } n \to \infty.
		\end{equation}
	\end{enumerate}
	Finally, 
	regarding the drift coefficients $(F_n)_{n\in\clos\N}$, we now suppose that
	\begin{enumerate}[label=(F{\arabic*}-B), series=Fconditions-B, leftmargin=1.3cm]
		\item\label{item:F-globalLip-linearGrowth-B} 
		Assumption~\ref{item:F-globalLip-linearGrowth} holds with 
		$(B_n)_{n\in\clos\N}$ in place of $(E_n)_{n\in\clos\N}$, and 
		$B_n$-valued
		Lipschitz and growth constants respectively denoted by $\widetilde L_F$ 
		and $\widetilde C_F$.
		\item\label{item:F-approx-B} For almost every 
		$(\omega, t) \in \Omega \times [0,T]$ and every $x \in B_\infty$ 
		we have
		\begin{equation}
			\widetilde F_n(\omega, t, x) 
			\to 
			F_\infty(\omega, t, x)
			\quad 
			\text{ in } \widetilde B
			\text{ as } n \to \infty.
		\end{equation}
	\end{enumerate}
	Note that in the approximation assumptions~\ref{ass:A-conv-B} and~\ref{item:F-approx-B}, we only impose convergence for $x \in B_\infty \subseteq \widetilde B$, and similarly we only consider $\xi_\infty \in L^p(\Omega; B_\infty)$ in~\ref{ass:IC-B}.
	Recall that 
	this is sufficient since Theorem~\ref{thm:d2c-TKapprox}, on which the 
	approximation results ultimately rely, is formulated in this setting.
	
	Under~\ref{ass:disc2}, \ref{ass:semigroup-B2}, \ref{ass:emb-Htheta-B-H2} (with $\theta + 2\beta < 1$), \ref{ass:IC-B} and~\ref{item:F-globalLip-linearGrowth-B}, we can derive well-posedness of \emph{$B_n$-valued} global solutions to~\eqref{eq:Langevin-equation-infinite-dim-A}; these are defined by simply replacing the space $E_n$ by $B_n$ in Definition~\ref{def:mildsol}.
	To this end, we again investigate the fixed-point operators $\Phi_{n,T}$ defined by~\eqref{eq:fixed-point-operator}, viewing them now as acting on $L^p(\Omega; C([0,T]; B_n))$.
	We have the following analog to Proposition~\ref{prop:existence-uniqueness-globLip-linGrowth}:
	\begin{proposition}\label{prop:existence-uniqueness-globLip-linGrowth-B}
		Suppose that Assumptions~\ref{ass:disc2}, \ref{ass:semigroup-B2}, \ref{ass:emb-Htheta-B-H2} hold with $\theta + 2\beta < 1$, and~\ref{item:F-globalLip-linearGrowth-B} is satisfied. 	
		Let $n \in \clos\N$, $p \in [1, \infty)$, $\xi_n \in L^p(\Omega, \cF_0, \bbP; B_n)$ and $T \in (0,\infty)$ be given. 
		The operator $\Phi_{n,T}$ given by~\eqref{eq:fixed-point-operator} is well defined and Lipschitz continuous on $L^p(\Omega; C([0,T]; B_n))$. 
		Its Lipschitz constant is independent of $\xi_n$, depends on $A_n$ and $F_n$ only through $\widetilde M_S$ and $\widetilde L_F$, and tends to zero as $T \downarrow 0$.
	\end{proposition}
	\begin{proof}
		The fact that $S_n \otimes \xi_n \in L^p(\Omega; C([0,T]; B_n))$ is immediate from Assumptions~\ref{ass:disc2}, \ref{ass:semigroup-B2} and $\xi_n \in L^p(\Omega, \cF_0, \bbP; E_n)$.
		By the first part of Proposition~\ref{prop:d2c-pathwise-conv-Wdelta2-B} (which also uses~\ref{ass:emb-Htheta-B-H2}), we find $W_{A_n} \in L^p(\Omega; C([0,T]; B_n))$.
		By~\ref{item:F-globalLip-linearGrowth-B} and Proposition~\ref{prop:app-basicPropsIA}\ref{prop:app-basicPropsIA:Lp-CHdot} (with $E = F = B_n$, $\alpha = 0$ and $s = 1$), we have $S_n * F_n(\,\cdot\,, u_n) \in L^p(\Omega; C([0,T]; B_n))$ for all $u_n \in L^p(\Omega; C([0,T]; B_n))$.
		This shows that $\Phi_{n,T}$ is well-defined.
		A straightforward estimate involving Assumptions~\ref{ass:disc2},~\ref{ass:semigroup-B2} and~\ref{item:F-globalLip-linearGrowth-B} yields, for all $u_n, v_n \in L^p(\Omega; C([0,T]; B_n))$,
		\[
		\norm{\Phi_{n,T}(u_n) - \Phi_{n,T}(v_n)}{L^p(\Omega; C([0,T]; B_n))}
		\le \widetilde M_S \widetilde L_F T \, \norm{u_n - v_n}{L^p(\Omega; C([0,T]; B_n))}. \qedhere
		\]
	\end{proof}
	
	Consequently, the proof of the following global well-posedness result is entirely analogous to that of Proposition~\ref{prop:global-well-posedness-En}:
	\begin{proposition}\label{prop:global-well-posedness-Bn}
		Suppose that Assumptions~\ref{ass:disc2}, \ref{ass:semigroup-B2}, \ref{ass:emb-Htheta-B-H2} hold with $\theta + 2\beta < 1$, and~\ref{item:F-globalLip-linearGrowth-B} is satisfied.
		Let $n \in \clos\N$, $p \in [1, \infty)$, $\xi_n \in L^p(\Omega, \cF_0, \bbP; E_n)$ and $T \in (0,\infty)$ be given.
		Then~\eqref{eq:Langevin-equation-infinite-dim-A} has a unique global mild solution $X_n$ in $L^p(\Omega; C([0,T]; B_n))$.
	\end{proposition}
	
	Under the additional convergence assumptions~\ref{ass:A-conv-B}, \ref{ass:IC-B} and~\ref{item:F-approx-B}, we can again set out to prove discrete-to-continuum convergence of $B_n$-valued global mild solutions to~\eqref{eq:Langevin-equation-infinite-dim-A} by showing that all the expressions appearing in the fixed point maps $\widetilde \Phi_{n,T}$ from~\eqref{eq:fixed-point-operator} are continuous as mappings on $L^p(\Omega; C([0,T]; \widetilde B))$.
	For the first term, the following can be proven 
	exactly in the same way as Lemma~\ref{lem:ICterm-convs}:
	\begin{lemma}\label{lem:ICterm-convs-B}
		If Assumptions~\ref{ass:disc2}--\ref{ass:A-conv-B}
		and~\ref{ass:IC-B} are satisfied,
		then we have 
		$\widetilde S_n \otimes \widetilde \xi_n
		\to S_\infty \otimes \xi_\infty$ in $L^p(\Omega; C([0,T]; \widetilde B))$ 
		as $n \to 
		\infty.$
	\end{lemma}
	The following is an analog to Lemma~\ref{lem:deterministic-convs}:
	\begin{lemma}\label{lem:deterministic-convs-B}
		Suppose that 
		Assumptions~\ref{ass:disc2}--\ref{ass:A-conv-B},~\ref{item:F-globalLip-linearGrowth-B}
		and~\ref{item:F-approx-B} are satisfied. Let $p \in [1,\infty)$ and $u \in 
		L^p(\Omega; C([0,T]; B_\infty))$ be given.	Then we have
		\begin{equation}
			\widetilde S_n * \widetilde F_n(\,\cdot\,, u)
			\to 
			S_\infty * F_\infty(\,\cdot\,, u )
			\quad 
			\text{in } L^p(\Omega; C([0,T]; \widetilde B))
			\text{ as } n \to \infty.
		\end{equation}
	\end{lemma}
	\begin{proof}
		As in Lemma~\ref{lem:deterministic-convs}, we split up the statement into 
		the following two assertions:
		\begin{enumerate}
			\item \label{lem:deterministic-convs-B:item1}
			$\widetilde S_n * \widetilde F_n(\,\cdot\, , u) - \widetilde S_n 
			* F_\infty(\,\cdot\, , u) \to 0$ in $L^p(\Omega; C([0,T]; \widetilde 
			B))$ 
			as $n \to \infty$; 
			\item \label{lem:deterministic-convs-B:item2}
			$\widetilde S_n * F_\infty(\,\cdot\, , u) \to S_\infty * 
			F_\infty(\,\cdot\, , u)$ in 
			$L^p(\Omega; C([0,T]; \widetilde B))$ 
			as $n \to \infty$.
		\end{enumerate}
		Part~\ref{lem:deterministic-convs-B:item1} is shown exactly as Lemma~\ref{lem:deterministic-convs}\ref{lem:deterministic-convs:item1}, up to replacing $E_\infty$ by $B_\infty$ (or $\widetilde B$).
		
		For~\ref{lem:deterministic-convs-B:item2}, we instead note that~\ref{item:F-globalLip-linearGrowth-B} implies, for almost every $\omega \in \Omega$,
		\begin{equation}
			t \mapsto F_\infty(\omega, t, u(\omega, t)) \in L^p(0,T; B_\infty).
		\end{equation}
		Hence, in order to argue as in Lemma~\ref{lem:deterministic-convs}\ref{lem:deterministic-convs:item2}, one needs to apply Proposition~\ref{prop:strong-continuity-IAntilde}\ref{prop:strong-continuity-IAntilde:a} and Corollary~\ref{cor:unifbb-IAtilde}\ref{cor:unifbb-IAtilde:b} with $E_n \deq B_n$ for all $n \in \clos\N$ and $\widetilde E \deq \widetilde B$.
	\end{proof}
	With these auxiliary results in place, we can prove the first main 
	discrete-to-continuum approximation result for solutions 
	to~\eqref{eq:reaction-diffusion} with globally Lipschitz drift coefficients of 
	linear growth, analogously to Theorem~\ref{thm:d2c-semilinear-global-linear}:
	\begin{theorem}\label{thm:d2c-semilinear-global-linear-B}
		Let
		Assumptions~\ref{ass:disc2}--\ref{ass:emb-Htheta-B-H2},~\ref{item:F-globalLip-linearGrowth-B}, 
		\ref{item:F-approx-B} 
		and~{\ref{ass:IC-B}} be satisfied, with
		$
		\theta + 2\beta < 1
		$
		and $p \in [1, \infty)$.
		Denoting by $X_n$ the unique $B_n$-valued global mild solution to~\eqref{eq:Langevin-equation-infinite-dim-A}, we have
		\begin{equation}
			\widetilde X_n \to X_\infty
			\quad 
			\text{in } 
			L^p(\Omega; C([0,T]; \widetilde B))
			\quad 
			\text{as } 
			n \to \infty.
		\end{equation}
	\end{theorem}
	\begin{proof}
		By Proposition~\ref{prop:existence-uniqueness-globLip-linGrowth-B} and~\ref{ass:disc2}, for small enough $T_0 \in (0,\infty)$ there exists a constant $c \in [0, 1)$, depending only on $\widetilde L_F$, $\widetilde M_S$, $\widetilde M_\Lambda$ and $\widetilde M_\Pi$, such that, for all $u, v \in L^p(\Omega; C([0,T_0]; B_\infty))$,
		\begin{equation}
			\sup\nolimits_{n\in\N} \, \norm{\widetilde \Phi_{n,T_0}(u) - \widetilde 
				\Phi_{n, T_0}(v)}{L^p(\Omega; C([0, T_0]; \widetilde B))}
			\le 
			c \norm{u - v}{L^p(\Omega; C([0, T_0]; \widetilde B))}.
		\end{equation}
		By Proposition~\ref{prop:global-well-posedness-Bn}, there exists a unique global solution $X_n \in L^p(\Omega; C([0,T]; B_n))$ to~\eqref{eq:Langevin-equation-infinite-dim-A} for every $n \in \clos \N$. 
		In particular, note that $X_\infty$ takes its values in $B_\infty \subseteq \widetilde B$.
		Thus, in order to finish the argument analogously to the proof of Theorem~\ref{thm:d2c-semilinear-global-linear}, it suffices to establish that $\widetilde \Phi_{n,T}(\phi) \to \Phi_{\infty, T}(\phi)$ in 	$L^p(\Omega; C([0,T]; \widetilde B))$ as $n \to \infty$ for all $\phi \in L^p(\Omega; C([0,T]; B_\infty))$.
		This is precisely the combined content of (the second part of) Proposition~\ref{prop:d2c-pathwise-conv-Wdelta2-B}, along with Lemmas~\ref{lem:ICterm-convs-B} and~\ref{lem:deterministic-convs-B}.
	\end{proof}
	%
	\subsection{Locally Lipschitz drifts}
	\label{sec:locLip-B}
	As in Section~\ref{sec:locLip}, we can extend Theorem~\ref{thm:d2c-semilinear-global-linear-B} to the locally 
	Lipschitz setting. Namely, we assume that
	\begin{enumerate}[label=(F{\arabic*}$'$-B), leftmargin=1.4cm]
		\item\label{item:F-locLip-unifBdd-B} 
		Assumption~\ref{item:F-locLip-unifBdd} holds with $(B_n)_{n\in\clos\N}$ in 
		place of $(E_n)_{n\in\clos\N}$. For every $r > 0$, the
		$B_n$-valued local Lipschitz and local boundedness constants are respectively denoted by $\widetilde 
		L_F^{(r)}$ 
		and $\widetilde C_{F,0}^{(r)}$.
	\end{enumerate}
	Then, arguing in the same way as Theorem~\ref{thm:d2c-conv-locLip-unifBound}, we obtain the following result.
	\begin{theorem}\label{thm:d2c-conv-locLip-unifBound-B}
		Suppose that
		Assumptions~\ref{ass:disc2}--\ref{ass:emb-Htheta-B-H2},~\ref{item:F-locLip-unifBdd-B}, 
		\ref{item:F-approx-B} 
		and~{\ref{ass:IC-B}} are satisfied, with
		$
		\theta + 2\beta < 1.
		$
		For $n \in \clos\N$, let $(X_n(t))_{t \in [0,\sigma_n)}$ be the maximal 
		local solution to~\eqref{eq:reaction-diffusion}
		with explosion time $\sigma_n \from \Omega \to [0,T]$. Then we have
		\begin{enumerate}[series=d2c-conv-locLip-unifBound-B]
			\item\label{item:d2c-conv-locLip-unifBound-sigmaConv-B} 
			$\widetilde X_n \indic{[0, \sigma_\infty \wedge \sigma_n)}
			\to 
			X_\infty \indic{[0,\sigma_\infty)}$ 
			in $L^0(\Omega \times [0,T]; \widetilde B)$
			as $n \to \infty$.
		\end{enumerate}
		If, moreover, $\sigma_n = T$ holds
		$\bbP$-a.s.\ for all $n \in \N$ and $p \in [1,\infty)$ is such 
		that 
		\begin{equation}\label{eq:Xn-unifBdd-LpOmegaCHn-B}
			\sup_{n\in\N} \norm{X_n}{L^p(\Omega; C([0,T]; B_n))} < \infty, 
		\end{equation}
		then the following assertions also hold:
		\begin{enumerate}[resume=d2c-conv-locLip-unifBound-B]
			\item\label{item:d2c-conv-locLip-unifBound-globalInfty-B} We have 
			$\sigma_\infty = T$, $\bbP$-a.s.
			\item\label{item:d2c-conv-locLip-unifBound-convInCty-B} If $p \in 
			(1,\infty)$, then for all $p^- \in [1, p)$ we have 
			\begin{equation}
				\widetilde X_n \to X_\infty 
				\quad 
				\text{ in } 
				L^{p^-}(\Omega; C([0,T]; \widetilde B))
				\text{ as } 
				n \to \infty.
			\end{equation}
		\end{enumerate}
	\end{theorem}
	%
	\subsection{Global well-posedness and convergence for dissipative drifts}
	\label{sec:global-wellposed-RDeq}
	
	In this section, we consider a class of equations whose drift coefficients satisfy not only~\ref{item:F-locLip-unifBdd-B} (which would only guarantee \emph{local} well-posedness), but additionally a type of dissipativity condition, also used in~\cite{KvN2012}, which allows us to deduce global existence.
	
	Let the subdifferential
	$\partial \norm{x_n}{B_n}$ of the 
	norm $\norm{\cdot}{B_n}$ at $x_n \in B_n$ be given by
	\begin{equation}
		\partial \norm{x_n}{B_n}
		=
		\{
		x^*_n \in B^*_n : 
		\norm{x^*_n}{B^*_n} \le 1 \text{ and } {}_{B_n}\scalar{x_n, x^*_n}{B^*_n} = 
		\norm{x_n}{B_n}
		\}.
	\end{equation}
	The assumptions on $(F_n)_{n\in\clos\N}$ under 
	which we can derive global well-posedness, see Lemma~\ref{lem:globalEx} below (which is a simplified version of~\cite[Theorem~4.3]{KvN2012} for equations driven by additive noise), 
	are as follows:
	\begin{enumerate}[label=(F{\arabic*}$''$-B), leftmargin=1.5cm]
		\item\label{item:F-diss-B} Let the conditions 
		of~\ref{item:F-locLip-unifBdd-B} 
		hold. Suppose that there exist
		$a', b' \in [0,\infty)$ and $N \in \N$ such that for all $n \in \clos \N$, 
		$(\omega, t) \in \Omega \times [0,T]$, $x_n \in \dom{A_n|_{B_n}}$, $x^*_n 
		\in \partial 
		\norm{x_n}{B_n}$ 
		and $y_n \in B_n$ we have
		\begin{equation}\label{eq:F-B}
			{}_{B_n} \langle -A_n x_n + F_n(\omega, t, x_n + y_n), x_n^* 
			\rangle_{B^*_n}
			\le 
			a'(1 + \norm{y_n}{B_n})^N
			+
			b'\norm{x_n}{B_n}.
		\end{equation}
	\end{enumerate}
	If the semigroups $(S_n(t)|_{B_n})_{t\ge0}$ are contractive on $B_n$, i.e., if $\widetilde M_S = 1$ in~\ref{ass:semigroup-B2}, then we know that $A_n \vert_{B_n}$ is \emph{accretive}, i.e.,
	\begin{equation}
		{}_{B_n} \scalar{A_n x_n, x^*_n}{B^*_n} \ge 0 \quad 
		\text{for all } 
		x_n \in \dom{A_n\vert_{B_n}}, \, x^*_n \in \partial \norm{x}{B_n}.
	\end{equation}
	Thus, in this case, it suffices to 
	check that 
	\begin{equation}\label{eq:F-B2}
		{}_{B_n} \langle F_n(\omega, t, x_n + y_n), x^*_n \rangle_{B^*_n}
		\le 
		a'(1 + \norm{y_n}{B_n})^N,
	\end{equation}
	in order to establish that~\ref{item:F-diss-B} holds for $b' = 0$.
	The next example shows how the relation~\eqref{eq:F-B2} can be verified in our 
	situation of main interest.
	It is an elaborated version of~\cite[Example~4.2]{KvN2012}.
	
	\begin{example}\label{ex:polynomial-nonlinearity-satisfies-FdissB}
		Given $n \in \clos\N$, let $B_n \deq C(\cM_n)$ be the space of continuous real-valued functions on a compact Hausdorff space $\cM_n$ equipped with the supremum norm $\norm{u_n}{B_n} \deq \sup_{\xi \in \cM_n} \abs{u_n(\xi)}$.
		In this case, for all $u_n \in B_n$, the subdifferential $\partial \norm{u_n}{B_n}$ is the weak${}^*$-closed convex hull of 
		\begin{equation}\label{eq:subdifferential-CMn}
			\{ 
			r \delta_{\hat \xi}
			:
			r \in \sgn u_n(\hat \xi) \text{ for }
			\hat \xi \in \cM_n
			\text{ such that }
			\norm{u_n}{B_n} = \abs{u_n(\hat \xi)}
			\},
		\end{equation}
		where $\delta_\xi \in C(\cM_n)^*$ denotes the evaluation functional at $\xi \in \cM_n$ and, for $y \in \R$,
		\begin{equation}
			\sgn y \deq \begin{cases}
				\{-1\}, &\text{if } y < 0; \\
				\{-1, 1\}, &\text{if } y = 0; \\
				\{1\}, &\text{if } y > 0.
			\end{cases}
		\end{equation}
		Indeed, since subdifferential sets are convex by definition and $\partial 
		\norm{u_n}{B_n}$, being contained in the closed unit ball in~$B_n^*$, is weak${}^*$ compact by the 
		Banach--Alaoglu theorem, the Krein--Milman theorem implies that it suffices 
		to argue that the extreme points of $\partial \norm{u_n}{B_n}$ are 
		precisely given by~\eqref{eq:subdifferential-CMn}. 
		This, in turn, follows from a 
		characterization of the extreme points of the closed unit ball in $C(\cM_n)^*$ 
		due to Arens and 
		Kelley~\cite{ArensKelley1947CharacterizationSpaceContinuous}.
		
		Moreover, suppose that $(F_n)_{n\in\clos\N}$ is a family of Nemytskii operators on $(B_n)_{n\in\clos\N}$ (see equation~\eqref{eq:Nemytskii}), generated by a family $(f_n)_{n\in\clos\N}$ of functions satisfying the polynomial form introduced in~\eqref{eq:fn-polynomial-def}--\eqref{eq:fn-polynomial-coefficients}.
		Fixing $n \in \clos\N$, $(\omega, t) \in \Omega \times [0,T]$ and $u_n, v_n \in B_n$, inequality~\eqref{eq:F-B2} becomes
		\begin{equation}\label{eq:polynomial-nonlinearity-satisfies-FdissB}
			{}_{B_n} \langle F_n(\omega, t, u_n + v_n), x^*_n \rangle_{B^*_n}
			\le 
			a'(1 + \norm{v_n}{B_n})^N 
			\quad 
			\text{for all } x_n^* \in \partial \norm{u_n}{B_n}.
		\end{equation}
		Since the inequality is preserved under convex combinations and weak${}^*$ limits of~$x_n^*$, the above characterization of $\partial \norm{u_n}{B_n}$ shows that it only needs to be checked for $x_n^* = r \delta_{\hat \xi}$, where $r \in \sgn u_n(\hat \xi)$ for $\hat \xi \in \cM_n$ such that $\norm{u_n}{B_n} = \abs{u_n(\hat \xi)}$.
		That is, it suffices to verify that
		\begin{equation}
			r f_n(\omega, t, u_n(\hat\xi) + v_n(\hat \xi))
			\le 
			a'(1 + \norm{v_n}{B_n})^N.
		\end{equation}
		Indeed, for $(f_n)_{n\in\clos\N}$ satisfying~\eqref{eq:fn-polynomial-def}--\eqref{eq:fn-polynomial-coefficients}, we can establish the estimate
		\begin{equation}
			r f_n(\omega, t, y + z)
			\lesssim_{(c, C, k)} 
			(1 + \abs{z})^{2k+1}
		\end{equation}
		for all $(\omega, t) \in \Omega \times [0,T]$,
		$y, z \in \R$ and $r \in \sgn y$.
		This implies the existence of a constant $a' \in [0,\infty)$, depending only on $c, C \in (0,\infty)$ from~\eqref{eq:fn-polynomial-coefficients} and $k \in \N_0$ from~\eqref{eq:fn-polynomial-def}, such that~\eqref{eq:polynomial-nonlinearity-satisfies-FdissB} holds with $N = 2k+1$.
	\end{example}
	\begin{lemma}\label{lem:globalEx}
		Let	Assumptions~\ref{ass:disc2}--\ref{ass:emb-Htheta-B-H2} and~\ref{item:F-diss-B} hold with $\widetilde M_S = 1$, let $n \in \clos\N$ and suppose that $\theta + 2\beta < 1$. If $\xi_n \in L^p(\Omega, \cF_0, \bbP; B_n)$ for some $p \in [1,\infty)$, then the maximal solution $(X_n(t))_{t \in [0,\sigma_n)}$ to~\eqref{eq:reaction-diffusion} is global (i.e., it holds $\bbP$-a.s.\ that $\sigma_n = T$), and we have 
		\begin{equation}
			\norm{X_n}{L^p(\Omega; C([0,T]; B_n))}
			\lesssim_{(a', b', T, N)}
			1 + \norm{\xi_n}{L^p(\Omega; B_n)}
			+ \norm{W_{A_n}}%
			{L^{Np}(\Omega; C([0,T]; B_n))}^N.
		\end{equation}
	\end{lemma}
	\begin{proof}
		Fix $n \in \clos \N$.
		For each $m \in \N$, let $F_{n,m}$ denote the globally Lipschitz 
		retraction of $F_n$ onto the closed ball of radius $m$ around $0 \in B_n$, cf.~\eqref{eq:globLip-retract} (replacing $E_n$ by~$B_n$).
		Then $F_{n,m}$ satisfies the global Lipschitz and (global) linear growth estimates in~\ref{item:F-globalLip-linearGrowth-B}, hence by Proposition~\ref{prop:global-well-posedness-Bn} there exists a unique global $B_n$-valued mild solution $X_{n,m} \in L^p(\Omega; C([0,T]; B_n))$ to~\eqref{eq:Langevin-equation-infinite-dim-A} with drift coefficient operator~$F_{n,m}$.
		By the triangle inequality,
		\begin{equation}
			\begin{aligned}
				&\norm{X_{n,m}}{L^p(\Omega; C([0,T]; B_n))}
				\\
				&\qquad \le 
				\norm{S_n \otimes \xi_n + S_n * F_{n,m}(\,\cdot\, , X_{n,m})}%
				{L^p(\Omega; C([0,T]; B_n))}
				+
				\norm{W_{A_n}}%
				{L^p(\Omega; C([0,T]; B_n))}.
			\end{aligned}
		\end{equation}
		As shown in the proof of~\cite[Theorem~4.3]{KvN2012}, $F_{n,m}$ inherits the dissipativity estimate~\ref{item:F-diss-B}, with the same constants $a'$, $b'$ and $N$, from $F_n$. 
		Thus, by~\cite[Lemma~4.4]{KvN2012},
		\begin{align}
			&\norm{S_n \otimes \xi_n + S_n * F_{n,m}(\,\cdot\, , X_m)}{C([0,T]; 
				B_n)}
			\\&\qquad \le 
			e^{b' T} \biggl( \norm{\xi}{B_n} + a' \int_0^T (1 + 
			\norm{W_{A_n}(s)}{B_n})^N \rd s\biggr)
			\\
			&\qquad \le 
			e^{b' T} \norm{\xi}{B_n} + a' T 2^{N-1} e^{b' T} (1 + 
			\norm{W_{A_n}}{C([0,T]; B_n)}^N), 
			\quad 
			\bbP\text{-a.s.}
		\end{align}
		It follows that
		\begin{equation}
			\begin{aligned}
				&\norm{S_n \otimes \xi_n + S_n * F_{n,m}(\,\cdot\, , X_m)}{L^p(\Omega; C([0,T]; 
					B_n))}
				\\&\qquad \lesssim_{(a', b', T, N)}
				1 +	\norm{\xi_n}{L^p(\Omega; B_n)} + \norm{W_{A_n}}{L^{Np}(\Omega; 
					C([0,T];B_n))}^N.
			\end{aligned}
		\end{equation}
		Note that $W_{A_n} \in L^{Np}(\Omega; C([0,T]; B_n))$ by 
		Proposition~\ref{prop:d2c-pathwise-conv-Wdelta2-B} (with $Np$ taking the role of $p$). Combining these 
		estimates, we find
		\begin{equation}
			\begin{aligned}
				\sup_{m\in\N} \norm{X_{n,m}}{L^p(\Omega; C([0,T]; B_n))}
				\lesssim_{(a', b', T, N)}
				1 +	\norm{\xi_n}{L^p(\Omega; B_n)}
				+ \norm{W_{A_n}}%
				{L^{Np}(\Omega; C([0,T]; B_n))}^N,
			\end{aligned}
		\end{equation}
		so the result follows by Theorem~\ref{thm:approx-local-processes}\ref{item:thm:approx-local-processes-iv}--\ref{item:thm:approx-local-processes-v} applied to $Y_m \deq X_{n,m}$.
	\end{proof}
	Combined with Theorem~\ref{thm:d2c-conv-locLip-unifBound-B}%
	\ref{item:d2c-conv-locLip-unifBound-globalInfty-B}--\ref{item:d2c-conv-locLip-unifBound-convInCty-B},
	whose uniform-boundedness hypothesis~\eqref{eq:Xn-unifBdd-LpOmegaCHn-B} is 
	verified by the combination of 
	Lemma~\ref{lem:globalEx} and
	Proposition~\ref{prop:d2c-pathwise-conv-Wdelta2-B} under the assumption in the following corollary, we derive:
	\begin{corollary}\label{cor:d2c-conv-dissB}
		Suppose that~\ref{ass:disc2}--\ref{ass:emb-Htheta-B-H2},~\ref{item:F-diss-B},~\ref{item:F-approx-B} 
		and~\ref{ass:IC-B} are 
		satisfied with $\widetilde M_S = 1$, $p \in (1,\infty)$ and $\theta + 2\beta < 1$. 
		Then for any $p^{-} \in [1,p)$, the sequence 
		$((X_n)_{t\in[0,T]})_{n\in\clos\N}$ of $B_n$-valued global solutions 
		to~\eqref{eq:reaction-diffusion} satisfies
		\begin{equation}
			\widetilde X_n \to X_\infty 
			\quad 
			\text{ in } 
			L^{p^{-}}\!(\Omega; C([0,T]; \widetilde B))
			\text{ as } 
			n \to \infty.
		\end{equation}
	\end{corollary}
	
	\section{Outlook}
	\label{sec:outlook}
	In this section we suggest some possible extensions of both the convergence of graph-discretized Whittle--Mat\'ern SPDEs shown in Section~\ref{sec:graph-WM} as well as the underlying abstract results from Sections~\ref{sec:OU-H}--\ref{sec:reaction-diffusion}.
	
	As discussed in Subsection~\ref{sec:example:discussion-assumptions}, the approximation results from Theorem~\ref{thm:d2c-conv-example-SPDEs} for~\eqref{eq:example-SPDE} might be extended to broader classes of domains $\cM$, connectivity length regimes $(h_n)_{n\in\N}$, coefficient functions $\tau, \kappa \from \cM \to [0,\infty)$ and (fractional) powers ${s \in (0,\infty)}$.
	Possible advancements to this end include the establishment of more general $L^\infty$-convergence results for graph Laplacian eigenfunctions (or convergence of Whittle--Mat\'ern operators without using spectral convergence), as well as uniform $L^\infty$-boundedness of the semigroups (for instance via heat kernel estimates).
	Under more restrictive assumptions on the connectivity parameter regime, rates of convergence for the case of purely spatial (i.e., stationary) graph-discretized linear SPDEs were established in~\cite{SanzAlonsoYang2022}.
	The same might be possible in the linear spatiotemporal setting since the discrete-to-continuum Trotter--Kato theorem can be extended to yield error estimates, see~\cite[Section~2.2]{ItoKappel1998}.
	The semilinear cases, however, appear to be out of reach for the methods used in this work.
	
	The proofs of the abstract discrete-to-continuum approximation results from Sections~{\ref{sec:OU-H}--\ref{sec:reaction-diffusion}} largely rely on incorporating projection and lifting operators into arguments from~\cite{KvN2011, KvN2012} in an appropriate way.
	By adapting other proofs from these sources along similar lines, it is likely that our results can be extended further, in particular enabling us to relax the simplifying assumptions that the UMD Banach spaces $(E_n)_{n\in\clos\N}$ have type $2$ and that the driving noise is additive. 
	In fact, we expect more generally that many results asserting continuous dependence of stochastic evolution problems on their coefficients can be extended to discrete-to-continuum approximation theorems via this procedure.
	
	One particular type of problem for which this would be interesting is the class of stochastic evolution \emph{inclusions}, whose drift operators are allowed to be multi-valued; this occurs, for instance, in the Langevin setting if $F_n = \partial\varphi_n$ is the subdifferential of a convex and lower-semicontinuous but non-differentiable functional $\varphi_n$ on the state space, taking values in $(-\infty, \infty]$.
	Continuous dependence results for stochastic inclusion problems have been established in two different settings in~\cite{GessToelle2016, SS2021}; however, neither of these covers the class of \emph{semilinear} inclusions driven by \emph{cylindrical} (i.e., \emph{white}) noise.
	Not much theory appears to be available for such problems, with even the question of well-posedness (for fixed $n \in \clos\N$) being highly nontrivial. 
	In fact, to the best of our knowledge, the only results in this direction concern the (important) subclass of \emph{stochastic reflection problems} (also known as \emph{Skorokhod problems} in the scalar-valued case), given by
	\begin{equation}\label{eq:inclusion}
		\left\lbrace
		\begin{aligned}
			\rd X(t) &\in -A X(t)\rd t - \partial I_{\Gamma}(X(t)) \rd t + \rd W(t), 
			\quad 
			t 
			\in 
			(0,T], 
			\\
			X(0) &= \xi,
		\end{aligned}
		\right.
	\end{equation}
	where $\Gamma$ is a convex subset of a (Hilbertian) state space $H$ and the \emph{indicator functional} $I_{\Gamma} \from H \to (-\infty,\infty]$ is defined to vanish on $\Gamma$ and equal $\infty$ outside of it.
	In the first work on this problem, Nualart and Pardoux~\cite{nualart_white_1992} used a direct approach to show existence and uniqueness in a setting which corresponds to $H \deq L^2(0,1)$, $A \deq -\frac{\rd^2}{\rd x^2}$ with homogeneous Dirichlet or Neumann boundary conditions, and $\Gamma \deq K_0$, where 
	\[
	K_\alpha \deq \{ u \in L^2(0,1) : u(x) \ge -\alpha \text{ for a.e.\ } x \in (0,1) \}, 
	\quad \alpha \in [0,\infty).
	\]
	In~\cite{RZZ2012}, the authors first use the theory of Dirichlet forms to establish well-posedness of~\eqref{eq:inclusion} in the case that $\Gamma$ is a ``regular'' convex subset of a general Hilbert space $H$, a condition which includes $\Gamma \deq K_\alpha$ for $\alpha > 0$ but not $K_0$, which is treated separately using different techniques.
	Lastly, the work~\cite{BdPT2012} describes a variational approach to study~\eqref{eq:inclusion} in a similar setting under the assumption that $0$ belongs to the interior of $\Gamma$, which excludes the choice $\Gamma \deq K_0$. 
	We point out that the argument used on~\cite[p.\ 362]{BdPT2012} to extract a weak${}^*$-convergent sequence from the set $(u_\varepsilon)_{\varepsilon>0}$ in the dual space of $L^\infty(0,T; H)$ appears to be flawed, as it seems to imply that the closed unit ball of the dual of this (non-separable) space is \emph{sequentially} compact, which is not the case.
	Hence, the argument would need to be finished using a \emph{generalized} subsequence (also known as a sub\emph{net}) converging in the weak${}^*$ sense to some $u^*$.
	For this reason, and since the theory for convergence of Dirichlet forms and their associated processes is well established---see for instance~\cite{kolesnikov_mosco_2006} for general results and~\cite{zanella_dirichlet_2017} for an application to Markov chain Monte Carlo scaling---the setting of~\cite{RZZ2012} is perhaps the most promising for an attempt at establishing discrete-to-continuum convergence results for~\eqref{eq:inclusion}.
	
	\section*{Conflict of interest and data availability statement}
	
	The authors declare that they have no conflict of interest.
	Data sharing is not applicable to this article as no datasets were generated or analyzed during the current study.
	
	\appendix 
	
	\section{Proofs of intermediate results in Section~\ref{sec:intermediate-results}}
	\label{app:proofs-of-intermediate-results}
	
	We start with the proof of Lemma~\ref{lem:spectr-conv-squaregrid} which establishes spectral convergence rates for the Laplace--Beltrami operator on the torus, discretized by a uniform grid.
	
	\begin{proof}[Proof of Lemma~\ref{lem:spectr-conv-squaregrid}]
		First suppose $m = 1$. 
		In this case, the continuum Laplace--Beltrami operator reduces to the second derivative $-\frac{\rd^2}{\rd x^2}$ with periodic boundary conditions.
		Its eigenvalues and $L^2(\cM)$-normalized eigenfunctions are respectively given, for all $j \in \N$ and $x \in [0,1]$, by
		\begin{equation}
			\lambda_\infty^{(j)}
			=
			\begin{cases}
				j^2 \pi^2 &\text{if $j$ is even}, \\
				(j-1)^2\pi^2 &\text{if $j$ is odd};
			\end{cases}
			\quad 
			\psi_\infty^{(j)}(x) 
			=
			\begin{cases}
				1 &\text{if } j = 1, \\
				\sqrt{2} \sin(j\pi x) &\text{if } j \text{ is even}, \\
				\sqrt{2} \cos((j-1)\pi x) &\text{if } j \text{ is odd}.
			\end{cases}
		\end{equation}
		That is, $0$ is an eigenvalue corresponding to the constant $1$ eigenfunction, and $(2k)^2\pi^2$ is an eigenvalue with eigenfunctions $x \mapsto \sin(2k\pi x), \cos(2k\pi x)$ for all $k \in \N$.
		
		The eigenvalues $(\lambda_{n}^{(j)})_{j=1}^n$ of the corresponding graph Laplacian, which in this case reduces to the finite difference approximation of the second derivative on the grid with $n \in \N$ points, are given by
		\begin{equation}
			\lambda_n^{(j)}
			=
			\begin{cases}
				4n^2 \sin^2 \Bigl(\frac{\pi j}{2n}\Bigr) &\text{if $j$ is even}, \\
				4n^2 \sin^2 \Bigl(\frac{\pi(j-1)}{2n}\Bigr) &\text{if $j$ is odd}.
			\end{cases}
		\end{equation}
		The corresponding $L^2(\cM_n)$-normalized eigenfunctions are 
		\begin{equation}
			\psi_n^{(j)}(x_n^{(i)}) 
			=
			\begin{cases}
				1 &\text{if } j = 1, \\
				(-1)^i &\text{if } j = n \text{ is even}, \\
				\sqrt{2} \sin\bigl(j\pi x_n^{(i)}\bigr) &\text{if } j \neq n \text{ and } j \text{ is even}, \\
				\sqrt{2} \cos\bigl((j-1)\pi x_n^{(i)}\bigr) &\text{if } j  \text{ is odd}.
			\end{cases}
		\end{equation}
		Let $j \in \{1,\dots,n\}$. 
		Supposing that $j$ is even (the odd case being analogous), we can write 
		\begin{equation}
			\lambda_\infty^{(j)} - \lambda_n^{(j)}
			=
			j^2\pi^2 \biggl(1 - \frac{4n^2}{j^2\pi^2}\sin^2 \Bigl(\frac{\pi j}{2n}\Bigr)\biggr)
			=
			j^2\pi^2 \biggl(1 - \biggl[\frac{2n}{j\pi}\sin \Bigl(\frac{\pi j}{2n}\Bigr)\biggr]^2\biggr),
		\end{equation}
		so that the estimate in~\eqref{eq:eigenval-conv-squaregrid} for $m = 1$ follows from the elementary inequality 
		$
		0 \le 1 - (\sin(x)/x)^2 \le \frac{1}{3}x^2,
		$
		which is valid for all $x \in \R$.
		Estimate~\eqref{eq:eigenfunc-conv-squaregrid} is a consequence of the fact that the sine and cosine functions are $1$-Lipschitz; note that we only consider $j \in \{1,\dots,n-1\}$ to avoid the case where $j = n$ for even $n$.
		
		The result for higher dimensions $m \in \N$ can be derived from the $m = 1$ case. 
		Indeed, by separation of variables in the continuum case, or by writing the discretized operator as a Kronecker sum of $m$ one-dimensional discretizations, one finds that the eigenvalues and eigenvectors of the $m$-dimensional operators are given by sums and products, respectively, of their $1$-dimensional counterparts.
		From this, one can deduce the desired result.
	\end{proof}
	
	Next we prove Theorem~\ref{thm:convergence-An} regarding the convergence of the fractional resolvent operators $\widetilde R_n^{\beta'}$ in various settings and norms.
	
	\begin{proof}[Proof of Theorem~\ref{thm:convergence-An}]
		Assertions~\ref{thm:convergence-An:a}--\ref{thm:convergence-An:c} can all be shown using analogous arguments. 
		Thus, we only provide a detailed proof for part~\ref{thm:convergence-An:b}, being the most involved case, and subsequently summarize the changes needed for~\ref{thm:convergence-An:a} and~\ref{thm:convergence-An:c}.
		
		\ref{thm:convergence-An:b} \textbf{Step 1 \textnormal{(Setup and notation)}.}
		The operator $\widetilde R_n^{\beta'}$ acts on functions $f \in L^2(\cM)$ in the following way:
		\begin{align}
			\widetilde R_n^{\beta'} f 
			&=
			\sum_{j=1}^n \bigl(1 + [\lambda^{(j)}_n]^{s}\bigr)^{-{\beta'}} \scalar{\Pi_n f, 
				\psi^{(j)}_n}{L^2(\cM_n)} \Lambda_n \psi^{(j)}_n
			\\&=
			\sum_{j=1}^n 
			\bigl(1 + [\lambda^{(j)}_n]^{s}\bigr)^{-{\beta'}} 
			\scalar{f, \Lambda_n \psi^{(j)}_n}{L^2(\cM)} 
			\Lambda_n \psi^{(j)}_n.
		\end{align}
		Here, we used the fact that $\Pi_n^* = \Lambda_n$ (see~\eqref{eq:Lambda-Pi-adjoint}) on the second line.
		Using the tensor notation from Section~\ref{sec:prelims:notation} and denoting $\widetilde \psi^{(j)}_n \deq \Lambda_n \psi^{(j)}_n$, we can write these operators more concisely as follows:
		\begin{equation}
			U_n 
			\deq 
			\widetilde R_n^{{\beta'}} 
			=
			\sum_{j=1}^n 
			\bigl(1 + [\lambda^{(j)}_n]^{s}\bigr)^{-{\beta'}} \,
			\widetilde \psi^{(j)}_n \otimes \widetilde \psi^{(j)}_n.
		\end{equation}
		By Assumption~\ref{ass:connectivity-regimes}\ref{ass:connectivity-regimes:suboptimal-range}, there exists a sequence of natural numbers $(k_n)_{n\in\N}$ which satisfies the conditions of Theorem~\ref{thm:eigenfunc-conv-randomgraph}\ref{thm:eigenfunc-conv-randomgraph:Linfty}, as well as the relation
		\begin{equation}\label{eq:kn-large-enough}
			k_n 
			\gg 
			n^{\frac{m}{4s\beta}}
			\gtrsim
			n^{\frac{m}{4s\beta'}}.
		\end{equation}
		We use the sequence $(k_n)_{n\in\N}$ to define the following approximations for $n \in \N$:
		\begin{align}
			U_n^1 
			&\deq 
			\sum_{j=1}^{k_n} 
			\bigl(1 + [\lambda^{(j)}_\infty]^s\bigr)^{-\beta'}
			\psi^{(j)}_\infty \otimes \psi^{(j)}_\infty, 
			\\
			U_n^2 
			&\deq 
			\sum_{j=1}^{k_n} 
			\bigl(1 + [\lambda^{(j)}_\infty]^s\bigr)^{-\beta'}
			\widetilde \psi^{(j)}_n \otimes \widetilde \psi^{(j)}_n, 
			\\
			U_n^3 
			&\deq 
			\sum_{j=1}^{k_n} 
			\bigl(1 + [\lambda^{(j)}_n]^s\bigr)^{-\beta'}
			\widetilde \psi^{(j)}_n \otimes \widetilde \psi^{(j)}_n.
		\end{align}
		For notational convenience, we will abbreviate the $\LO(L^2(\cM); L^\infty(\cM))$-norm by $\norm{\,\cdot\,}{2\to\infty}$ throughout this proof.
		We will make repeated use of the following estimate:
		Given an operator of the form $U = \sum_j \alpha_j e_j \otimes f_j$ for some scalars $(\alpha_j)_{j} \subseteq \R$, an orthonormal system $(e_j)_j \subseteq L^2(\cM)$ and some functions $(f_j)_j \subseteq L^\infty(\cM)$, we have by the Cauchy--Schwarz inequality:
		\begin{equation}\label{eq:L2-Linfty-estimate}
			\norm{U}{2\to\infty}
			\le 
			\sup\nolimits_j \norm{f_j}{L^\infty(\cM)}
			\biggl(\sum\nolimits_{j} \abs{\alpha_j}^2 \biggr)^{\frac{1}{2}}.
		\end{equation}
		Moreover, it is immediate from the definition of the $\norm{\,\cdot\,}{2\to\infty}$-norm that
		\begin{equation}\label{eq:L2Linfty-rank1tensor}
			\norm{h \otimes f}{2\to\infty} = \norm{h}{L^2(\cM)} \norm{f}{L^\infty(\cM)}
			\quad
			\text{for all } h \in L^2(\cM) \text{ and } f \in L^\infty(\cM).
		\end{equation}
		
		\textbf{Step 2 \textnormal{($U_\infty - U_n^1$ and $U_n^3 - U_n$)}.}
		For the difference between $U_\infty$ and $U_n^1$, we find using~\eqref{eq:L2-Linfty-estimate} and Assumption~\ref{ass:uniform-Linfty-bdd-eigenfuncs}:
		\begin{equation}\label{eq:truncation-infseries}
			\begin{aligned}
				\norm{U_\infty - U_n^1}{2\to\infty}^2
				\le 
				M^2_{\psi,\infty}
				\sum_{j=k_n+1}^\infty 
				\bigl(1 + [\lambda^{(j)}_\infty]^s\bigr)^{-2\beta'}.
			\end{aligned}
		\end{equation}
		Recalling Weyl's law~\eqref{eq:weyls-law}, which implies that
		$\bigl(1 + [\lambda^{(j)}_\infty]^s\bigr)^{-2\beta'} \eqsim_{\cM} j^{-4s\beta'/m}$, we observe that the series on the right-hand side is convergent precisely when $\beta' > \frac{m}{4s}$. 
		Since $k_n \to \infty$ as $n \to \infty$, this implies $U_\infty \to U_n^1$ in $\LO(L^2(\cM); L^\infty(\cM))$.
		
		Similarly, for the difference between $U_n^3$ and $U_n$, we have
		\begin{equation}
			\begin{aligned}
				\norm{U_n^3 - U_n}{2 \to \infty}^2
				&\le 
				M_{\psi,\infty}^2
				\sum_{j=k_n + 1}^n 
				\bigl( 1 + [\lambda^{(j)}_n]^s \bigr)^{-2\beta'}
				\le 
				M_{\psi,\infty}^2
				\bigl(n - k_n\bigr)
				\bigl( 1 + [\lambda^{(k_n)}_n]^s \bigr)^{-2\beta'},
			\end{aligned}
		\end{equation}
		where the second inequality is due to the non-decreasing order of $(\lambda_n^{(j)})_{j=1}^n$.
		Moreover, as a consequence of Theorem~\ref{thm:eigenval-conv-randomgraph}, there exists a constant $C' > 0$ such that, $\widetilde \bbP$-a.s., for all $n \in \N$ and $j \in \{1,\dots,k_n\}$, we have $\lambda_n^{(j)} \ge C'\lambda_\infty^{(j)}$.
		In particular,  $\lambda_n^{(k_n)} \gtrsim \lambda_\infty^{(k_n)}$.
		Together with Weyl's law, we find $1 + [\lambda^{(k_n)}_n]^s \gtrsim_{\cM} k_n^{{2s}/{m}}$, so that the convergence of this difference is due to~\eqref{eq:kn-large-enough}:
		\begin{equation}\label{eq:Un3-Un-converges}
			\begin{aligned}
				\norm{U_n^3 - U_n}{2 \to \infty}^2
				\lesssim_{\cM} 
				M_{\psi,\infty}^2 n k_n^{-{4s\beta'}/{m}}
				\to 0 
				\quad\text{as } n \to \infty.
			\end{aligned}
		\end{equation}
		
		\textbf{Step 3 \textnormal{($U_n^1 - U_n^2$)}.}
		In order to show that 
		\begin{equation}
			U_n^1 - U_n^2 \to 0 
			\quad 
			\text{in } L^0(\widetilde\Omega; \LO(L^2(\cM); L^\infty(\cM))) \quad \text{as } n \to \infty,
		\end{equation}
		we first fix an arbitrary $\varepsilon > 0$.
		Then, for all $\ell, n \in \N$ such that $k_n > \ell$, we split off the first $\ell$ terms and use the triangle inequality to obtain
		\begin{equation}\label{eq:step3}
			\begin{aligned}
				&\norm{U_n^1 - U_n^2}{2 \to \infty}
				\le
				\sum_{j=1}^{\ell} 
				\bigl(1 + [\lambda^{(j)}_\infty]^s\bigr)^{-\beta'}
				\norm{\psi^{(j)}_\infty \otimes \psi^{(j)}_\infty 
					-
					\widetilde \psi^{(j)}_n \otimes \widetilde \psi^{(j)}_n}{2 \to \infty}
				\\
				&+\Norm[\bigg]{\sum_{j=\ell + 1}^{k_n} 
					\bigl(1 + [\lambda^{(j)}_\infty]^s\bigr)^{-\beta'}
					\psi^{(j)}_\infty \otimes \psi^{(j)}_\infty}_{2 \to \infty}
				+
				\Norm[\bigg]{
					\sum_{j=\ell + 1}^{k_n} 
					\bigl(1 + [\lambda^{(j)}_\infty]^s\bigr)^{-\beta'}
					\widetilde \psi^{(j)}_n \otimes \widetilde \psi^{(j)}_n}_{2 \to \infty}.
			\end{aligned}
		\end{equation}
		Using the triangle inequality once more, followed by~\eqref{eq:L2Linfty-rank1tensor} and Assumption~\ref{ass:uniform-Linfty-bdd-eigenfuncs}, the
		norms in the summation over $j \in \{1,\dots,\ell\}$ can be bounded by
		\begin{align}
			&\norm{\psi^{(j)}_\infty \otimes \psi^{(j)}_\infty 
				-
				\widetilde \psi^{(j)}_n \otimes \widetilde \psi^{(j)}_n}{2 \to \infty}
			\\&\quad \le 
			\norm{\psi_\infty^{(j)} \otimes (\psi_\infty^{(j)} - \widetilde \psi_n^{(j)})}{2 \to \infty}
			+
			\norm{(\psi_\infty^{(j)} - \widetilde\psi_n^{(j)}) \otimes \widetilde \psi_n^{(j)}}{2 \to \infty}
			\\&\quad = 
			\norm{\psi_\infty^{(j)} - \widetilde \psi_n^{(j)}}{L^\infty(\cM)}
			+
			\norm{\psi_\infty^{(j)} - \widetilde\psi_n^{(j)}}{L^2(\cM)}
			\norm{\widetilde\psi_n^{(j)}}{L^\infty(\cM)}
			\\&\quad \le 
			(1 + M_{\psi,\infty}) \norm{\psi_\infty^{(j)} - \widetilde\psi_n^{(j)}}{L^\infty(\cM)},
		\end{align}
		whereas the remaining two summations can be treated by arguing as for~\eqref{eq:truncation-infseries}. 
		Together, this yields
		\begin{equation}
			\norm{U_n^1 - U_n^2}{2\to\infty}
			\le
			\begin{aligned}[t]
				(1 + M_{\psi,\infty})
				&\sum_{j=1}^{\ell} 
				\bigl(1 + [\lambda^{(j)}_\infty]^s\bigr)^{-\beta'}
				\norm{\psi_\infty^{(j)} - \widetilde\psi_n^{(j)}}{L^\infty(\cM)}
				\\
				&+2\biggl(\sum_{j=\ell + 1}^{\infty} \bigl(1 + [\lambda^{(j)}_\infty]^s\bigr)^{-2\beta'} \biggr)^{\frac{1}{2}}
				\label{eq:sum-tail}.
			\end{aligned}
		\end{equation}
		Since we have already seen in Step 2 that the latter series converges, we can fix $\ell \in \N$ so large that the second sum on the right-hand side is less than $\frac{1}{2}\varepsilon$. 
		Moreover, it follows from Theorem~\ref{thm:eigenfunc-conv-randomgraph}\ref{thm:eigenfunc-conv-randomgraph:Linfty} that $\norm{\psi_\infty^{(j)} - \widetilde\psi_n^{(j)}}{L^\infty(\cM)} \to 0$ in $L^0(\widetilde\Omega,\widetilde\bbP)$ as $n \to \infty$ for every $j \in \{1,\dots,\ell\}$. 
		In particular, there exists $N \in \N$ such that, for all $n \ge N$, the first sum on the right-hand side is less than $\frac{1}{2}\varepsilon$, and thus the whole right-hand side is less than $\varepsilon$, with probability $\widetilde \bbP \ge 1 - \varepsilon$.
		This shows $\norm{U_n^1 - U_n^2}{2\to\infty} \to 0$ in probability, as desired.
		
		\textbf{Step 4 \textnormal{($U_n^2 - U_n^3$)}.}
		Finally, the difference $U_n^2 - U_n^3$ can be treated in the same manner as Step 3, namely by writing, for all $\ell, n \in \N$ such that $k_n > \ell$,
		\begin{align}\label{eq:approximation-of-eigenvalues}
			\begin{aligned}
				\norm{U_n^2 - U_n^3}{2 \to \infty}^2
				&\lesssim_{M_{\psi,\infty}} 
				\sum_{j=1}^{k_n} 
				\abs[\big]{\bigl(1 + [\lambda^{(j)}_n]^{s}\bigr)^{-\beta'} - \bigl(1 + [\lambda^{(j)}_\infty]^{s}\bigr)^{-\beta'}}^2
				\\
				&\le 
				\begin{aligned}[t]
					&\sum_{j=1}^{\ell} 
					\abs[\big]{\bigl(1 + [\lambda^{(j)}_n]^{s}\bigr)^{-\beta'} - \bigl(1 + [\lambda^{(j)}_\infty]^{s}\bigr)^{-\beta'}}^2
					\\&+
					2 \sum_{j=\ell + 1}^{k_n}\bigl(1 + [\lambda^{(j)}_n]^{s}\bigr)^{-2\beta'}
					+
					2 \sum_{j=\ell + 1}^{k_n}\bigl(1 + [\lambda^{(j)}_\infty]^{s}\bigr)^{-2\beta'}.
				\end{aligned}
			\end{aligned}
		\end{align}
		Using the fact that, $\widetilde\bbP$-a.s., we have $\lambda_n^{(j)} \gtrsim \lambda_\infty^{(j)}$ for all $j \in \{1, \dots, k_n\}$ (see Step~2), the latter two summations can be bounded, up to a multiplicative constant, by the convergent series
		$
		\sum_{j=\ell + 1}^{\infty}\bigl(1 + [\lambda^{(j)}_\infty]^{s}\bigr)^{-2\beta'}.
		$
		Combined with the eigenvalue convergence asserted by Theorem~\ref{thm:eigenval-conv-randomgraph}, which can be applied to the remaining summation, we obtain $\norm{U_n^2 - U_n^3}{2\to\infty} \to 0$ as $n \to \infty$, $\widetilde \bbP$-a.s., by arguing as in Step~3. 
		Thus, we have shown part~\ref{thm:convergence-An:b}.
		
		\ref{thm:convergence-An:a}
		Replace~\eqref{eq:L2-Linfty-estimate} by the identity
		$
		\norm{U}{\LO_2(L^2(\cM))}^2
		=
		\sum\nolimits_{j} \abs{\alpha_j}^2 \norm{f_j}{L^2(\cM)}^2,
		$
		which follows directly from the definition of the Hilbert--Schmidt norm.
		Assumption~\ref{ass:uniform-Linfty-bdd-eigenfuncs} is not required since all the eigenfunctions are $L^2$-normalized.
		The sufficiency of Assumption~\ref{ass:connectivity-regimes}\ref{ass:connectivity-regimes:optimal-range} and the $\widetilde{\bbP}$-a.s.\ convergence in the conclusion are due to the use of Theorem~\ref{thm:eigenfunc-conv-randomgraph}\ref{thm:eigenfunc-conv-randomgraph:L2} instead of Theorem~\ref{thm:eigenfunc-conv-randomgraph}\ref{thm:eigenfunc-conv-randomgraph:Linfty}.
		
		\ref{thm:convergence-An:c}
		Recall from Setting~\ref{ex:square-grid} that $h_n \deq n^{-\frac{1}{m}}$ by definition.
		In view of Lemma~\ref{lem:spectr-conv-squaregrid}, we can take $k_n \deq n - 1$, hence neither of the bounds on $h_n$ from Assumption~\ref{ass:connectivity-regimes} is needed.
		Indeed, in the proof of~\ref{thm:convergence-An:b}, the lower bound~\eqref{eq:kn-large-enough} on $k_n$ was only used in~\eqref{eq:Un3-Un-converges}, which becomes $\norm{U_n^3 - U_n}{2 \to \infty}^2 \lesssim_{\cM} M_{\psi,\infty}^2 k_n^{-{4s\beta'}/{m}}$ in the current situation, and this tends to zero since, trivially, $k_n \to \infty$ as $n \to \infty$.
	\end{proof}
	Lastly, we prove Lemma~\ref{lem:uniform-ultracontractivity} asserting the uniform ultracontractivity of the semigroups $(S_n(t))_{t\ge0}$ associated to the (discretized) generalized Whittle--Mat\'ern operators $-(\cL_n^{\kappa,\tau})^s$.
	\begin{proof}[Proof of Lemma~\ref{lem:uniform-ultracontractivity}]
		\ref{lem:uniform-ultracontractivity:a}  
		For $p = \infty$, the statement holds by Assumption~\ref{ass:uniform-Linfty-bdd-semigroup}\ref{ass:uniform-Linfty-bdd-semigroup:bdd}.
		For $p = 2$, we note that, for all $n \in \clos\N$, $t > 0$, and $f \in L^2(\cM_n)$,
		\begin{equation}
			\begin{aligned}
				\norm{S_n(t) f}{L^\infty(\cM_n)}
				=
				\norm{R_n^\beta (\id_n + A_n)^{\beta} S_n(t) f}{L^\infty(\cM_n)}
				\le 
				t^{-\beta} \norm{R_n^\beta}{2 \to \infty}
				\norm{f}{L^2(\cM_n)},
			\end{aligned}
		\end{equation}
		where we used~\eqref{eq:app-analytic-semigroup-est}, as $(S_n(t))_{t\ge0}$ is a contractive analytic semigroup on $L^2(\cM_n)$.
		In the proof of Theorem~\ref{thm:convergence-An}\ref{thm:convergence-An:b}, we found 
		$
		\norm{R_n^\beta}{2 \to \infty}
		\le 
		M_{\psi,\infty}
		\sum_{j=1}^n (1 + [\lambda_n^{(j)}]^s)^{-2\beta}
		$
		under Assumption~\ref{ass:uniform-Linfty-bdd-eigenfuncs}.
		Since Assumption~\ref{ass:connectivity-regimes}\ref{ass:connectivity-regimes:suboptimal-range} implies~\ref{ass:connectivity-regimes}\ref{ass:connectivity-regimes:optimal-range} with the same $\beta$, and since the estimate of $\norm{R_n^\beta}{2 \to \infty}$ only involves the eigenvalues (and not the eigenfunctions), we can in both cases argue as in Theorem~\ref{thm:convergence-An}\ref{thm:convergence-An:a}, under Assumption~\ref{ass:kappa-tau}\ref{ass:kappa-tau:general}, to deduce that, $\widetilde \bbP$-a.s., the right-hand side can be bounded independently of $n$.
		This proves the statement for $p = 2$, hence by the Riesz--Thorin interpolation theorem~\cite[Theorem~1.3.4]{Grafakos2014ClassicalFourier}, the lemma holds for all $p \in [2,\infty]$, with $M_{S,p} \le M_{S,2}^{\frac{2}{p}} M_{S,\infty}^{1-\frac{2}{p}}$.
		
		\ref{lem:uniform-ultracontractivity:b}
		The differences with part~\ref{lem:uniform-ultracontractivity:a} are the use of Theorem~\ref{thm:convergence-An}\ref{thm:convergence-An:c} instead of~\ref{thm:convergence-An}\ref{thm:convergence-An:b}, and the fact that Assumption~\ref{ass:uniform-Linfty-bdd-eigenfuncs} is automatically satisfied.
	\end{proof}
	
	\section{Fractional parabolic integration}
	\label{app:frac-par-int}
	
	Let $(E, \norm{\,\cdot\,}{E})$ be a Banach space. Suppose that $-A \from \dom{A} \subseteq E \to E$ 
	generates a strongly continuous semigroup $(S(t))_{t\ge0}$. This implies the 
	existence of constants $M \in [1,\infty)$ and $w \in \R$ such that 
	\begin{equation}\label{eq:app-semigroup-est}
		\norm{S(t)}{\LO(E)}
		\le 
		M e^{wt} 
		\quad 
		\text{for all } t \in [0,\infty).
	\end{equation}
	Fixing $T \in (0,\infty)$, we define
	$k_s \from \R \to \LO(E)$ by $k_s(\tau) \deq 
	\frac{1}{\Gamma(s)} 
	\tau^{s-1}S(s)\indic{[0,T]}(\tau)$ for $\tau \in \R$.
	For any function $f \from (0,T) \to E$ such that the following Bochner integral converges in $E$ for $s \in (0,\infty)$ and a.e.\ $t \in (0,T)$, 
	we set
	\begin{equation}\label{eq:app-def-fracParab}
		\mathfrak I_{A}^s f(t)
		\deq 
		k_s * f(t)
		= 
		\frac{1}{\Gamma(s)} 
		\int_0^t (t-\tau)^{s - 1} S(t-\tau) f(\tau) \rd \tau.
	\end{equation}
	For $s = 0$ we set $\fI_{A}^0 \deq \id_E$. 
	The following properties of the \emph{fractional parabolic integration operator} $\fI_A^s$ (for a single operator $A$) are 
	well known, see~\cite[Proposition~5.9]{DaPrato2014}, and used 
	throughout the main text. We will state them here for the sake of 
	self-containedness.
	
	\begin{proposition}\label{prop:app-basicPropsIA}
		Suppose that $-A \from \dom{A} \subseteq E \to E$ generates a strongly 
		continuous semigroup $(S(t))_{t\ge0} \subseteq \LO(E)$ 
		satisfying~\eqref{eq:app-semigroup-est}.
		Then, for every $s \in [0,\infty)$, $p \in [1,\infty]$ and $T \in 
		(0,\infty)$, we have:
		\begin{enumerate}[(a), series=app1]
			\item \label{prop:app-basicPropsIA:Lp} $\fI_{A}^s$ is bounded 
			from $L^p(0,T; E)$ to itself, with an operator norm depending only on 
			$s$, $p$, $T$, $w$ and $M$.
		\end{enumerate}
		If $(F, \norm{\,\cdot\,}{F})$ is a Banach space for which there 
		exist $M' \in [1,\infty)$ and $\alpha \in [0,\infty)$ such that 
		\begin{equation}\label{eq:app-assumption-analyticEst}
			S(t) \in \LO(E; F) \quad \text{with} \quad
			\norm{S(t)}{\LO(E; F)} \le M' t^{-\alpha/2} 
			\quad 
			\text{for all } t \in [0,\infty),
		\end{equation}
		and in addition we have either $p = 1, \, s \ge 1 + \frac{\alpha}{2}$ or 
		$p > 1, \, s > \frac{1}{p} + \frac{\alpha}{2}$, then
		\begin{enumerate}[(a), resume=app1]
			\item\label{prop:app-basicPropsIA:Lp-CHdot} $\fI_{A}^s$ is bounded 
			from  
			$L^p(0,T; E)$ to $C([0,T]; F)$, with an operator norm depending only on 
			$s$, $p$, $T$ and $M'$.
		\end{enumerate}
	\end{proposition}
	
	For a sequence of operators $(A_n)_{n\in\clos\N}$ on Banach spaces which satisfy the appropriate discrete-to-continuum assumptions from the main text, the proposition above implies the following corollary regarding uniform boundedness of the sequence $(\widetilde{\mathfrak I}_{A_n}^s)_{n\in\clos\N}$, where $\widetilde{\mathfrak I}_{A_n}^s \deq \Lambda_n \mathfrak I_{A_n}^s \Pi_n$ for all $n \in \clos\N$.
	From this, in turn, one can derive Proposition~\ref{prop:strong-continuity-IAntilde} below asserting the strong convergence of these operators.
	
	\begin{corollary}\label{cor:unifbb-IAtilde}
		Let the Banach spaces $(E_n, \norm{\,\cdot\,}{E_n})_{n\in\clos\N}$, $(\widetilde E, \norm{\,\cdot\,}{\widetilde E})$ and the linear operators $(A_n)_{n\in\clos\N}$ satisfy the assumptions of Theorem~\ref{thm:d2c-TKapprox}, and suppose that ${p \in [1, \infty]}$ and $s \in [0,\infty)$. 
		The following assertions hold:
		\begin{enumerate}[(a), series=app-cor]
			\item The sequence $(\widetilde{\fI}^s_{A_n})_{n\in\clos \N}$ is uniformly bounded in $\LO(L^p(0,T; \widetilde E))$.
			\item \label{cor:unifbb-IAtilde:b}
			The sequence $(\widetilde{\fI}^s_{A_n})_{n\in\clos \N}$ is uniformly bounded in $\LO(L^p(0,T; \widetilde E); C([0,T]; \widetilde E))$ if, either, $p = 1$ and $s \ge 1$, or $p > 1$ and $s > \frac{1}{p}$.
		\end{enumerate}
		If the spaces $(E_n)_{n\in\clos\N}$, $(B_n)_{n\in\clos\N}$ and $\widetilde B$ are as in Assumptions~\ref{ass:disc2}, \ref{ass:semigroup-B2} 
		and~\ref{ass:emb-Htheta-B-H2}, and we have $s > \frac{1}{p} + 
		\frac{\theta}{2}$, then
		\begin{enumerate}[(a), resume=app-cor]
			\item \label{cor:unifbb-IAtilde:c}
			the sequence $(\widetilde{\fI}^s_{A_n})_{n\in\clos \N}$ is 
			uniformly bounded 
			in $\LO(L^p(0,T; E_\infty); C([0,T]; \widetilde B))$.
		\end{enumerate}
	\end{corollary}
	
	\begin{proposition}\label{prop:strong-continuity-IAntilde}
		Let the Banach spaces $(E_n, \norm{\,\cdot\,}{E_n})_{n\in\clos\N}$, $(\widetilde E, \norm{\,\cdot\,}{\widetilde E})$ and the linear operators $(A_n)_{n\in\clos\N}$ satisfy the assumptions of Theorem~\ref{thm:d2c-TKapprox}.
		Let ${p \in [1, \infty]}$ and $s \in [0,\infty)$. 
		The	following assertions hold:
		\begin{enumerate}[(a), series=app-strong-conv]
			\item\label{prop:strong-continuity-IAntilde:a} 
			If either $p = 1$ and $s \ge 1$, or $p > 1$ and $s > \frac{1}{p}$, then we have $\widetilde{\fI}^s_{A_n} f \to \fI^s_{A_\infty} f$ in $C([0,T]; \widetilde E)$, as $n \to \infty$, for every $f \in L^p(0,T; E_\infty)$.
		\end{enumerate}
		Moreover, let Assumptions~\ref{ass:disc2},~\ref{ass:semigroup-B2},~\ref{ass:A-conv-B} and~\ref{ass:emb-Htheta-B-H2} hold.
		\begin{enumerate}[(a), resume=app-strong-conv]
			\item\label{prop:strong-continuity-IAntilde:b} 
			If $s > \frac{1}{p} + \frac{\theta}{2}$, then we have $\widetilde{\fI}^s_{A_n} f \to \fI^s_{A_\infty} f$ in $C([0,T]; \widetilde B)$ as $n \to \infty$ for every $f \in L^p(0,T; E_\infty)$.
		\end{enumerate}
		\begin{proof}
			We only present the details of the argument for part~\ref{prop:strong-continuity-IAntilde:b}, the proof of~\ref{prop:strong-continuity-IAntilde:a} being similar.
			
			Let $p \in [1,\infty)$, $s \in (\frac{1}{p} + \frac{\theta}{2}, \infty)$, $f \in L^p(0,T; E_\infty)$ and fix an arbitrary $\varepsilon > 0$. 
			By the density of $B_\infty$ in $E_\infty$ (see~\ref{ass:emb-Htheta-B-H2}), and that of $B_\infty$-valued simple functions in $L^p(0,T; B_\infty)$, there exists a function $g \from [0,T] \to B_\infty$ of the form 
			\begin{equation}
				g = \sum_{j=1}^K \indic{(a_j, b_j)} \otimes x_j, 
				\quad
				K \in \N; \: 
				0 \le a_j < b_j \le T, \; x_j \in B_\infty \text{ for all } j \in 
				\{1,\dots,K\}
			\end{equation}
			such that
			\begin{equation}
				\norm{f - g}{L^p(0,T; E_\infty)} < \frac{\varepsilon}{4} 
				\Bigl(\sup\nolimits_{n\in\clos\N} \norm{\widetilde 
					\fI_{A_n}^s}{\LO(L^p(0,T; E_\infty); C([0,T]; \widetilde 
					B))}\Bigr)^{-1}.
			\end{equation}
			Note that the expression between the parentheses is finite by Corollary~\ref{cor:unifbb-IAtilde}\ref{cor:unifbb-IAtilde:c} and can be assumed to be nonzero without loss of generality, as otherwise $\widetilde \fI_{A_n}^s = 0$ for all $n \in \clos\N$ and the asserted convergence would be trivial. 
			Thus, for every $n \in \N$,
			\begin{align}
				&\norm{\widetilde\fI_{A_n}^s f - \fI_{A}^s f}{C([0,T]; \widetilde 
					B)}
				\\&\quad \le 
				\norm{\widetilde\fI_{A_n}^s (f - g)}{C([0,T]; \widetilde B)}
				+
				\norm{\widetilde\fI_{A_n}^s g - \fI_{A}^s g}{C([0,T]; \widetilde B)}
				+
				\norm{\fI_{A}^s (g - f)}{C([0,T]; \widetilde B)}
				\\
				&\quad< 
				\tfrac{1}{2}\varepsilon
				+
				\norm{\widetilde \fI_{A_n}^s g - \fI_{A}^s g}{C([0,T]; \widetilde 
					B)}.
			\end{align}
			For any $j \in \{1,\dots,K\}$, by~\ref{ass:A-conv-B} and the 
			discrete-to-continuum Trotter--Kato approximation theorem, we can 
			choose $N_j \in \N$ so large that 
			\begin{equation}
				\norm{\widetilde S_n \otimes x_j - S \otimes x_j }{C([0,T]; 
					\widetilde B)} 
				< 
				\frac{s\Gamma(s)}{2T^sK}\varepsilon 
				\quad 
				\text{for all }
				n \ge N_j.
			\end{equation}
			Thus, setting $N \deq \max_{j=1}^K N_j$, we find for all $n \ge N$ and 
			$t \in [0,T]$:
			\begin{align}
				\norm{\widetilde \fI_{A_n}^s g(t) - \fI_{A}^s g(t)}{\widetilde B}
				&\le 
				\frac{1}{\Gamma(s)}
				\sum_{j=1}^K \int_{a_j}^{b_j} (t-r)^{s-1} \norm{\widetilde S_n(t-r) 
					x_j - S(t-r) x_j}{\widetilde B} \rd r
				\\&\le 
				\frac{1}{\Gamma(s)}
				\sum_{j=1}^K \int_{0}^{T} r^{s-1} \norm{\widetilde S_n(r) x_j - 
					S(r) x_j}{\widetilde B} \rd r
				\\&\le 
				\frac{T^s}{s \Gamma(s)}
				\sum_{j=1}^K \norm{\widetilde S_n \otimes x_j - S\otimes 
					x_j}{C([0,T];\widetilde B)} < \frac{\varepsilon}{2}.
			\end{align}
			Since $t \in [0,T]$ was arbitrary, we conclude that 
			$\norm{\widetilde \fI_{A_n}^s g - \fI_{A}^s g}{C([0,T]; \widetilde B)}  
			< \frac{1}{2}\varepsilon$, and therefore
			$\norm{\widetilde\fI_{A_n}^s f - \fI_{A}^s f}{C([0,T]; \widetilde 
				B)} < \varepsilon$ for all $n \ge N$.
		\end{proof}
	\end{proposition}
	
	\section{Uniformly sectorial sequences of operators}
	\label{app:uniform-sectoriality}
	
	We first recall the concept of sectorial operators.
	Given $\omega \in (0,\pi)$, we say that a linear operator $A \from \dom{A} \subseteq E \to E$ on a (real or complex) Banach space $E$, with \emph{spectrum} $\sigma(A) \deq \C\setminus\rho(A)$, is said to be \emph{$\omega$-sectorial} if
	\begin{equation}\label{eq:app-sectorial} 
		\sigma(A) 
		\subseteq 
		\clos \Sigma_\omega
		\quad 
		\text{ and } 
		\quad 
		M(\omega, A) \deq 
		\sup \{ 
		\norm{\lambda R(\lambda, A)}{\mathscr L(E)} 
		:
		\lambda \in \C\setminus\clos\Sigma_\omega
		\}
		< \infty, 
	\end{equation}
	where $\Sigma_\omega$ is as in~\eqref{eq:def-sector} and $M(\omega, A)$ is called the \emph{$\omega$-sectoriality constant}.
	Its \emph{angle of sectoriality} $\omega(A) \in [0,\pi)$ is defined as 
	the infimum of all $\omega$ for which \eqref{eq:app-sectorial} holds. 
	
	If $A$ is closed and densely defined, then by~\cite[Theorem~13.30]{Neerven2022}, there exists $\omega \in (0,\frac{1}{2}\pi)$ such that $A$ is $\omega$-sectorial if and only if there exists $\eta \in (0, \frac{1}{2}\pi)$ such that $-A$ generates a bounded analytic semigroup $(S(t))_{t\ge0}$ on $\Sigma_\eta$. 
	The latter means that the mapping $[0,\infty) \ni t \mapsto S(t) \in \LO(E)$ extends to a bounded holomorphic function $\Sigma_\eta \ni z \mapsto S(z) \in \LO(E)$.
	Inspecting the proof of the cited theorem reveals that, whenever these equivalent conditions hold, we have
	\begin{equation}\label{eq:equiv-sect-holoBound}
		\sup\nolimits_{z \in \Sigma_\eta} \norm{S(z)}{\LO(E)} \eqsim_{(\omega, \eta)} M(\omega, A).
	\end{equation}
	This theorem also asserts that the supremum of the set of $\eta \in (0, \frac{1}{2}\pi)$ for which $(S(t))_{t\ge0}$ extends to a bounded analytic semigroup on $\Sigma_\eta$ equals $\frac{1}{2}\pi - \omega(A)$.
	
	Moreover, by~\cite[Propositions~3.4.1 and~3.4.3]{Haase2006} we have
	\begin{equation}\label{eq:app-analytic-semigroup-est}
		\norm{A^\alpha S(t)}{\LO(E)}
		\lesssim_{(\omega, \alpha)} M(\omega, A) \, t^{-\alpha},
	\end{equation}
	for all $\omega \in (\omega(A), \frac{1}{2}\pi)$ and $t \in (0,\infty)$, where the implicit constant is non-decreasing in $\alpha$ for any fixed $\omega$. 
	
	We say that a sequence $(A_n)_{n\in\N}$ of linear operators $A_n \from \dom{A_n} \subseteq E_n \to E_n$ is \emph{uniformly sectorial of angle $\omega \in [0,\pi)$} if $A_n$ is sectorial of angle $\omega$ for all $n \in \N$ and 
	\begin{equation}\label{eq:unif-sect}
		M_{\text{Unif}}(\omega', A) \deq \sup\nolimits_{n\in\N} M(\omega', A_n) < \infty
		\quad 
		\text{for all } 
		\omega' \in (\omega, \pi).
	\end{equation}
	Lemma~\ref{lem:strong-conv-AnalphaS} complements Theorem~\ref{thm:d2c-TKapprox} in the situation where the semigroup generators are uniformly sectorial of angle less than $\frac{1}{2}\pi$, in which case we obtain the uniform convergence $\Lambda_n A_n^{\alpha}  S(\,\cdot\,) \Pi_n x \to A_\infty^\alpha S_\infty(\,\cdot\,)x$ on compact subsets of $(0,\infty)$.
	It is an analog to~\cite[Lemma~4.1(2)]{KvN2011} in the discrete-to-continuum setting and for general $\alpha \in (0,\infty)$ (instead of $\alpha = 1$).
	
	\begin{lemma}\label{lem:strong-conv-AnalphaS}
		Let the linear operators $A_n \from \dom{A_n} \subseteq E_n \to E_n$ on the Banach spaces $(E_n, \norm{\,\cdot\,}{E_n})_{n\in\clos\N}$ be uniformly sectorial of angle $\omega \in [0,\frac{1}{2}\pi)$, and denote by $(S_n(t))_{t\ge0}$ the bounded analytic $C_0$-semigroups generated by $-A_n$.
		Let Assumptions~\ref{ass:disc3} and~\ref{ass:operators3} be 
		satisfied with $w = 0$, and suppose that the equivalent statements~\ref{thm:d2c-TKapprox:a} and~\ref{thm:d2c-TKapprox:b} in Theorem~\ref{thm:d2c-TKapprox} hold.
		Then we have, for all $\alpha \in (0,\infty)$, $x \in E_\infty$ and $0 < a < b < \infty$,
		\begin{equation}
			\sup_{t\in[a,b]}
			\norm{\Lambda_n A_n^\alpha S_n(t) \Pi_n x - A_\infty^\alpha S_\infty(t)x}{\widetilde E} \to 0
			\quad 
			\text{as } n \to \infty.
		\end{equation}
	\end{lemma}
	\begin{proof}
		Fix $\omega' \in (\omega, \frac{1}{2}\pi)$, $n \in \clos\N$ and $t \in (0,\infty)$.
		We begin by sketching the functional calculus argument (see~\cite[Chapter~15]{HvNVWVolumeIII} or~\cite{Haase2006} for a more comprehensive overview of this topic) which shows that we have the following Cauchy integral representation:
		\begin{equation}\label{eq:integral-representation-AnalphaSn}
			A_n^\alpha S_n(t) 
			=
			\frac{1}{2\pi i}
			\int_{\partial \Sigma_{\omega'}} z^\alpha e^{-tz} R(z, A_n) \rd z.
		\end{equation}
		To see this, define the functions $f_{\alpha}, g_t \from \Sigma_{\omega'} \to \C$ by $f_\alpha(z) \deq z^\alpha$ and $g_t(z) \deq e^{-tz}$ for $z \in \Sigma_{\omega'}$ and $t \in (0,\infty)$.
		Denote by $f_\alpha(A_n)$ and $g_t(A_n)$ the operators obtained via the extended Dunford calculus for sectorial operators as defined in~\cite[Definition~15.1.8]{HvNVWVolumeIII}.
		Then $f_\alpha(A_n)$ is the fractional power $A_n^\alpha$ in the sense of~\cite[Definition~15.2.2]{HvNVWVolumeIII}, which satisfies $f_\alpha(A_n)x = f_\alpha(\lambda)x = \lambda^\alpha x$ if $A_n x = \lambda x$, hence this (more general) definition agrees with the spectrally defined fractional powers in the setting of~\eqref{eq:def-WM-operator}.
		Moreover, $g_t(A_n) = S_n(t)$ by~\cite[Theorem~15.1.7]{HvNVWVolumeIII}.
		Since $S_n(t)$ is bounded, we have $(f_\alpha g_t)(A_n) = f_\alpha(A_n) g_t(A_n)$ by~\cite[Proposition~15.1.12]{HvNVWVolumeIII}.
		Finally, the function $(f_\alpha g_t)(z) =  z^\alpha e^{-z}$ is holomorphic and has (super)polynomial decay at $0$ and $\infty$, and thus belongs to the domain of the primary Dunford calculus~\cite[Definition~15.1.1]{HvNVWVolumeIII}, so that $(f_\alpha g_t)(A_n) = \frac{1}{2\pi i} \int_{\partial\Sigma_{\omega'}} z^\alpha e^{-tz} R(z, A_n) \rd z$.
		Putting all these observations together yields~\eqref{eq:integral-representation-AnalphaSn}.
		Applying the projection and lifting operators and parametrizing the complex integral yields, for all $x \in E_\infty$,
		\begin{align}
			\Lambda_n A_n^\alpha S_n(t) \Pi_n x 
			\label{eq:integral-representation-AnalphaSn2}
			&=
			\begin{aligned}[t]
				&-\frac{e^{i (\alpha + 1) \omega'}}{2\pi i} \int_0^\infty r^\alpha \exp(-t e^{i \omega'} r) \widetilde R(e^{i \omega'} r, A_n)x \rd r 
				\\
				&+ \frac{e^{-i (\alpha + 1) \omega'}}{2\pi i} \int_0^\infty r^\alpha \exp(-t e^{-i \omega'} r) \widetilde R(e^{-i \omega'} r, A_n)x \rd r,
			\end{aligned}
		\end{align}
		where we recall that $\Pi_\infty = \Lambda_\infty = \id_{\widetilde E}$ for $n = \infty$.
		It follows that the above estimate implies the following uniform bound on the interval $[a,b]$:
		\begin{equation}
			\begin{aligned}
				&\sup_{t\in[a,b]} \norm{\Lambda_n A_n^\alpha S_n(t) \Pi_n x 
					-
					A_\infty^\alpha S_\infty(t) x}{\widetilde E}
				\\
				&\quad \le \begin{aligned}[t]\label{eq:unif-bd-diff-AnalphaS}
					\frac{1}{2\pi} \int_0^\infty r^\alpha e^{-a \cos(\omega') r}
					\Bigl[&\norm{\widetilde R(r e^{i\omega'}, A_n)x - R(r e^{i\omega'}, A_\infty)x}{\widetilde E}
					\\&+
					\norm{\widetilde R(r e^{-i\omega'}, A_n)x - R(r e^{-i\omega'}, A_\infty)x}{\widetilde E}\Bigr] \rd r.
				\end{aligned}
			\end{aligned}
		\end{equation}
		Setting $\eta \deq \frac{1}{2}\pi - \vartheta$ for some $\vartheta \in (\omega, \omega')$, the uniform sectoriality of $(A_n)_{n\in\clos\N}$ implies that the operators $(-A_n e^{\pm i \eta})_{n \in \clos\N}$ generate $C_0$-semigroups $(S_n(t e^{\pm i \eta}))_{t\ge0}$ which are uniformly bounded in $t$ and $n$. 
		Therefore, we can apply Theorem~\ref{thm:d2c-TKapprox} to these two sequences of semigroups to find that
		$\widetilde R(\lambda, A_n)x \to R(\lambda, A_\infty)x$ for all $\abs{\arg \lambda} > \vartheta$ if this convergence holds for one such $\lambda$. 
		We have in fact $\widetilde R(\lambda, A_n)x \to R(\lambda, A_\infty)x$ for \emph{all} $\lambda$ such that $\abs{\arg \lambda} > \frac{1}{2}\pi > \vartheta$ by our hypothesis that the operators $(A_n)_{n\in\clos\N}$ satisfy statements~\ref{thm:d2c-TKapprox:a} and~\ref{thm:d2c-TKapprox:b} in Theorem~\ref{thm:d2c-TKapprox}, so we conclude $\widetilde R(r e^{\pm i \omega'}, A_n)x \to R(r e^{\pm i \omega'}, A_\infty)x$ for all $r \in (0,\infty)$.
		On the other hand, by~\eqref{eq:unif-sect} and Assumption~\ref{ass:disc3} we have $\norm{\widetilde R(r e^{\pm i\omega'}, A_n)x}{\widetilde E} \le M_\Pi M_\Lambda M_{\text{Unif}}(\vartheta, A) \norm{x}{\widetilde E} / r$ for all $n \in \clos\N$.
		Hence, we can bound the integrand in~\eqref{eq:unif-bd-diff-AnalphaS},
		up to $n$-independent constants, by the integrable function $r \mapsto r^{\alpha-1} \exp(-a \cos(\omega') r)$, so that the integrals tend to zero by the dominated convergence theorem.
	\end{proof}
	
	\bibliographystyle{siam}
	\bibliography{discrete-to-continuum-spde.bib}

\begin{thebibliography}{10}

\bibitem{AppellZabrejko1990}
{\sc J.~Appell and P.~P. Zabrejko}, {\em Nonlinear superposition operators},
  vol.~95 of Cambridge Tracts in Mathematics, Cambridge University Press,
  Cambridge, 1990.

\bibitem{ArensKelley1947CharacterizationSpaceContinuous}
{\sc R.~F. Arens and J.~L. Kelley}, {\em Characterization of the space of
  continuous functions over a compact {{Hausdorff}} space}, Trans. Amer. Math.
  Soc., 62 (1947), pp.~499--508.

\bibitem{ArmstrongVenkatraman2023}
{\sc S.~Armstrong and R.~Venkatraman}, {\em Optimal convergence rates for the
  spectrum of the graph {Laplacian} on {Poisson} point clouds}, Found. Comp.
  Math., https://doi.org/10.1007/s10208-025-09696-9 (2025).

\bibitem{BdPT2012}
{\sc V.~Barbu, G.~Da~Prato, and L.~Tubaro}, {\em The stochastic reflection
  problem in {H}ilbert spaces}, Comm. Partial Differential Equations, 37
  (2012), pp.~352--367.

\bibitem{BENSOUSSAN1973195}
{\sc A.~Bensoussan and R.~Temam}, {\em {\'E}quations stochastiques du type
  {N}avier-{S}tokes}, J. Funct. Anal., 13 (1973), pp.~195--222.

\bibitem{bertozzi}
{\sc A.~L. Bertozzi and A.~Flenner}, {\em Diffuse interface models on graphs
  for classification of high dimensional data}, SIAM Review, 58 (2016),
  pp.~293--328.

\bibitem{BertozziLuoStuart}
{\sc A.~L. Bertozzi, X.~Luo, A.~M. Stuart, and K.~C. Zygalakis}, {\em
  Uncertainty quantification in graph-based classification of high dimensional
  data}, SIAM/ASA Journal on Uncertainty Quantification, 6 (2018),
  pp.~568--595.

\bibitem{Brehier}
{\sc C.-E. Br\'ehier, M.~Hairer, and A.~M. Stuart}, {\em Weak error estimates
  for trajectories of {SPDE}s under spectral {G}alerkin discretization}, J.
  Comput. Math., 36 (2018), pp.~159--182.

\bibitem{Brzezniak1997}
{\sc Z.~Brze\'zniak}, {\em On stochastic convolution in {B}anach spaces and
  applications}, Stochastics Stochastics Rep., 61 (1997), pp.~245--295.

\bibitem{budd}
{\sc J.~Budd, Y.~van Gennip, and J.~Latz}, {\em Classification and image
  processing with a semi-discrete scheme for fidelity forced allen–cahn on
  graphs}, GAMM-Mitteilungen, 44 (2021), p.~e202100004.

\bibitem{CGT2022}
{\sc J.~Calder and N.~Garc\'ia~Trillos}, {\em Improved spectral convergence
  rates for graph {L}aplacians on {$\varepsilon$}-graphs and {$k$}-{NN}
  graphs}, Appl. Comput. Harmon. Anal., 60 (2022), pp.~123--175.

\bibitem{CalderGTLewicka2022}
{\sc J.~Calder, N.~Garc\'ia~Trillos, and M.~Lewicka}, {\em Lipschitz regularity
  of graph {L}aplacians on random data clouds}, SIAM J. Math. Anal., 54 (2022),
  pp.~1169--1222.

\bibitem{carmichael2009open}
{\sc H.~Carmichael}, {\em An open systems approach to quantum optics: lectures
  presented at the Universit{\'e} Libre de Bruxelles, October 28 to November 4,
  1991}, vol.~18, Springer Science \& Business Media, 2009.

\bibitem{Chak}
{\sc M.~Chak, N.~Kantas, and G.~A. Pavliotis}, {\em On the generalized
  {L}angevin equation for simulated annealing}, SIAM/ASA J. Uncertain.
  Quantif., 11 (2023), pp.~139--167.

\bibitem{Chiang}
{\sc T.-S. Chiang, C.-R. Hwang, and S.~J. Sheu}, {\em Diffusion for global
  optimization in {${\mathbb R}^n$}}, SIAM J. Control Optim., 25 (1987),
  pp.~737--753.

\bibitem{Cotter}
{\sc S.~L. Cotter, G.~O. Roberts, A.~M. Stuart, and D.~White}, {\em M{CMC}
  methods for functions: modifying old algorithms to make them faster},
  Statist. Sci., 28 (2013), pp.~424--446.

\bibitem{DaCosta2019}
{\sc C.~da~Costa, B.~F.~P. da~Costa, and M.~Jara}, {\em Reaction-diffusion
  models: from particle systems to {SDE}'s}, Stochastic Process. Appl., 129
  (2019), pp.~4411--4430.

\bibitem{DAPRATO1996241}
{\sc G.~Da~Prato and A.~Debussche}, {\em Stochastic {C}ahn-{H}illiard
  equation}, Nonlinear Anal., 26 (1996), pp.~241--263.

\bibitem{DaPratoKwapienZabczyk1987}
{\sc G.~Da~Prato, S.~Kwapie\'{n}, and J.~Zabczyk}, {\em Regularity of solutions
  of linear stochastic equations in {H}ilbert spaces}, Stochastics, 23 (1987),
  pp.~1--23.

\bibitem{DaPrato2014}
{\sc G.~Da~Prato and J.~Zabczyk}, {\em Stochastic equations in infinite
  dimensions}, vol.~152 of Encyclopedia of Mathematics and its Applications,
  Cambridge University Press, Cambridge, second~ed., 2014.

\bibitem{DGK2016a}
{\sc D.~Daners, J.~Gl\"uck, and J.~B. Kennedy}, {\em Eventually and
  asymptotically positive semigroups on {B}anach lattices}, J. Differential
  Equations, 261 (2016), pp.~2607--2649.

\bibitem{DGK2016b}
\leavevmode\vrule height 2pt depth -1.6pt width 23pt, {\em Eventually positive
  semigroups of linear operators}, J. Math. Anal. Appl., 433 (2016),
  pp.~1561--1593.

\bibitem{Davies1989}
{\sc E.~B. Davies}, {\em Heat kernels and spectral theory}, vol.~92 of
  Cambridge Tracts in Mathematics, Cambridge University Press, Cambridge, 1989.

\bibitem{Donnelly2006}
{\sc H.~Donnelly}, {\em Eigenfunctions of the {L}aplacian on compact
  {R}iemannian manifolds}, Asian J. Math., 10 (2006), pp.~115--125.

\bibitem{DWW2021}
{\sc D.~B. Dunson, H.-T. Wu, and N.~Wu}, {\em Spectral convergence of graph
  {L}aplacian and heat kernel reconstruction in {$L^\infty$} from random
  samples}, Appl. Comput. Harmon. Anal., 55 (2021), pp.~282--336.

\bibitem{VandenEijn}
{\sc W.~E and E.~Vanden~Eijnden}, {\em Statistical theory for the stochastic
  {B}urgers equation in the inviscid limit}, Comm. Pure Appl. Math., 53 (2000),
  pp.~852--901.

\bibitem{ElBouchairiFadiliHafieneElmoataz23}
{\sc I.~El~Bouchairi, J.~Fadili, Y.~Hafiene, and A.~Elmoataz}, {\em Nonlocal
  continuum limits of {$p$}-{L}aplacian problems on graphs}, Elements in
  Non-local Data Interactions: Foundations and Applications, Cambridge
  University Press, Cambridge, 2023.

\bibitem{EOR1991}
{\sc N.~D. Elkies, A.~M. Odlyzko, and J.~A. Rush}, {\em On the packing
  densities of superballs and other bodies}, Invent. Math., 105 (1991),
  pp.~613--639.

\bibitem{EngelNagel2000}
{\sc K.-J. Engel and R.~Nagel}, {\em One-parameter semigroups for linear
  evolution equations}, vol.~194 of Graduate Texts in Mathematics,
  Springer-Verlag, New York, 2000.

\bibitem{FadiliForcadelNguyenZantout23}
{\sc J.~Fadili, N.~Forcadel, T.~T. Nguyen, and R.~Zantout}, {\em Limits and
  consistency of nonlocal and graph approximations to the {E}ikonal equation},
  IMA J. Numer. Anal., 43 (2023), pp.~3685--3728.

\bibitem{Frisch_Lesieur_Brissaud_1974}
{\sc U.~Frisch, M.~Lesieur, and A.~Brissaud}, {\em A {M}arkovian random
  coupling model for turbulence}, J. Fluid Mech., 65 (1974), p.~145–152.

\bibitem{GTGHS2020}
{\sc N.~Garc\'ia~Trillos, M.~Gerlach, M.~Hein, and D.~Slep{\v c}ev}, {\em Error
  estimates for spectral convergence of the graph {L}aplacian on random
  geometric graphs toward the {L}aplace-{B}eltrami operator}, Found. Comput.
  Math., 20 (2020), pp.~827--887.

\bibitem{GTS2016}
{\sc N.~Garc\'{\i}a~Trillos and D.~Slep\v{c}ev}, {\em Continuum limit of total
  variation on point clouds}, Arch. Ration. Mech. Anal., 220 (2016),
  pp.~193--241.

\bibitem{GG2008}
{\sc F.~Gazzola and H.-C. Grunau}, {\em Eventual local positivity for a
  biharmonic heat equation in {$\Bbb R^n$}}, Discrete Contin. Dyn. Syst. Ser.
  S, 1 (2008), pp.~83--87.

\bibitem{GessToelle2016}
{\sc B.~Gess and J.~M. T\"{o}lle}, {\em Stability of solutions to stochastic
  partial differential equations}, J. Differential Equations, 260 (2016),
  pp.~4973--5025.

\bibitem{GigavanGennipOkamoto22}
{\sc Y.~Giga, Y.~van Gennip, and J.~Okamoto}, {\em Graph gradient flows: from
  discrete to continuum}.
\newblock Preprint, arXiv:2211.03384v1, 2022.

\bibitem{Grafakos2014ClassicalFourier}
{\sc L.~Grafakos}, {\em Classical {F}ourier analysis}, vol.~249 of Graduate
  Texts in Mathematics, Springer, New York, third~ed., 2014.

\bibitem{GregorioMugnolo2020}
{\sc F.~Gregorio and D.~Mugnolo}, {\em Bi-{L}aplacians on graphs and networks},
  J. Evol. Equ., 20 (2020), pp.~191--232.

\bibitem{GHL2014}
{\sc A.~Grigor'yan, J.~Hu, and K.-S. Lau}, {\em Heat kernels on metric measure
  spaces}, in Geometry and analysis of fractals, vol.~88 of Springer Proc.
  Math. Stat., Springer, Heidelberg, 2014, pp.~147--207.

\bibitem{Gyongy2014}
{\sc I.~Gy\"ongy}, {\em On stochastic finite difference schemes}, Stoch.
  Partial Differ. Equ. Anal. Comput., 2 (2014), pp.~539--583.

\bibitem{Haase2006}
{\sc M.~Haase}, {\em The functional calculus for sectorial operators}, vol.~169
  of Operator Theory: Advances and Applications, Birkh\"{a}user Verlag, Basel,
  2006.

\bibitem{HSV2007}
{\sc M.~Hairer, A.~M. Stuart, and J.~Voss}, {\em Analysis of {SPDE}s arising in
  path sampling. {II}. {T}he nonlinear case}, Ann. Appl. Probab., 17 (2007),
  pp.~1657--1706.

\bibitem{HSVW2005}
{\sc M.~Hairer, A.~M. Stuart, J.~Voss, and P.~Wiberg}, {\em Analysis of {SPDE}s
  arising in path sampling. {I}. {T}he {G}aussian case}, Commun. Math. Sci., 3
  (2005), pp.~587--603.

\bibitem{Havlin01011987}
{\sc S.~Havlin and D.~Ben-Avraham}, {\em Diffusion in disordered media}, Adv.
  Phys., 36 (1987), pp.~695--798.

\bibitem{HeinAudibertvonLuxburg07}
{\sc M.~Hein, J.-Y. Audibert, and U.~von Luxburg}, {\em Graph {L}aplacians and
  their convergence on random neighborhood graphs}, J. Mach. Learn. Res., 8
  (2007), pp.~1325--1368.

\bibitem{Hormander1968}
{\sc L.~H\"ormander}, {\em The spectral function of an elliptic operator}, Acta
  Math., 121 (1968), pp.~193--218.

\bibitem{HraivoronskaTse23}
{\sc A.~Hraivoronska and O.~Tse}, {\em Diffusive limit of random walks on
  tessellations via generalized gradient flows}, SIAM J. Math. Anal., 55
  (2023), pp.~2948--2995.

\bibitem{HvNVWVolumeI}
{\sc T.~Hyt\"{o}nen, J.~van Neerven, M.~Veraar, and L.~Weis}, {\em Analysis in
  {B}anach spaces. {V}ol. {I}. {M}artingales and {L}ittlewood-{P}aley theory},
  vol.~63 of Ergebnisse der Mathematik und ihrer Grenzgebiete. 3. Folge. A
  Series of Modern Surveys in Mathematics, Springer, Cham, 2016.

\bibitem{HvNVWVolumeII}
\leavevmode\vrule height 2pt depth -1.6pt width 23pt, {\em Analysis in {B}anach
  spaces. {V}ol. {II}}, vol.~67 of Ergebnisse der Mathematik und ihrer
  Grenzgebiete. 3. Folge. A Series of Modern Surveys in Mathematics, Springer,
  Cham, 2017.
\newblock Probabilistic methods and operator theory.

\bibitem{HvNVWVolumeIII}
\leavevmode\vrule height 2pt depth -1.6pt width 23pt, {\em Analysis in {B}anach
  spaces. {V}ol. {III}. {H}armonic {A}nalysis and {S}pectral {T}heory},
  Ergebnisse der Mathematik und ihrer Grenzgebiete. 3. Folge. A Series of
  Modern Surveys in Mathematics, Springer, Cham, 2023.

\bibitem{ItoKappel1998}
{\sc K.~Ito and F.~Kappel}, {\em The {T}rotter-{K}ato theorem and approximation
  of {PDE}s}, Math. Comp., 67 (1998), pp.~21--44.

\bibitem{kolesnikov_mosco_2006}
{\sc A.~V. Kolesnikov}, {\em Mosco convergence of {D}irichlet forms in infinite
  dimensions with changing reference measures}, J. Funct. Anal., 230 (2006),
  pp.~382--418.

\bibitem{konig2016parabolic}
{\sc W.~K{\"o}nig}, {\em The parabolic {A}nderson model: {R}andom walk in
  random potential}, Birkh\"auser/Springer, Cham, 2016.

\bibitem{Kruse2014}
{\sc R.~Kruse}, {\em Strong and Weak Approximation of Semilinear Stochastic
  Evolution Equations}, Springer International Publishing, Cham, 2014.

\bibitem{KvN2011}
{\sc M.~Kunze and J.~van Neerven}, {\em Approximating the coefficients in
  semilinear stochastic partial differential equations}, J. Evol. Equ., 11
  (2011), pp.~577--604.

\bibitem{KvN2012}
\leavevmode\vrule height 2pt depth -1.6pt width 23pt, {\em Continuous
  dependence on the coefficients and global existence for stochastic reaction
  diffusion equations}, J. Differential Equations, 253 (2012), pp.~1036--1068.

\bibitem{LangerMazya1999}
{\sc M.~Langer and V.~Maz'ya}, {\em On {$L^p$}-contractivity of semigroups
  generated by linear partial differential operators}, J. Funct. Anal., 164
  (1999), pp.~73--109.

\bibitem{LauxLelmi21}
{\sc T.~Laux and J.~Lelmi}, {\em Large data limit of the {MBO} scheme for data
  clustering: {$\Gamma$}-convergence of the thresholding energies}.
\newblock Preprint, arXiv:2112.06737v4, 2021.

\bibitem{LauxLelmi23}
\leavevmode\vrule height 2pt depth -1.6pt width 23pt, {\em Large data limit of
  the {MBO} scheme for data clustering}, PAMM, 22 (2023), p.~e202200308.

\bibitem{LiuRockner2015}
{\sc W.~Liu and M.~R\"{o}ckner}, {\em Stochastic partial differential
  equations: an introduction}, Universitext, Springer, Cham, 2015.

\bibitem{lord2014introduction}
{\sc G.~J. Lord, C.~E. Powell, and T.~Shardlow}, {\em An introduction to
  computational stochastic PDEs}, vol.~50, Cambridge University Press, 2014.

\bibitem{Mattingly}
{\sc J.~C. Mattingly and {\'E}.~Pardoux}, {\em Malliavin calculus for the
  stochastic 2{D} {N}avier-{S}tokes equation}, Comm. Pure Appl. Math., 59
  (2006), pp.~1742--1790.

\bibitem{MercervanGennip25}
{\sc S.~Mercer and Y.~van Gennip}, {\em A {B}rezis--{P}azy theorem for
  discrete-to-continuum limits with applications to gradient flows on graphs},
  (in prep.).

\bibitem{Mugnolo2006}
{\sc D.~Mugnolo}, {\em Gaussian estimates for a heat equation on a network},
  Netw. Heterog. Media, 2 (2007), pp.~55--79.

\bibitem{pmlr-v151-nikitin22a}
{\sc A.~V. Nikitin, S.~John, A.~Solin, and S.~Kaski}, {\em Non-separable
  spatio-temporal graph kernels via {SPDEs}}, in Proceedings of The 25th
  International Conference on Artificial Intelligence and Statistics,
  G.~Camps-Valls, F.~J.~R. Ruiz, and I.~Valera, eds., vol.~151 of Proceedings
  of Machine Learning Research, PMLR, 28--30 Mar 2022, pp.~10640--10660.

\bibitem{nualart_white_1992}
{\sc D.~Nualart and {\'E}.~Pardoux}, {\em White noise driven quasilinear
  {SPDE}s with reflection}, Probab. Theory Related Fields, 93 (1992),
  pp.~77--89.

\bibitem{OlverBoisvertClark2010}
{\sc F.~W.~J. Olver, D.~W. Lozier, R.~F. Boisvert, and C.~W. Clark}, eds., {\em
  N{IST} handbook of mathematical functions}, U.S. Department of Commerce,
  National Institute of Standards and Technology, Washington, DC; Cambridge
  University Press, Cambridge, 2010.

\bibitem{Ouhabaz2005}
{\sc E.~M. Ouhabaz}, {\em Analysis of heat equations on domains}, vol.~31 of
  London Mathematical Society Monographs Series, Princeton University Press,
  Princeton, NJ, 2005.

\bibitem{Roberts}
{\sc G.~O. Roberts and R.~L. Tweedie}, {\em Exponential convergence of
  {L}angevin distributions and their discrete approximations}, Bernoulli, 2
  (1996), pp.~341--363.

\bibitem{RZZ2012}
{\sc M.~R\"{o}ckner, R.-C. Zhu, and X.-C. Zhu}, {\em The stochastic reflection
  problem on an infinite dimensional convex set and {BV} functions in a
  {G}elfand triple}, Ann. Probab., 40 (2012), pp.~1759--1794.

\bibitem{SanzAlonsoYang2022}
{\sc D.~Sanz-Alonso and R.~Yang}, {\em The {SPDE} approach to {M}at\'{e}rn
  fields: graph representations}, Statist. Sci., 37 (2022), pp.~519--540.

\bibitem{SanzAlonsoYang2022JMLR}
\leavevmode\vrule height 2pt depth -1.6pt width 23pt, {\em Unlabeled data help
  in graph-based semi-supervised learning: a {B}ayesian nonparametrics
  perspective}, J. Mach. Learn. Res., 23 (2022).
\newblock Paper No. [97], 28.

\bibitem{SS2021}
{\sc L.~Scarpa and U.~Stefanelli}, {\em The energy-dissipation principle for
  stochastic parabolic equations}, Adv. Math. Sci. Appl., 30 (2021),
  pp.~429--452.

\bibitem{SoggeZelditch2002}
{\sc C.~D. Sogge and S.~Zelditch}, {\em Riemannian manifolds with maximal
  eigenfunction growth}, Duke Math. J., 114 (2002), pp.~387--437.

\bibitem{Taylor1981}
{\sc M.~E. Taylor}, {\em Pseudodifferential operators}, vol.~No. 34 of
  Princeton Mathematical Series, Princeton University Press, Princeton, NJ,
  1981.

\bibitem{ThorpeTheil19}
{\sc M.~Thorpe and F.~Theil}, {\em Asymptotic analysis of the
  {G}inzburg--{L}andau functional on point clouds}, Proc. Roy. Soc. Edinburgh
  Sect. A, 149 (2019), pp.~387--427.

\bibitem{TothZelditch2002}
{\sc J.~A. Toth and S.~Zelditch}, {\em Riemannian manifolds with uniformly
  bounded eigenfunctions}, Duke Math. J., 111 (2002), pp.~97--132.

\bibitem{vanGennipBertozzi12}
{\sc Y.~van Gennip and A.~L. Bertozzi}, {\em {$\Gamma$}-convergence of graph
  {G}inzburg--{L}andau functionals}, Adv. Differential Equations, 17 (2012),
  pp.~1115--1180.

\bibitem{vanGennipBudd25a}
{\sc Y.~van Gennip and J.~Budd}, {\em A Prolegomenon to Differential Equations
  and Variational Methods on Graphs}, Elements in Non-local Data Interactions:
  Foundations and Applications, Cambridge University Press, 2025.

\bibitem{vanGennipBudd25b}
\leavevmode\vrule height 2pt depth -1.6pt width 23pt, {\em Differential
  Equations and Variational Methods on Graphs --- With Applications in Machine
  Learning and Image Analysis}, to appear.

\bibitem{Neerven2022}
{\sc J.~van Neerven}, {\em Functional analysis}, vol.~201 of Cambridge Studies
  in Advanced Mathematics, Cambridge University Press, Cambridge, 2022.

\bibitem{vNVW2015}
{\sc J.~van Neerven, M.~Veraar, and L.~Weis}, {\em Stochastic integration in
  {B}anach spaces---a survey}, in Stochastic analysis: a series of lectures,
  vol.~68 of Progr. Probab., Birkh\"auser/Springer, Basel, 2015, pp.~297--332.

\bibitem{vNVW2008}
{\sc J.~M. A.~M. van Neerven, M.~C. Veraar, and L.~Weis}, {\em Stochastic
  evolution equations in {UMD} {B}anach spaces}, J. Funct. Anal., 255 (2008),
  pp.~940--993.

\bibitem{WeihsFadiliThorpe23}
{\sc A.~Weihs, J.~Fadili, and M.~Thorpe}, {\em Discrete-to-continuum rates of
  convergence for $p$-{L}aplacian regularization}.
\newblock Preprint, arXiv:2310.12691v1, 2023.

\bibitem{WeihsThorpe23}
{\sc A.~Weihs and M.~Thorpe}, {\em Consistency of fractional graph-{L}aplacian
  regularization in semisupervised learning with finite labels}, SIAM J. Math.
  Anal., 56 (2024), pp.~4253--4295.

\bibitem{WR2021}
{\sc C.~L. Wormell and S.~Reich}, {\em Spectral convergence of diffusion maps:
  improved error bounds and an alternative normalization}, SIAM J. Numer.
  Anal., 59 (2021), pp.~1687--1734.

\bibitem{Xiong}
{\sc J.~Xiong and X.~Yang}, {\em Uniqueness problem for {SPDE}s from population
  models}, Acta Math. Sci. Ser. B (Engl. Ed.), 39 (2019), pp.~845--856.

\bibitem{zanella_dirichlet_2017}
{\sc G.~Zanella, M.~B\'edard, and W.~S. Kendall}, {\em A {D}irichlet form
  approach to {MCMC} optimal scaling}, Stochastic Process. Appl., 127 (2017),
  pp.~4053--4082.

\bibitem{LapRef}
{\sc X.~Zhu, Z.~Ghahramani, and J.~Lafferty}, {\em Semi-supervised learning
  using {G}aussian fields and harmonic functions}, in Proceedings of the
  Twentieth International Conference on Machine Learning, ICML'03, AAAI Press,
  2003, p.~912–919.

\end{thebibliography}
	
\end{document}